\documentclass[11pt,a4]{amsart}
\usepackage{color,amssymb, amsmath,epsfig,graphicx,comment}
\usepackage[colorlinks=true,allcolors=blue]{hyperref}

\newcommand{\red}{\color{darkred}}
\newcommand{\blue}{\color{darkblue}}
\newcommand{\green}{\color{darkgreen}}

\newcommand{\clt}{central limit theorem}

\newcommand{\ex}{{\rm e}\,}

\newcommand{\asy}{asymptotic}

\newcommand{\ts}{time series}

\definecolor{darkblue}{rgb}{.1, 0.1,.8}
\definecolor{darkgreen}{rgb}{0,0.8,0.2}
\definecolor{darkred}{rgb}{.8, .1,.1}

\newtheorem{lemma}{Lemma}[section]

\newtheorem{theorem}[lemma]{Theorem}

\newtheorem{proposition}[lemma]{Proposition}
\newtheorem{definition}[lemma]{Definition}
\newtheorem{corollary}[lemma]{Corollary}
\newtheorem{example}[lemma]{Example}
\newtheorem{exercise}[lemma]{Exercise}
\newtheorem{remark}[lemma]{Remark}
\newtheorem{fig}[lemma]{Figure}
\newtheorem{tab}[lemma]{Table}

\newcommand{\bfQ}{{\bf Q}}
\newcommand{\bfu}{{\bf u}}

\newcommand{\bth}{\begin{theorem}}
\newcommand{\ethe}{\end{theorem}}

\newcommand{\bre}{\begin{remark}\em }
\newcommand{\ere}{\end{remark}}

\newcommand{\ble}{\begin{lemma}}
\newcommand{\ele}{\end{lemma}}

\newcommand{\bde}{\begin{definition}}
\newcommand{\ede}{\end{definition}}
\newcommand{\bco}{\begin{corollary}}
\newcommand{\eco}{\end{corollary}}

\newcommand{\bpr}{\begin{proposition}}
\newcommand{\epr}{\end{proposition}}

\newcommand{\bexer}{\begin{exercise}}
\newcommand{\eexer}{\end{exercise}}

\newcommand{\bexam}{\begin{example}}
\newcommand{\eexam}{\end{example}}

\newcommand{\bfi}{\begin{fig}}
\newcommand{\efi}{\end{fig}}

\newcommand{\btab}{\begin{tab}}
\newcommand{\etab}{\end{tab}}

\newcommand{\lhs}{left-hand side}
\newcommand{\fidi}{finite-dimensional distribution}
\newcommand{\rv}{random variable}

\newcommand{\sign}{{\rm sign}}

\newcommand{\var}{{\rm var}}

\newcommand{\corr}{{\rm corr}}
\newcommand{\as}{{\rm a.s.}}

\newcommand{\bfTh}{\mbox{\boldmath$\Theta$}}

\newcommand{\bfxi}{\mbox{\boldmath$\xi$}}

\newcommand{\rhs}{right-hand side}

\newcommand{\beao}{\begin{eqnarray*}}
\newcommand{\eeao}{\end{eqnarray*}\noindent}

\newcommand{\beam}{\begin{eqnarray}}
\newcommand{\eeam}{\end{eqnarray}\noindent}

\newcommand{\beqq}{\begin{equation}}
\newcommand{\eeqq}{\end{equation}\noindent}

\newcommand{\bce}{\begin{center}}
\newcommand{\ece}{\end{center}}
\newcommand{\bfzero}{{\bf 0}}
\newcommand{\barr}{\begin{array}}
\newcommand{\earr}{\end{array}}

\newcommand{\stp}{\stackrel{\P}{\longrightarrow}}
\newcommand{\std}{\stackrel{d}{\longrightarrow}}
\newcommand{\stas}{\stackrel{\rm a.s.}{\longrightarrow}}

\newcommand{\stw}{\stackrel{w}{\longrightarrow}}

\newcommand{\eqd}{\stackrel{d}{=}}

\newcommand{\vague}{\stackrel{\lower0.2ex\hbox{$\scriptscriptstyle
                    \it{v} $}}{\rightarrow}}
\newcommand{\weak}{\stackrel{\lower0.2ex\hbox{$\scriptscriptstyle
                    \it{w} $}}{\rightarrow}}
\newcommand{\what}{\stackrel{\lower0.2ex\hbox{$\scriptscriptstyle
                    \it{\hat{w}} $}}{\rightarrow}}

\newcommand{\bdis}{\begin{displaymath}}
\newcommand{\edis}{\end{displaymath}\noindent}

\renewcommand{\P}{\mathbb P}
\newcommand{\R}{\mathbb{R}}
\newcommand{\bfX}{{\mathbf{X}}}

\newcommand{\nto}{n\to\infty}
\newcommand{\kto}{k\to\infty}
\newcommand{\xto}{x\to\infty}

\newcommand{\ov}{\overline}
\newcommand{\wt}{\widetilde}

\newcommand{\vep}{\varepsilon}

\newcommand{\regvary}{regularly varying}
\newcommand{\slvary}{slowly varying}
\newcommand{\regvar}{regular variation}

\newcommand{\bbr}{{\mathbb R}}

\newcommand{\bbz}{{\mathbb Z}}

\newcommand{\con}{convergence}

\newcommand{\st}{such that}
\newcommand{\fif}{if and only if}
\newcommand{\wrt}{with respect to}
\newcommand{\chf}{characteristic function}
\newcommand{\fct}{function}

\newcommand{\ds}{distribution}

\newcommand{\rep}{representation}

\newcommand{\seq}{sequence}

\newcommand{\pro}{probabilit}

\newcommand{\ms}{measure}

\newcommand{\ld}{large deviation}
\newcommand{\bfx}{{\bf x}}
\newcommand{\bfB}{{\bf B}}
\newcommand{\bfY}{{\bf Y}}
\newcommand{\bfy}{{\bf y}}

\newcommand{\bfz}{{\bf z}}
\newcommand{\bfS}{{\bf S}}
\newcommand{\bfa}{{\bf a}}
\newcommand{\bfb}{{\bf b}}

\def\1{\ensuremath{\mathrm{1}\hspace{-.35em} \mathrm{1}}} % indicatrice

\def\E{{\mathbb E}}

\def\P{{\mathbb{P}}}
\def\R{\mathbb{R}}
\def\Z{\mathbb{Z}}

\def\a{{\alpha}}%                      %esperance
\def\corr{\mathop{\rm corr}\nolimits}%  %correlation
\newcommand{\conditionsmallsets}[1]{\mathbf{CS}_{#1}}
\newcommand{\conditionMX}[1]{\mathbf{MX}_{#1}}

\newcommand{\conditionsmallsetsp}[1]{\mathbf{CS}_{#1}}
\newcommand{\anticlust}{\mathbf{AC}}
\newcommand{\rvalpha}{\mathbf{RV}_\alpha}

\begin{document}
\today
\bibliographystyle{plain}
\title[Large deviation principles of $\ell^p$--blocks]{Large deviations of $\ell^p$--blocks of regularly varying time series and applications to cluster inference
}

\thanks{Thomas Mikosch's research is partially supported by Danmarks Frie Forskningsfond Grant No 9040-00086B. Olivier Wintenberger would like to thank Riccardo Passeggeri for useful discussions on the topic. Gloria Buritic\'a and Olivier Wintenberger would like to acknowledge the support of the French Agence Nationale de la Recherche (ANR) under reference ANR20-CE40-0025-01 (T-REX project).}

\author[G. Buritic\'a]{Gloria Buritic\'a}
\address{LPSM, Sorbonne Universit\'e\\
UPMC Universit\'e Paris 06\\ 
F-75005, Paris\\ France}
\email{gloria.buritica@upmc.fr}

\author[T. Mikosch]{Thomas Mikosch}
\address{Department  of Mathematics\\
University of Copenhagen\\
Universitetspar\-ken 5\\
DK-2100 Copenhagen\\
Denmark}
\email{mikosch@math.ku.dk}

\author[O. Wintenberger]{Olivier Wintenberger}
\address{LPSM, Sorbonne Universit\'es\\
UPMC Universit\'e Paris 06\\ 
F-75005, Paris\\ France}
\email{olivier.wintenberger@upmc.fr}

\begin{abstract}
In the regularly varying time series setting, a cluster of exceedances is a short period for which the supremum norm exceeds a high threshold. We propose to study a generalization of this notion considering short periods, or blocks, with $\ell^p-$norm above a high threshold. {Our main result derives new large deviation principles of extremal $\ell^p-$blocks, which guide us to define and characterize spectral cluster processes in $\ell^p$. We then study cluster inference in $\ell^p$ to motivate our results. We design consistent disjoint blocks estimators to infer features of cluster processes. Our estimators promote the use of large empirical quantiles from the $\ell^{p}-$norm of blocks as threshold levels which eases implementation and also facilitates comparison for different $p>0$. %Our approach motivates cluster inference based on extremal $\ell^p$--blocks with $p<\infty$ rather than the classical one with $p = \infty$ where the bias is more difficult to control as illustrated in simulations from an example.
Our approach highlights the advantages of cluster inference based on extremal $\ell^\alpha$--blocks, where $\alpha > 0$ is the index of regular variation of the series. 
 We focus on inferring important indices in extreme value theory, e.g., the {\it extremal index}.
} 
%We also derive an adaptive threshold selection method for cluster inference in $\ell^p$. }
\end{abstract}
\maketitle
\section{Introduction}\setcounter{equation}{0}
{
For various applications of extreme value statistics with stationary time series, it is natural to wonder how a recorded high level can affect the future behavior of the sequence or how time dependencies perturb inferential methodologies. For example, for high quantile marginal estimation it is well known that %However, the extremal dependencies of the time series can disturb 
the inference procedures tailored for independent observations are disturbed by temporal dependencies and must be corrected to produce accurate estimates; cf. Leadbetter \cite{leadbetter:1983}, Embrechts {\it et al.} \cite{embrechts:kluppelberg:mikosch:2013}. We aim to model extremal time dependencies in the setting of $\bbr^d$-valued 
 stationary regularly varying time series $(\bfX_t)_{t\in\bbz}$ with generic element $\bfX_0$; 
see Section~\ref{subsec:regvar} for a definition, cf.
Basrak and Segers \cite{basrak:segers:2009}. 
In this framework, an exceedance of a high threshold by the norm $|\bfX_t|$ at time $t$  might trigger consecutive exceedances in some small time interval around $t$.
These short periods with at least one exceedance were introduced implicitly in the seminal paper by Davis and Hsing \cite{davis:hsing:1995}. We refer to them as \textit{clusters (of exceedances)}. They were further reviewed in Basrak and Segers \cite{basrak:segers:2009} and Basrak {\it et al.} \cite{basrak:planinic:soulier:2018}. 
\par 
{ 
The main motivation for studying {clusters (of exceedances)} can be traced back to Theorem 2.5 in Davis and Hsing \cite{davis:hsing:1995}. 
For weakly dependent regularly varying time series, the point process with atoms at
$a_n^{-1}\bfX_{[0,n]} =a_n^{-1}(\bfX_0,\ldots,\bfX_n)$ admits a limit distribution that can be characterized
%the point process $N_n := \sum_{t=1}^n \epsilon_{a_n^{-1}X_t}$ admits a weak limit $N$ where $n\P(|X_0| > a_n) \to 1$; see Davis and Hsing \cite{davis:hsing:1995}. The limit 
{in terms of three features: the index of regular variation, the distribution of cluster ({of exceedances}),
%which corresponds to the case $p=\infty$ in \eqref{eq:ld:blocks} 
and the {\it extremal index} of $(|\bfX_t|)$, denoted by $\theta_{|\bfX|}$, where $a_n$ are moderate threshold levels satisfying $n\P(|\bfX_0| > a_n)\to 1$ as $n \to  \infty$. 
In this {setting, clusters are modeled as rare events of $\bfX_{[0,n]}$}  when its supremum norm exceeds the high level $x_n$ such that $\P(\|\bfX_{[0,n]}\|_\infty > x_n) \sim \theta_{|\bfX|}n \, \P(|\bfX_0| > x_n) \to 0$. %This means that the probability of recording a cluster tends to zero. 
%, i.e. the distribution of $x_n^{-1}\bfX_{[0,n]}$ where $(x_n)$ verifies $\P(\max_{t=1,\dots,n}|\bfX_t| > x_n) \to 0$. Also, the 
From  the last relation we also see that $\theta_{|\bfX|}$ arises when comparing 
the extremal behavior of blocks of maxima %{number of clusters (of exceedances)} 
%{\green there is no cluster in the independent case. Delete :of clusters ({of exceedances})} 
in $(\bfX_t)$ with the corresponding behavior of the blocks in  an iid \seq\ $(\bfX_t^\prime)$
with the same marginal \ds . In particular $\theta_{|\bfX|}$ describes how the blocks of maxima reach high levels compared to the iid setting.
\par 
 In view of the previous discussion a cluster (of exceedances) is tied together with the extremal index by the supremum norm. Our main theoretical result extends the aforementioned ideas from the $\ell^\infty$--norm
to $\ell^p$--norms for $p < \infty$. In Theorem~\ref{thm:main:theorem} we investigate the behavior of $\bfX_{[0,n]}$ when its $\ell^p$--norm {exceeds}  high levels $(x_n)$ 
{satisfying} $\P(\|\bfX_{[0,n]}\|_p > x_n) = \P(\sum_{t=1}^n|{\bfX_t}|^p > x_n^p) \to 0$ as $n \to  \infty$. We call this a large deviation result since it describes the probability that the partial sums $\|\bfX_{[0,n]}\|_p^p$ 
exceed the extreme threshold $x_n^p$. This leads us to a new definition of a \textit{cluster process} in the space $ \ell^p=\ell^p(\mathbb{R}^d)$
%A detailed study of these in Section~\ref{sec:spectral:cluster:process} reveals the connections between them. 
and, {in the limiting case $p=\infty$,} one recovers the classical clusters ({of exceedances}).
Similarly, large deviation principles for sums were considered by Nagaev \cite{nagaev:1979},
Cline and Hsing \cite{cline:hsing:1998} in the independent heavy-tailed case, and by Mikosch and Wintenberger \cite{mikosch:wintenberger:2013, mikosch:wintenberger:2014, mikosch:wintenberger:2016}, Mikosch and Rodionov \cite{mikosch:rodionov:2021} in the dependent heavy-tailed case. We extend large deviation principles to $\ell^p$--norms $\|\bfX_{[0,n]}\|_p$, and extremal $\ell^p$--blocks, i.e., blocks $\bfX_{[0,n]}$ with large $\ell^p$--norm.
%\begin{figure}
%    \centering
%    \includegraphics[width=.8\textwidth,height=.285\textwidth]{Img01Dec/cluster.pdf}
%    \caption{Time series of norm observations of an AR(1) with student($\alpha$) noise for $\alpha = 4$ with $\theta_{|\bfX|} = 0.5904$; see Section~\ref{ex:AR}. In blue blocks of length $8$ with $\ell^\infty$--norm exceeding the large threshold in the dotted line. In black the additional blocks also exceeding this threshold for the $\ell^\alpha$--norm.
%   }
%    \label{fig:cluster}
%\end{figure}
\par 
{We apply our findings to cluster inference. For this purpose we divide the}  sample $\bfX_1,\dots, \bfX_n$
into disjoint blocks  $({\bf B}_t)_{1 \le t \le \lfloor n/b_n \rfloor }$, ${\bf B}_t := \bfX_{(t-1)\,b_n+[1,b_n]}$, for a sequence of block lengths $(b_n)$ such that $b_n\to\infty$ and $b_n/n\to 0$. 
%Typically, inference on clusters ({of exceedances}) starts with selecting blocks whose supremum norm exceeds a high threshold $x_{b_n}$; cf. Kulik and Soulier \cite{kulik:soulier:2020}. 
%Instead, 
Then we select blocks whose $\ell^p$--norms
exceed a high threshold $x_{b_n}$. Our goal is to infer features of the cluster processes from these extremal $\ell^p$--blocks. %Since $(x_{b_n})$ might also depend on $p$, we replace thresholds by order statistics of $\ell^p$--norms.
{
%To address the aforementioned threshold/block length difficulties of cluster inference, 
In Theorem~\ref{cor:consistency}}
we design consistent  disjoint blocks methods with thresholds chosen as order statistics of $\ell^p$--norms. % that adapt to the length of the block.
%For our main application, we review inference procedures for functionals in $\ell^p$ acting on cluster processes through the so called blocks estimators
%We propose an adaptive threshold selection method for $\ell^p$--blocks and reveal the role which $p$ plays in this context. 
%Our approach heavily relies on the new large deviation results from Theorem~\ref{thm:main:theorem} and leads to unbiased estimates. 
%Readily, %our adaptive threshold method for inference on cluster functionals in $\ell^p$ can be detailed as follows. 
%for a fixed block length, 
%one typically the number of clusters.start by choosing a large threshold and a 
%extreme thresholds are more suitable for block selection than moderate ones. Indeed, the latter lead to asymptotically biased estimates.
%we aim to use the upper order statistics of the $p$--norms sample as threshold levels. 
Hereby
we must choose 
%our inference procedure resumes to choosing 
a number $k_n = k_n(p)$ of blocks with large $\ell^p$--norm such that $n/(b_nk_n) \to \infty$, $k_n \to \infty$. The sequence $(k_n)$ appeals to the classical bias-variance trade-off in extreme value statistics;  see for example Resnick \cite{resnick:2007}. { When choosing a small number of blocks $k_n$ for inference the variance of the estimates  increases while a large number $k_n$ leads to strong bias. This calls for a rigorous definition of 
extremal $\ell^p$--blocks  with the goal of revealing how $p$ plays a key role for tuning the sequence $k_n = k_n(p)$. %Indeed, we show that the choice of $p$ is crucial to correct the bias asthe smaller $p$ the larger the proportion of selected extreme blocks we might consider without introducing a bias.
{Moreover, we can derive the same quantity by using extremal $\ell^p$--blocks for different values of $p$ if we apply a change-of-norm technique.} Our large deviations result allows us then to compare the different estimators through the tuning parameter $k_n(p)$. 
The key argument of our analysis is
the relationship we stress between the sequence $k_n = k_n(p)$
and the large deviations of $\ell^p$--norms.}
%When comparing these estimators, our large deviation principles support our approach for $p<\infty$ reduces the bias of the classical perspective based on $p=\infty$ since we can consider a more significant number of order statistics without introducing any bias. %To resume, {our large deviation principles ensure that the new proposed estimators perform nicely as regards bias.} 
\par
%The sequence $(k_n)$ replaces thresholds $(x_{b_n})$ as $\|{\bf B}\|_{p,(k)}/x_{b_n} \xrightarrow[]{\P} 1$, for the $k$-th order statistic of the $\ell^p$-norms sample, and for a suitable sequence $k_n \to \infty$. Both $(k_n)$ and $(x_b)$ are linked through the $\ell^p$-norms large deviation results that we state in Theorem~\ref{thm:main:theorem}. 
One advantage of using empirical $\ell^p$--norm thresholds 
%The thresholds $(x_{b_n})$ should
is that they adapt to the block lengths $(b_n)$ and
take into account
the value of $p$. 
In the existing literature for $p=\infty$
no detailed advice is given {as to} how $(b_n)$ and $(x_{b_n})$ must be chosen; see for example Drees, Rootz\'en \cite{drees:rootzen:2010}, Drees, Neblung \cite{drees:neblung:2021}, Cissokho, Kulik \cite{cissokho:kulik:2021}, Drees 
{\it et al.} \cite{drees:janssen:neblung:2021}, who assume
growth conditions on the sequence $(x_n)$ 
such as $n \P(|\bfX_0| > x_n) \to 0$ as $n\to \infty$. It is common practice to replace $x_{b_n}$ by an upper 
order statistic of $(|\bfX_t|)_{1\le t\le n}$; see  
for example 
the blocks estimator of the {\em extremal index} proposed by Hsing \cite{hsing:1993}. In our setting, the order statistics of the $\ell^p$--norms adapt naturally to different values of $p$. %, and to adjust to the block lengths as $b_n$ grows. 
%; this is supported by a Monte Carlo study for inference on the extremal index in  Buritic\'a {\it et al.} \cite{buritica:mikosch:meyer:wintenberger}.
% and show consistency of the blocks estimators.
Asymptotic normality of our estimators could be derived by combining arguments from 
Theorem 4.3 in  Cissokho and Kulik \cite{cissokho:kulik:2021} and the large deviation arguments developed below; this topic is the subject of ongoing work and will not be presented here.}
\par 
\begin{comment}
The case when $p$ and the index $\a$ of regular variation of $(\bfX_t)$ coincide is rather specific. The relation $\P(\|\bfX_{[0,n]}\|_\alpha > x_n) \sim  n \, \P(|\bfX_0| > x_n) \to 0$ indicates that large deviations of the $\ell^\alpha$-norms are not affected by serial dependence. {\blue Therefore, choosing the number $k_n=k_n(p)$ of extremal $\ell^p$ blocks we use for inference is less sensitive to the temporal ties for $p = \alpha$ than for the classical choice $p = \infty$. } We apply our inference procedure to estimate the
extremal index using extremal $\ell^\alpha$--blocks. 
We also consider inference of {\em cluster indices} as defined by Mikosch and Wintenberger \cite{mikosch:wintenberger:2014} on partial sum functionals by considering extremal $\ell^\a$--blocks. Our
simulation study supports the fact that extremal $\ell^\alpha$ cluster inference is robust to handle extremal time dependencies regarding the choice of the number $k_n$ of extremal $\ell^\alpha$-blocks. {\blue Typically, we must fine-tune $k_n=k_n(p)$ to guarantee that the $k_n$--th order statistic of the $p$--norms sample is a consistent estimate of the high levels $(x_{b_n})$ from the large deviation principles of $p$-norms. Hereby the parameter $k_n=k_n(p)$ might be easier to tune for $p=\alpha$ than for the classical choice $p=\infty$ as extremal $\ell^\alpha$-blocks are not perturbed by serial dependence.}
\end{comment}
\par
The case when $p$ and the index $\a$ of regular variation of $(\bfX_t)$ coincide is rather specific. 
The relation $\P(\|\bfX_{[0,n]}\|_\alpha > x_n) \sim  n \, \P(|\bfX_0| > x_n) \to 0$ indicates that serial dependence does not affect large deviations of the $\ell^\alpha$--norm. { From this relation we see that the $\ell^\alpha$--norm of the series reaches high levels at the same rate as in the iid case. 
Consequently, when $p$ coincides with the index $\alpha$, the temporal dependencies of the sequence do not perturb the number $k_n=k_n(\alpha)$ of extremal $\ell^\alpha$--blocks we can consider for inference. In practice, this fact might ease tuning the parameter $k_n$.
  Hereby we focus on inferring classical indices of serial dependence based on extremal $\ell^\alpha$--blocks.  
We apply our inference procedure to estimate the
extremal index using extremal $\ell^\alpha$--blocks. 
We also consider inference of {\em cluster indices} as defined by Mikosch and Wintenberger \cite{mikosch:wintenberger:2014} on partial sum functionals by considering extremal $\ell^\a$--blocks. Our
simulation study supports the fact that $\ell^\alpha$--cluster inference is robust regarding the number $k_n$ of extremal $\ell^\alpha$--blocks we can choose. }
\par 
The previous indices are based on functionals that are shift-invariant with respect to the backward shift in \seq\ spaces; see Kulik and Soulier \cite{kulik:soulier:2020} for details. We  extend cluster inference to functionals acting on $\ell^p$  by studying {$\alpha$th-power} sum functionals acting on $\ell^p$. The key argument for this extension is the random shift 
analysis of Janssen \cite{janssen:2019} expressed in terms of the $\alpha$th moment of the cluster process. A similar idea has been investigated in Drees {\it et al.} \cite{drees:segers:warchol:2015} and Davis {\it et al.} \cite{davis:drees:segers:warchol:2018} for inference of the tail process. Here we focus on cluster inference. % 

%We heavily rely on Theorem~\ref{thm:main:theorem} to provide the main argument for comparing inference methods as we build therein on the definition of cluster processes from the new large deviation principles.}
\subsection{Outline of the paper.}\label{sec:outline}

In 
Section~\ref{sec:preliminaries},
after introducing preliminaries on regular variation,
 we present the main large deviation principle (Theorem~\ref{thm:main:theorem}).  In Section~\ref{sec:spectral:cluster:process} we study  $\ell^p$-valued
cluster processes which were introduced in Theorem~\ref{thm:main:theorem}. We apply this theorem  in
Section~\ref{sec:Applications} where we deal with  inference for shift-invariant functionals} acting on these cluster processes (see Theorem~\ref{cor:consistency}),  choosing thresholds as empirical quantiles of the $\ell^p$--norms of blocks. 
%We {introduce} new consistent estimators of the extremal index based on extremal $\ell^\a$--blocks {\red and of the cluster indices based on extremal $\ell^1$--blocks,} and highlight advantages of using extremal $\ell^p$--blocks with $p < \infty$.
%{\green ?? A similar argument based on extremal $\ell^1$--blocks can also be used for deriving  the of parameters of the limit \ds\ in the $\a$-stable central limit theorem 
%for regularly varying time series with $\a\in (0,2)$.} 
We continue with an
in-depth analysis of the assumptions of Theorem~\ref{thm:main:theorem}; see
Section~\ref{sec:large:deviation}.
In Section \ref{sec:beyond} we consider inference for 
non-shift-invariant functionals. We also illustrate our approach of $\ell^p$-based cluster inference for $p=\a$ 
and compare it with the case $p=\infty$; see Section~\ref{ex:AR}.
We defer all proofs to Section \ref{sec:proof}.

\begin{comment}
{\red
We define spectral cluster processes $\bfQ^{(p)}$ in $\ell^p$ for some $p\in(0,\infty]$ and show that $\bfQ^{(p)}$ for $p=1$ and $p=\alpha$ might be a good competitor to the classical \textit{cluster of exceedances}. We give large deviation principles for blocks: $x_n^{-1}X_{[0,n]}$, and derive adaptive threshold selection methods for improving cluster inference for $\ell^p$ functionals focusing on $p=1$ and $p=\a$.  }
\end{comment}

\subsection{Notation}\label{subsec:notation}
For integers $i$ and $a<b$ we write $i+[a,b]= \{i+a,\ldots,i+b\}$. 
It is convenient to embed the vectors
$\bfx_{[a,b]}\in\bbr^{d(b-a+1)}$ in $(\R^d)^{\bbz}$ by
assigning zeros to indices $i\not\in [a,b]$, and we then also
write $\bfx_{[a,b]}\in (\R^d)^\bbz$. We write $\bfx := (\bfx_t) = (\bfx_t)_{t \in \mathbb{Z}} $, and define truncation at level $\vep>0$
from above and below 
by $\underline\bfx_{\epsilon} = (\underline{\bfx_t}_\epsilon)_{t \in\bbz}$, 
$ \ov\bfx^\vep =(\ov{\bfx_t}^\epsilon )_{t\in\bbz}$,
where
$\underline{\bfx_t}_\epsilon=\bfx_t\,\1(|\bfx_t| > \epsilon)$, $\ov \bfx_t^\epsilon=\bfx_t \1(|\bfx_t| \le  \epsilon)$.

\par
We focus on the \seq\ space 
$\ell^p$, $p\in (0,\infty]$ equipped with the 
metric  
\beao
d_p(\bfx,\bfy) :=  \left\{\barr{ll}
\|\bfx - \bfy\|_p= \Big(\sum_{t\in\bbz}|\bfx_t -\bfy_t|^p\Big)^{1/p}\,,&p\in (1,\infty)\,,\\
\|\bfx-\bfy\|_p^p\,,&p\in (0,1)\,,\earr\right.\; \bfx,\bfy\in \ell^p\,,
\eeao
and the supremum distance in the case $p=\infty$. We know that $d_p$ makes $\ell^p$ a separable Banach space for $p \in( 1,\infty]$, and a separable complete metric space for $p\in (0,1)$.
Recall the {\em backshift operator} acting on $\bfx\in (\R^d)^\bbz$:
$B^k\bfx= (\bfx_{t-k})_{t\in\bbz}$, $k\in\bbz$. Then we define the shift-invariant space $\wt \ell^p = \ell^p/\sim$ as the quotient
space with respect to the equivalence relation $\sim$ in $\ell^p$: 
$\bfx \sim \bfy$ 
if there exists $k \in \mathbb{Z}$ such that $B^k\bfx = \bfy$. 
An element of $\wt \ell^p$ is denoted by $[\bfx] = 
\{ B^k\bfx : k \in \mathbb{Z}\}$. For ease of notation,  
we often write $\bfx$ instead of $[\bfx]$, and we notice that any 
element in $\ell^p$ can be embedded in $\wt \ell^p$ by using the
equivalence relation. We define for $[\bfx],[\bfy]\in\wt \ell^p$,
\beao
\wt d_p([\bfx],[\bfy])&:=& \inf_{k\in\bbz}\big\{d_p(B^k\bfa,\bfb): \bfa\in[\bfx]\,,\bfb\in [\bfy]\big\}\,.
\eeao
For $p\in(0,\infty]$, $\wt d_p$ is a metric on $\wt\ell^p$
and turns it into a complete metric space; see Basrak {\it et al.} 
\cite{basrak:planinic:soulier:2018}.
{

\section{Preliminaries and main result}\setcounter{equation}{0}\label{sec:preliminaries}
\subsection{About \regvar\ of time series}\label{subsec:regvar}
 We consider an $\bbr^d$-valued stationary process $(\bfX_t)$. Following Davis and Hsing \cite{davis:hsing:1995}, we call it 
\regvary\ if the \fidi s of the process are \regvary . This notion 
involves the vague \con\ of certain tail \ms s; see 
Resnick \cite{resnick:2007}. {Avoiding} the concept of
vague \con\ and infinite limit \ms  s, 
Basrak and Segers \cite{basrak:segers:2009} showed that
\regvar\ of $(\bfX_t)$ is equivalent to the weak \con\ relations:
for every $h\ge 0$,
\beam\label{eq:may27a}
\P\big(x^{-1}(\bfX_t)_{[-h,h]}\in \cdot \, \mid\, |\bfX_0|>x \big)\stw 
\P\big(Y\, (\bfTh_t)_{[-h,h]}\in \cdot\big)\,,\qquad \xto \,,\nonumber\\
\eeam
where $Y$ is Pareto$(\a)$-distributed, i.e., it has tail 
$\P(Y>y)=y^{-\a}$, $y>1$, independent of  the vector $(\bfTh_t)_{[-h,h]}$ in $(\mathbb{R}^d)^{2h +1}$
and $|\bfTh_0|=1$.
According to Kolmogorov's consistency theorem, one can extend 
the latter finite-dimensional vectors to a \seq\ $\bfTh=(\bfTh_t)_{t\in\bbz}$ 
in $(\bbr^d)^\bbz$ called 
the {\em spectral tail process} of $(\bfX_t)$.
\par 
Following Planini\'c and Soulier \cite{planinic:soulier:2018}, the spectral tail process $(\bfTh_{t})$ satisfies the {\em time-change formula:} for every measurable function  $f:({\ell}^p, {d}_p) \to \mathbb{R}$ such that $f(\lambda\, \bfx)=f(\bfx)$ for all $\lambda>0$, we have for all $t,s \in \mathbb{Z},$
\beam\label{eq:june5a}
\E[f(B^{s}(\bfTh_t))\1(\bfTh_{-s}\neq\bf0)]&=&\E\big[
|\bfTh_s|^\a\,f\big((\bfTh_t)\big)\big]\,. 
\eeam
The \regvar\ property of $(\bfX_t)$, denoted by  $\rvalpha$,
 is determined by the (tail)-index $\a>0$ and the spectral
tail process.
\par
Furthermore, Segers {\it et al.} \cite{segers:zhao:meinguet:2017} characterized \regvar\ of
random elements with values in star-shaped metric spaces. Their results
are based on weak \con\ in the spirit of \eqref{eq:may27a}. 
Our focus will be on a special star-shaped space: the \seq\ space $(\ell^p,d_p)$.
Using the {\em $p$-modulus \fct} $\|\cdot\|_p$, the $\ell^p-$valued stationary process $(\bfX_t)$ has the 
property $\rvalpha$ \fif\ relation \eqref{eq:may27a} holds with
$|\bfX_0|$ replaced by $\|\bfX_{[0,h]}\|_p$. 
Equivalently, (Proposition~3.1 in Segers {\it et al.} \cite{segers:zhao:meinguet:2017}), for every $h\ge 0$,
\beam\label{eq:december8a}
\P\big(x^{-1}\bfX_{[0,h]}\in\cdot \mid \|\bfX_{[0,h]}\|_p>x\big)
\stw \P\big(Y\,\bfQ^{(p)}(h)\in\cdot\big)\,,\qquad \xto\,,\nonumber\\
\eeam
the Pareto$(\a)$ variable $Y$ is independent of $\bfQ^{(p)}(h) \in \mathbb{R}^{d\,(h+1)}$, such that $\|\bfQ^{(p)}(h)\|_p=1$ a.s. We call $\bfQ^{(p)}(h)$
the {\em spectral component} of $\bfX_{[0,h]}$ in $\ell^p$. 

\subsection{Main result}\label{subsec:main:result}
{We start by giving our main result} on large deviations of the sequence $\bfX_{[0,n]}$, that we embed in the space $(\tilde{\ell}^p, \tilde{d}_p)$. The proof is postponed to Section~\ref{sec:LD}.
%which refers to the case $p=\infty$ in \eqref{eq:ld}.
\begin{theorem}\label{thm:main:theorem}
{Consider an $\bbr^d$-valued stationary \ts\ $(\bfX_t)$ {satisfying} $\rvalpha$ for some $\a>0$.} For a given $p > 0$, assume that there exists a sequence $(x_n)$ \st\ $n\,\P(|\bfX_0| > x_n) \to 0$ %and some $\kappa > 0$ such that $n/x_n^{p\land(\alpha - \kappa)}\to 0$ 
as $n \to \infty$. Furthermore, assume that for every $\delta > 0$,
\begin{itemize}
    \item[\,$\anticlust$\; :]$ 
    \lim_{\kto}\limsup_{\nto}
    \P\big(\|\bfX_{[k,n]}\|_\infty> \delta \,x_n \mid  \,|\bfX_0|> \delta\, x_n \big)=0,\,$\\[.3mm]
    \item[$\conditionsmallsets{p}$ :] 
    %$\mbox{\green Do we need $\|\|_p^p$? Why not $\|\|_p$ }
    \mbox{$\lim_{\epsilon \to 0}\limsup_{n \to \infty}
    \dfrac{\P(\|\ov {x_n^{-1}\bfX_{[0,n]}}^\epsilon\|_p>\delta)}
    %\sum_{t=1}^n\big(|\ov{ \bfX_t}^\epsilon| -\E[\ov {|\bfX|}^\epsilon]\big)\big|> \delta \,x_n^p\big) }
    { n\,\P(|\bfX_0| >  x_n )  } =0,
   $}\\[.3mm]
\end{itemize}
$n/x_n^p \to 0$ if $p < \alpha$, and there exists $\kappa > 0$ such that $n/x_n^{\alpha-\kappa} \to 0$ if $p=\a$.
Then, there exists $c(p) > 0$ such that
\beam \label{eq:ld}
\lim_{n \to + \infty} \dfrac{\P(\|\bfX_{[0,n]}\|_p > x_n)}{n\P(|\bfX_0| > x_n)} &=& c(p)\, . 
\eeam 
Moreover, $c(p)<\infty$ if $p\ge \alpha$, in particular, $c(\infty) %= \theta_{|\bfX|}
\le c(p) \le c(\alpha) = 1$. 
If $c(p) < \infty$ there exists $\bfQ^{(p)} \in \tilde \ell^p$ {\st } $\|\bfQ^{(p)}\|_p = 1$ a.s. and
\beam \label{eq:ld:blocks}
& &\P(x_n^{-1}\bfX_{[0,n]} \in \cdot \, | \, \|\bfX_{[0,n]}\|_p > x_n) \stw
\P(Y \bfQ^{(p)} \in \cdot ), \quad \nto\,,  \nonumber \\
& &
\eeam
in the space $(\tilde \ell^p,\tilde d_p)$ where $Y$ is Pareto$(\a)$ distributed, independent of $\bfQ^{(p)}$.
\end{theorem}
First, notice that under $\rvalpha$, $\anticlust$ and $\conditionsmallsetsp{\a}$ we obtain $c(\alpha)=1 $. This {motivates}  
the study of extremal $\ell^\alpha$--blocks since they
%suggests $\bfQ^{(\alpha)}$ as a potential} competitor of $\bfQ^{(\infty)}$.}
%the extremal $\ell^\a$--blocks %in a sample 
reach high levels at a constant rate regardless of the temporal dependencies traced through $c(p)$ in ~\eqref{eq:ld}. Second, notice that for $p > \alpha$ the result of Theorem~\ref{thm:main:theorem} holds under $\rvalpha$ and $\anticlust$ solely. Indeed, condition $\conditionsmallsetsp{p}$ holds for $p> \alpha$ by a Karamata--type argument; see Remark~\ref{rem:may27}. We state Theorem~\ref{thm:main:theorem} under the one-sided anti-clustering condition $\anticlust$. We use this condition together with a telescoping sum argument to compensate for the classical two-sided condition \eqref{eq:twosided} used in Kulik and Soulier \cite{kulik:soulier:2020}. { Conditions 
similar to $\conditionsmallsetsp{p}$ are standard when dealing with sum \fct als acting on $(\bfX_t)$; see e.g., Mikosch and Wintenberger \cite{mikosch:wintenberger:2013}. We refer to Section~\ref{sec:large:deviation} for a thorough discussion on the conditions $\anticlust$, $\conditionsmallsetsp{p}$, and the growth conditions imposed on $(x_n)$.}  }
%Since $\|\bfQ^{(p)}\|_p = 1$ a.s., {$\bfQ^{(p)}$  has interpretation as the spectral component of a \regvary\ $\ell^p$-valued random element. From \eqref{eq:ld} we infer that $c(p)$ is a non-decreasing function of $p$.} Letting $p=\infty$, one recovers the classical definition of clusters ({of exceedances}) in \eqref{eq:ld:blocks}. 
\par 
We refer to a relation of the type \eqref{eq:ld} as {\em \ld\  \pro ies} motivated by the following observation. Write $S_{k}^{(p)}=\sum_{t=0}^{k} |\bfX_t|^p$, for $k\ge 1$. Then  $|\bfX|^p$ is
\regvary\ with index $\alpha/p$.
Relation \eqref{eq:ld}
implies that
\beao
\P(\|\bfX_{[0,n]}\|_p > x_n)&=& \P\big(S_{n}^{(p)} >x_n^p\big)\\
&\sim& c(p)\,n\,\P(|\bfX_0|>x_n)\to 0\,,\quad\nto\,.
\eeao
Thus the left-hand \pro y describes the rare event that the sum
process $S_{n}^{(p)}$ exceeds the 
extreme threshold $x_n^p$.
\par
Relation~\eqref{eq:ld:blocks} extends the \ld\ result for $\|\bfX_{[0,n]}\|_p$ in \eqref{eq:ld} to {one} for the process $\bfX_{[0,n]}$ in the \seq\ space $(\tilde{\ell^p},\tilde{d}_p)$. Motivated by inference of the spectral cluster process $\bfQ^{(p)}$,
we establish \eqref{eq:ld:blocks}  employing 
weak convergence in the spirit of the polar decomposition from \eqref{eq:december8a}. %{given} in Segers {\it et al.}~\cite{segers:zhao:meinguet:2017}.

\bre
Recall Hult and Lindskog \cite{hult:lindskog:2006}
introduced \regvar\ for random elements assuming values in a 
general complete separable metric space
 %If $p> \a$, or $p < \alpha$ and $n/x_n^p \to 0$, or $p=\alpha$ and $n/x_n^{\alpha - \kappa} \to 0$ for some $\kappa > 0$
by extending the vague \con\ approach (see Resnick \cite{resnick:2007}) to $M_0$-\con ; see also Lindskog {\it et al.} \cite{lindskog:resnick:roy:2014}.
Relation~\eqref{eq:ld:blocks} provides a family of Borel sets in $(\tilde \ell^p, \tilde d_p)$ 
for which the weak limit of the self-normalized 
blocks $\bfX_{[0,n]}/\|\bfX_{[0,n]}\|_p$ exists. 
This result implies that the \seq\ of \ms s 
\beao
&&\mu_{n}(\cdot) :=
\P(x_n^{-1}\bfX_{[0,n]} \in \cdot)/\P(\|\bfX_{[0,n]}\|_p > x_n)\\\to&& 
\mu(\cdot) := \int_0^\infty \P(y\,\bfQ^{(p)} \in \cdot)\, d(-y^{-\alpha})\,,\quad \nto\,,
\eeao 
in the $M_0$--sense in $(\tilde\ell^p,\tilde{d}_p)$. By the 
portmanteau theorem for measures 
(Theorem 2.4. in Hult and Lindskog  \cite{hult:lindskog:2006}) 
\beao
\mu_{n}(A)= \P(x_n^{-1}\bfX_{[0,n]} \in A)/\P(\|\bfX_{[0,n]}\|_p > x_n)
\to \mu(A)\,,
\eeao
 for all Borel sets $A$ in $(\tilde{\ell}^p, \tilde{d}_p)$ 
satisfying $\mu(\partial{A}) = 0$ and $\bfzero \not \in \overline{A}$. 
This approach is discussed in Kulik and Soulier \cite{kulik:soulier:2020}.% where similar conditions as $\anticlust$ and $\conditionsmallsetsp{1}$ are stated for obtaining limit results for $\tilde\ell^1$-functionals. 
\ere 

\section{Spectral cluster process representation}\setcounter{equation}{0}\label{sec:spectral:cluster:process}
%{\red We define \textit{cluster processes} in terms of their spectral component $\bfQ^{(p)}$ verifying $\|\bfQ^{(p)}\|_p = 1$ a.s.}
\subsection{The spectral cluster process in $\ell^p$}\label{subsec:ellp}
{From \eqref{eq:may27a} recall the spectral tail
process $\bfTh$ of a stationary \seq\ $(\bfX_t)$
satisfying $\rvalpha$. We start by showing a representation of the 
spectral cluster process $\bfQ^{(p)}$ from \eqref{eq:ld:blocks} in terms of $\bfTh$.} We deduce that the spectral cluster process is also well defined in $(\ell^p,d_p)$. The proof is deferred to Section~\ref{sec:proof:sec:3}.

\begin{proposition}\label{lem:representation}
For $p,\a > 0$ assume $\|\bfTh\|_p+\|\bfTh\|_\a<\infty$ a.s. Under the assumptions of Theorem~\ref{thm:main:theorem}  the constant $c(p)$ defined in \eqref{eq:ld}   admits the representation
\beam\label{eq:dec10a}
c(p) &=& \mathbb{E}\big[\|\bfTh /\|\bfTh \|_\a\|_p^\a \big]\,.
\eeam
In addition, if $c(p) < \infty$ then {the \ds\ of the spectral cluster process $\bfQ^{(p)}$ is given by
\beam\label{eq:def:cluster:process1}\label{eq:def:cluster:process}
\P(\bfQ^{(p)} \in \; \cdot) &=& (c(p))^{-1} \E\big[\|\bfTh /\|\bfTh \|_\a \|_p^\a\;  \1(\bfTh /\|\bfTh\|_p  \in \; \cdot \; )\big]   \,,
\eeam
in the space $({\ell}^p,{d}_p)$.}
\end{proposition}
\begin{comment}
\beao\label{eq:defclusterprocess}
\P(\bfQ^{(p)} \in \; \cdot) &:=& \frac{ \E\big[ \|\bfTh \|_p^\alpha / \|\bfTh \|_\alpha^\alpha \;  \1( \bfTh /\|\bfTh \|_p  \in \; \cdot \; )\big]}{\E[\|\bfTh \|_p^\alpha/ \|\bfTh \|_\alpha^\alpha]}   \,. 
\eeao
\end{comment}

{This result provides a new representation of the \ds\ of 
$\bfQ^{(p)}$ for fixed $p$. In what follows, under the assumptions of Theorem~\ref{thm:main:theorem}, the spectral cluster processes are assumed to be defined in the space $(\ell^p,d_p)$ via \eqref{eq:def:cluster:process}. %Since $\bfTh$ in \eqref{eq:def:cluster:process} is defined in $(\ell^p, d_p)$ when $c(p)<\infty$ we extend $p-$clusters from $(\tilde \ell^p, d_p)$ to the space $(\ell^p, d_p)$. 
Proposition \ref{lem:representation} also relates distinct 
spectral cluster processes to each other {by the change-of-norms transform in \eqref{eq:def:cluster:process1}}. In the next section we deal with the case $p = \alpha $.}
%By definition of $\bfQ^{(p)}$ we have and further assume that $c(p) \in (0,\infty)$. This allows one to define a spectral cluster process in the \seq\ space $\ell^p$.
\subsection{The spectral cluster process in $\ell^{\a}$}\label{sec:spectral:a} In view of \eqref{eq:def:cluster:process1} the process 
$\bfTh/\|\bfTh\|_\alpha$ is the candidate for the 
%{\green there is some confusion about $p$-spectral cluster process and $\ell^p$-spectral cluster process}
$\ell^\a-$spectral cluster process $\bfQ^{(\alpha)}$ introduced in \eqref{eq:ld:blocks}, and it plays a key role for characterizing $\bfQ^{(p)}$ in general. The following result shows that $\bfTh/\|\bfTh\|_\alpha$ is well defined in $(\ell^\alpha, d_\a)$ under $\anticlust$. 
\bpr \label{prop:Janssenequivalences}
Let $(\bfX_t)$ be a stationary \seq\ satisfying $\rvalpha$ with 
spectral tail process $(\bfTh_t)$.  
Then the following statements are equivalent: 
\begin{enumerate}
    \item[\rm i)] $\|\bfTh \|_\alpha < \infty $ \as\ and $\bfTh/\|\bfTh\|_\a$ is well defined in $\ell^\a$.
    \item[\rm ii)] $|\bfTh_t| \to 0$ \as\ as $t \to \infty$.
    \item[\rm iii)] The time of the largest record $T^* := \inf\{s :s \in \mathbb{Z}\mbox{ such that } |\bfTh_s| = \sup_{t \in \mathbb{Z}}|\bfTh_t|\}$ is finite a.s.
\end{enumerate}

Moreover, these statements hold under $\anticlust$.
\epr
 A proof of Proposition~\ref{prop:Janssenequivalences} is given in Lemma~3.6 of Buritic\'a {\it et al.}~\cite{buritica:mikosch:meyer:wintenberger}, appealing to results by Janssen \cite{janssen:2019}.
 \par 
From \eqref{eq:december8a} recall 
the \seq\ of spectral components $ (\bfQ^{({\a})}(h))_{h\ge 0}$ 
of the vectors $(\bfX_{[0,h]})_{h\ge 0}$, satisfying {the property
$\|\bfQ^{(\a)}(h)\|_\a=1$ a.s.} 
Our next result relates this sequence of spectral components %$(\bfQ^{(\a)}(h))_{h \ge 0}$ 
to $ \bfTh/\|\bfTh\|_\alpha$.% the candidate process for $\bfQ^{(\alpha)}$. 
\begin{proposition}\label{prop:limitSI} Let $(\bfX_t)$ be a stationary time series satisfying $\rvalpha$  
and $\lim_{t\to \infty}|\bfTh_t|= 0$ \as\ Then
$\bfQ^{(\a)}(h)\std \bfQ^{(\a)}(\infty)$ as $h\to\infty$ in $(\wt\ell^\alpha,\wt d_\a)$ with $\bfQ^{(\a)}(\infty) \eqd \bfTh/\|\bfTh\|_\alpha$.
\end{proposition}
This result
gives raise to the interpretation of $\bfTh/\|\bfTh\|_\a$ as {\em the spectral component of $(\bfX_t)$ in $\ell^\alpha$.} The proof is given in Section~\ref{sec:proof:sec:3}.
We deduce the following almost sure relation in terms of the spectral cluster process in $\ell^\alpha$.
 \begin{proposition}\label{prop:a.s.:alpha}
 Let $(\bfX_t)$ be a stationary \seq\ satisfying $\rvalpha$ with 
spectral tail process $(\bfTh_t)$. Under the assumptions  $\anticlust$ and $\conditionsmallsetsp{\a}$, 
%of Theorem~\ref{thm:main:theorem} 
we deduce the a.s. \rep s %of the spectral cluster process
$\bfQ^{(\alpha)}=\bfTh/\|\bfTh\|_\a$ and $\bfTh=\bfQ^{(\a)}/|\bfQ_0^{(\a)}|$ in $(\wt \ell^\a, \wt d_\a)$. 
 \end{proposition}
 \par 
 Proposition~\ref{prop:a.s.:alpha} follows directly from Propositions \ref{lem:representation} and \ref{prop:Janssenequivalences}.

\section{Consistent cluster inference based on spectral cluster processes}\label{sec:Applications}\setcounter{equation}{0}
Let $\bfX_1,\ldots,\bfX_n$ be a sample from a 
stationary \seq\ $(\bfX_t)$ satisfying $\rvalpha$ for some $\a>0$, and choose 
$p>0$. 
We split the sample into disjoint blocks $\bfB_t:=\bfX_{(t-1)b + [1,b]}$, $t=1,\ldots,m_n$, where $b=b_n\to\infty$ and $m=m_n:=[n/b_n]\to\infty$.  
{Throughout we assume that the sequence $(x_n)$ satisfies the conditions of 
Theorem~\ref{thm:main:theorem} for $p > 0$. We denote $k=k_n:=[m_n\P(\|\bfB\|_p > x_{b_n})]\to\infty$. Then, in particular $\P(\|\bfB_1\|_p > x_b)\to 0$, $m_n \to \infty$ and $k_n \to \infty$. }
\par
{
%Under the conditions of Theorem~\ref{thm:main:theorem},

%{\green meaning? or any other spectral cluster process, provided its existence is granted. Delete: We study in the following sections the advantages from this consideration. 
%To avoid problems related to moderate threshold levels (see the discussion of Lemma~\ref{prop2:spt3:finite_constant}), we will rather focus on extreme threshold levels $(x_n)$ satisfying $n\,\P(|\bfX_0|>x_n)\to 0$ and the additional conditions for $p\le \a$ described in Lemma~\ref{prop:spt3:finite_constant}.
%The objective of this section is to promote cluster inference for $p < \infty$, as well as the 

%for inference on $\bfQ^{(p)}$. 
\subsection{Cluster functionals and mixing}
The real-valued \fct\ 
 $g$ on $\tilde{\ell}^p$ is a {\em cluster \fct al} 
if it vanishes  in some 
neighborhood of the origin and $\P(Y\bfQ^{(p)} \in D(g)) = 0 $ where $D(g)$ denotes the set of discontinuity points of~$g$. 
In what follows, it will be convenient
to write $\mathcal{G}_{+}(\tilde{\ell}^p)$ 
for the {class of} non-negative \fct s on $\wt{\ell}^p$ which vanish in some neighborhood of the origin. 
\par 
For \asy\ theory we will need the following {\em mixing condition}.\\[2mm]
\textbf{Condition $\conditionMX{p}$.}
There exists an integer sequence {$b_n \to \infty$ \st\  $m_n\to  \infty$, $k_n \to \infty$,
and for every Lipschitz-continuous
$f\in\mathcal{G}_{+}(\tilde{\ell^p})$, 
%{\green this is not only a condition on $(x_n)$. Delete:
the sequence $(x_n)$ satisfies}
\beam
\label{eq:mixing}
        \E\big[  \ex^{-    \frac{1}{k  } \sum_{t=1}^{m} f(x_{b}^{-1} \bfB_t)}\big] = 
\big(\E \big[ \ex^{- \frac{1}{k} \sum_{t=1}^{\lfloor m/k \rfloor} f( x_{b}^{-1} \bfB_t) } \big]\big)^{k} +o(1)\,,\quad\nto\,.\nonumber\\
\eeam   
with $m_n :=  \lfloor  n/b_n \rfloor$ and $k_n:=\lfloor  m_n\P(\|\bfB\|_p > x_{b_n})\rfloor$. \\
\par 
{If $\conditionMX{p}$ is required in the sequel we will refer to 
the sequences $(b_n)$, $(m_n)$ and $(k_n)$ chosen in this condition.
\bre
Condition $\conditionMX{p}$ is similar to the mixing conditions $\mathcal{A}$, $\mathcal{A}^\prime$ in 
Davis and Hsing  \cite{davis:hsing:1995}, Basrak {\it et al.} \cite{basrak:krizmanic:segers:2012}, respectively. These are defined  
in terms of \seq s $(f(\bfX_t))$ while our \fct als $f$ act on blocks.
$\conditionMX{p}$ holds under mild conditions, for example, under strong mixing
with quite general rate; {cf. Lemma 6.2. in Basrak {\it et al.} \cite{basrak:planinic:soulier:2018}.} 
%In contrast, we focus on functionals applied to the entire block and extend the% mixing-condition in \cite{basrak:planinic:soulier:2018} to this situation. Then, borrowing the triangular-array arguments 
%from the proof of Lemma 6.2. in \cite{basrak:planinic:soulier:2018}, we can est%ablish sufficient conditions for \eqref{eq:mixing} to hold for classical models.
\begin{comment}
This condition is of a similar form of the classical mixing conditions $\mathcal{A}$ and $\mathcal{A}^\prime$ from Davis and Hsing  \cite{davis:hsing:1995} and Basrak {\it et al.} \cite{basrak:krizmanic:segers:2012}, respectively. They are satisfied by classical time series; see Lemma 6.2 in Basrak {\it et al.} \cite{basrak:planinic:soulier:2018} for a proof using triangular-array arguments.}
\end{comment}
\ere

\subsection{Consistent cluster inference} {The following result is the basis
for an empirical procedure for spectral cluster inference built on disjoint blocks. The proof is given in Section~\ref{prooof:propcons}.}
%Following the classical argument of Resnick \cite{resnick:2007} in the iid setting, the level $x_b$ can be replaced by 
%$\|\mathbf{B}_{(k)}\|_p$ for a sequence $k_n \to  \infty$ as in $\conditionMX{p}$. This suggests to search for a suitable integer sequence $(k_n)$ 
%counting the number of sample blocks with large $\ell^p$--norm such that 
%$m_n/k_n \to \infty$ as $n \to \infty$.
\begin{theorem}\label{cor:consistency}
Assume the conditions of Theorem~\ref{thm:main:theorem} hold for $p>0$ with $c(p) < \infty$ together with
$\conditionMX{p}$.
Then $ \|\bfB\|_{p,( k+1 )}/x_{b} \xrightarrow[]{\P} 1$ and for every $g \in \mathcal{G}_{+}(\tilde{\ell^p})$,
    \begin{equation}\label{eq:consistent_estimator2}
         \frac{1}{k} \sum_{t=1}^{ m}  g\big( \|\bfB\|_{p,( k+1 )}^{-1}\bfB_t\big) \xrightarrow[]{\P}  \int_{0}^{\infty} \E\big[g(y \bfQ^{(p)})\big] d(-y^{-\alpha}), \quad n \to \infty\,.
      \end{equation}
     % for sequences $k_n\to  \infty$, $m_n = [n/b_n]  \to \infty$ 
      such that
      %of $\|\bfB_t\|_p$, $t=1,\ldots,m_n$, i.e.,
  $
  \|\bfB\|_{p,(1)} \ge  \|\bfB\|_{p,(2)} \ge \cdots \ge \|\bfB\|_{p,(m)}\,.
  $
     \end{theorem}
{By virtue of Proposition~\ref{lem:representation} we can derive 
 %$c(q)$ for $q > 0$ by 
 the same spectral cluster statistic by
letting the functionals $g_p:\tilde{\ell}^p \to \mathbb{R}$ act on $\bfQ^{(p)}$ for different pairs $(p,g_p)$.
This opens the road to different ways to estimate the same constant, for example, $c(q)$ for $q > 0$.  
%Our goal is to apply {Theorem~\ref{thm:main:theorem} for} inference on functionals acting
%on $\bfQ^{(p)}$, 
%{in particular for $p<\infty$.
%{The sequence $(k_n)$ in \eqref{eq:consistent_estimator2} depends on $p$.
%{\green meaning? and is fixed in $\conditionMX{p}$} 
%The choice of $p$ is flexible due to Proposition~\ref{lem:representation} relating spectral cluster processes by a change of norms functional. {The constant at the right-hand side of \eqref{eq:consistent_estimator2} can be written in different ways by combining the choice of $p$ with the choice of the functional $g_p$. The choice of $p,g_p$ are left to the practitioner.
 %This leads to a theory of cluster inference built on the spectral cluster processes $\bfQ^{(p)}$. 
%Two key points can be highlighted from Theorem~\ref{cor:consistency}. }
 %\par 
To compare inference procedures tuned with different $p$, we observe that Theorem~\ref{cor:consistency} promotes the use of order statistics 
of the sample of $\ell^p$--norms. %In this case, thresholds levels adapt automatically to block lengths and to the choice of $p$. This strategy also allows us to exploit the link between the block length and the threshold through the sequence $(k_n)$ satisfying
 %\beam\label{eq:k} 
%{k_n =[ m_n \P(\|\bfB\|_p > x_{b_n})]} \;\sim\; c(p)\, n\, \P(|\bfX_0| > x_{b_n}) \;=\; o(n/b_n^{( \alpha-\kappa/p) \lor 1 } ), \nonumber \\
% n \to + \infty.
% \eeam 
%The first equivalence follows from $\conditionMX{p}$
%The \asy\ equivalence follows from the restriction on the thresholds in Theorem~\ref{thm:main:theorem} where for $p < \alpha$, $n/x_n^p \to 0$, and for  $p=\alpha$, {$n/x_n^{\alpha - \kappa} \to 0$} for some $\kappa > 0$.
%\par
The sequence $(k_n)$ in \eqref{eq:consistent_estimator2} corresponds to the number of extreme blocks used for inference. The large deviation principles of Theorem~\ref{thm:main:theorem} allow us then to compare the sequences $k_n = k_n(p)$.
%allow us then to compare inference procedures tuned with different $p$.  
For inference through $\bfQ^{(p)}$ the relation
\beam\label{eq:k} 
k_n\;=\;[ m_n \P(\|\bfB\|_p > x_{b_n})] \;\sim\; c(p)\, n\, \P(|\bfX_0| > x_{b_n})\,,\,
\eeam
justifies taking $k_n$ larger as $p$ decreases, for $p \in (\alpha,\infty]$, since $c(\cdot)$ is a non-increasing function of $p$, and $(x_n)$ is a sequence satisfying $\anticlust$ and $n\P(|\bfX_0| > x_n) \to 0$}. For $p \in (0,\alpha]$, the sequence $(x_n)$ must satisfy the
additional condition $\conditionsmallsetsp{p}$, which restricts  the range of possible values for $k_n$, but allows us to consider continuous functionals on $(\tilde\ell^p,\tilde d^p)$. One advantage of choosing $p =\a$ is that $c(\alpha)=1$, thus the choice of $k_n = k_n(\alpha)$  % shrunken 
does not rely on the serial dependencies
%clustering effect 
summarized in $c(p)$.
%In particular, for $p=\a$ the choice of $k_n$ is not perturbed by the serial dependencies.} }
%$\bfQ^{(\infty)}$: a larger number $k_n$ of extremal $\ell^p$--blocks is permitted.} %More generally, disjoint blocks estimation from functional acting on $\bfQ^{(p)}$ with $p < \infty$ should improve the classical cluster inference based on $\bfQ^{(\infty)}$, notably in terms of bias.} 
\par 
%{In the next section we illustrate these two points.}
\subsection{Applications}
%Theorem~\ref{cor:consistency} motivates the adaptive threshold selection method based on $\ell^p$-norm order statistics for functionals acting on $\bfQ^{(p)}$ with $p < \infty$.
In this section we apply Theorem~\ref{cor:consistency} for inference on some indices related to the extremes in a time-dependent sample and focus on cluster inference using $\bfQ^{(\a)}$. We illustrate our %index 
estimators 
%of Corollaries~\ref{cor:extremal_index} and \ref{cor:stable:limit} 
for  a \regvary\ linear process in Section~\ref{ex:AR}.  
\subsubsection{The extremal index}
The extremal index of a regularly varying stationary time series has interpretation as a measure of clustering of serial exceedances, and was originally introduced in Leadbetter \cite{leadbetter:1983} and Leadbetter {\it et al.} \cite{leadbetter:lindgren:rootzen:1983}. If $(\bfX^\prime_t)$ is iid with the same marginal \ds\  as $(\bfX_t)$ then the extremal index $\theta_{|\bfX|}$ relates the expected number of serial exceedances of $(|\bfX_t|)$ with the serial exceedances of $(|\bfX^\prime_t|)$.  Assuming $\anticlust$ and additional mixing assumptions (see e.g. Theorem 2.3. in \cite{buritica:mikosch:meyer:wintenberger}), the extremal index $\theta_{|\bfX|}$ of $(|\bfX_t|)$ exists and equals $c(\infty)$. 
\par 
 We aim at applying Theorem~\ref{cor:consistency} with $p=\a$. In this setting, the change-of-norm formula in \eqref{eq:dec10a} leads to the identities
\beao
\theta_{|\bfX|} \;=\; c(\infty) \;=\; \E\Big[\dfrac{\| \bfTh \|^\alpha_\infty}{\| \bfTh \|^\alpha_\alpha}\Big]\;=\,\E[\| \bfQ^{(\alpha)} \|^\alpha_\infty].
\eeao
Then, letting $p=\a$ and  $g(\bfx)=\big({\|\bfx\|_\infty^\a}/{\|\bfx\|_\alpha^\a}\big) \1(\|\bfx\|_\alpha>1)$ on the right-hand side of \eqref{eq:consistent_estimator2}, {we obtain}
\beao
\int_{0}^{\infty} \E\big[g(y \bfQ^{(\a)})\big] d(-y^{-\alpha})
&=&\int_{0}^{\infty} \E\Big[\dfrac{\|\bfQ^{(\a)}\|_\infty^\a}{\|\bfQ^{(\a)}\|_\alpha^\a} \1(\|\bfQ^{(\a)}\|_\alpha^\a>y^{-\a})\Big] d(-y^{-\alpha})\\
&=& \E[\| \bfQ^{(\alpha)} \|^\alpha_\infty]\;= \; c(\infty)\,.
\eeao
{Next we introduce} a new consistent disjoint blocks estimator of the extremal index defined from exceedances of $\ell^\alpha$-norm blocks.
\begin{corollary}\label{cor:extremal_index}
Assume the conditions of Theorem~\ref{cor:consistency} for $p=\a$. 
Then 
%there exist sequences $k_n \to \infty$, $m_n=[n/b_n] \to \infty$ 
%{\blue verifying is the wrong word.} {\green  condition is superfluous $m_n/k_n \to + \infty$ because it is satisfied by virtue of Theorem~\ref{cor:consistency} }
%\st\
 \beam\label{eq:estimate:thet}
      \frac{1}{k} 
\sum_{t=1}^m  \frac{\|\bfB_t\|_\infty^\alpha}{\|\bfB_t\|^\alpha_\alpha} \,  
\1(\|\bfB_t\|_\alpha > \|\bfB\|_{\alpha,(k+1)} )\stp c(\infty) \,,\quad \nto\,.
\eeam
\end{corollary}
An advantage of inferring the extremal index using extremal $\ell^\a$--blocks is that the tuning parameter $k_n$ of the estimator does not rely on the clustering effect of the series since $c(\alpha)=1$ in Equation \eqref{eq:k}.

\begin{comment}
To motivate this estimator we compare it with a more classical cluster-based estimator of $c(\infty)$. Let for example $g(\bfx) := \1(\|\bfx\|_
\alpha> 1)$ (which evaluated on the cluster processes has the same expected value as evaluation on $\bfx \mapsto \sum_{j \in \mathbb{Z}}\1(|\bfx_t|>1)$ from the blocks estimator in Hsing \cite{hsing:1993}), then an application of Theorem~\ref{cor:consistency} with $p = \infty$ entails that there exists an integer sequence $k^\prime=k^\prime_n \to \infty $ such that
    \beam\label{eq:extremal:index:2}
     \frac{1}{k^\prime }
    \sum_{t=1}^m  \1(\|\bfB_t\|_\alpha \ge  \|\bfB\|_{\infty,( k^\prime  )})\stp \theta_{|\bfX|}^{-1} \,,\quad \nto\,,
    \eeam 
  with $k^\prime_n \sim  m_n\P(\|\bfB\|_\infty > x_b)  \sim  \theta_{|\bfX|}\, n\P(|\bfX_0| > x_b)   \sim  \theta_{|\bfX|}\, m_n \P(\|\bfB\|_\alpha > x_b)  \sim  \theta_{|\bfX|}\, k_n$. Thus the proportion of extreme blocks we shall consider for estimating the extremal index with \eqref{eq:extremal:index:2} (or the blocks estimator in Hsing \cite{hsing:1993}) is as small as the extremal index itself compared with the number of extreme blocks we can consider by using \eqref{eq:estimate:thet}.
\end{comment}
\bre 
{ We can compare this estimator of $c(\infty)$ with %an more classical  cluster-based estimator
one based on the clusters (of exceedances). Motivated by the blocks estimator 
of the extremal index
in Hsing \cite{hsing:1993}, we let $g(\bfx) := \sum_{j \in \mathbb{Z}}\1(|\bfx_t|>1)$ act on large $\ell^\infty$--blocks.} {Choosing} $p=\infty$ and 
using this $g$ on the right-hand side of \eqref{eq:consistent_estimator2}, 
we {can find} an integer sequence $k=k_n(\infty) \to \infty $ such that
    \beam\label{eq:extremal:index:2}
     \Big(\frac{1}{k}
    \sum_{t=1}^n \1(|\bfX_t| >  \|\bfB\|_{\infty,( k+1  )}) \Big)^{-1}\stp c(\infty)  \,,\quad \nto\,.
    \eeam 
    Arguing as for \eqref{eq:k},  $k_n \sim  m_n\P(\|\bfB\|_\infty > x_b)  \sim  c(\infty)\, n\P(|\bfX_0| > x_b)$.
    Thus, 
the number of extreme blocks {used in}
\eqref{eq:extremal:index:2} shrinks when $c(\infty) < 1$, compared to its implementation in an iid setting. In practice, this can make the choice of $k_n$ sensitive to the temporal ties. %This can make selecting the right $k_n$ to be a difficult task. 
    %\sim  c(\infty)\, m_n \P(\|\bfB\|_\alpha > x_b)  \sim  c(\infty)\, k_n$. 
    \begin{comment} {\blue 
  Arguing as for \eqref{eq:k}, we obtain $k_n \sim  m_n\P(\|\bfB\|_\infty > x_b)  \sim  c(\infty)\, n\P(|\bfX_0| > x_b)   \sim  c(\infty)\, m_n \P(\|\bfB\|_\alpha > x_b)  \sim  c(\infty)\, k_n$. Thus, 
the proportion of extreme blocks {used in}
%{\green ?we shall consider for estimating the extremal index with} 
\eqref{eq:extremal:index:2} is as small as the extremal index itself, 
compared with the number of extreme blocks {used in} \eqref{eq:estimate:thet}. }
\end{comment}
%{\green ?This proportion also depends on the extremal index which troubles implementation.}
\ere 

%highlights the delicate choice of threshold for cluster inference. It supports the adaptive threshold method and 

%we can consider a larger number of order statistics than for extremal $\ell^\infty$--blocks without introducing a bias term. 
\subsubsection{A cluster index for sums}\label{sec:sums}
In this section we assume that $\a\in (0,2)$ and $\E[\bfX]=\bf0$ for $\a\in (1,2)$.
We study the partial sums 
$\bfS_n := \sum_{t=1}^{n}\bfX_t$, $n\ge 1$, and introduce a normalizing \seq\
$(a_n)$ \st\ $n\,\P(|\bfX_0|>a_n)\to 1$.
Starting with Davis and Hsing \cite{davis:hsing:1995},
$\a$-stable central limit theory for $(\bfS_n/a_n)$ was proved under
suitable anti-clustering and mixing conditions.
\par 
In this setting, the quantity $c(1)$ appears naturally and was coined cluster index in Mikosch and Wintenberger \cite{mikosch:wintenberger:2014}. For $d=1$ it can be interpreted as an equivalent of the extremal index for partial sums rather 
than maxima. 
{ Indeed, consider a real-valued  \regvary\ stationary \seq\ 
$(X_t)$ with index of regular variation $\a\in(0,2)$ satisfying $\P( X \le -x) = o\big(\P(X > x) \big)$ or $X\eqd -X$. Consider an iid \seq\ $(X_t')$ with $X\eqd X'$ and partial sums $(S_n')$. Then 
$a_n^{-1} S_n\std \xi_\a$ and $a_n^{-1}S_n'\std \xi_\a'$, both $\xi_\a$ and $\xi_ \a'$ are $\a$-stable and }
\beao 
    \E[\ex^{i u\xi_\alpha }] &=&   \big( \E[\ex^{i u\xi_\alpha' }]\big)^{c(1)} \,.
\eeao
\par 
{Under the assumptions of Proposition \ref{lem:representation} and for $p=\a$ we have $c(1) = \E[\|\bfQ^{(\alpha)}\|_1^\alpha]$. For $\a \in (0,1]$, take $p=\a$ and $g(\bfx)=\big({\|\bfx\|_1^{\a}}/{\|\bfx\|_\a^\a} \big) \1(\|\bfx\|_\a>1)$ on the right-hand side of \eqref{eq:consistent_estimator2}. Then an application of Theorem~\ref{cor:consistency} with $p=\a$ and $g$ as mentioned yields a consistent estimator of $c(1)$.
\begin{corollary}\label{cor:stable:limit}
We assume the conditions of Theorem~\ref{cor:consistency} for $p=\a$ and $\a \in (0,1]$. Then we have for $k=k_n\to \infty$,
\beam\label{eq:c11}
   \frac{1}{k}\sum_{t=1}^m \frac{\|\bfB_t\|_1^\alpha}{\|\bfB_t\|^\alpha_\a}\, \1(\|\bfB_t\|_\a > \|\bfB\|_{\a,( k +1 )}) \stp c(1)\,,\quad \nto\,.
\eeam
\end{corollary}
The estimator on the \lhs\ of \eqref{eq:c11} 
has the advantage that $k_n \sim \,n\,\P(|\bfX_0| > x_{b_n})$. Relation \eqref{eq:c11} holds by virtue of \eqref{eq:k} regardless of the temporal dependence in the series.
\begin{comment}
\begin{corollary}\label{cor:stable:limit}
Assume the conditions of Theorem~\ref{cor:consistency} for $p=1$. 
Then, for $\a > 1$
%there exist 
%sequences $k_n \to \infty$, $m_n \to \infty$ {\green superfluous: satisfying $m_n/k_n \to \infty$} such that
\beam\label{eq:c11}
    \Big(\frac{1}{k}\sum_{t=1}^m \frac{\|\bfB_t\|_\alpha^\alpha}{\|\bfB_t\|^\alpha_1}\, \1(\|\bfB_t\|_1 > \|\bfB\|_{1,( k  )})\Big)^{-1}\stp c(1)\,,\quad \nto\,.
\eeam
and for $\a \le 1$
\beam\label{eq:c11}
   \frac{1}{k}\sum_{t=1}^m \frac{\|\bfB_t\|_1^\alpha}{\|\bfB_t\|^\alpha_\a}\, \1(\|\bfB_t\|_\a > \|\bfB\|_{\a,( k  )}) \stp c(1)\,,\quad \nto\,.
\eeam
\end{corollary}
\end{comment}
%{As we have argued above, such an estimator is} appealing since it is based on large $\ell^\a$-norms of blocks.
%instead of $\ell^\infty$-norms.
\bre  For $\alpha \in (1,2)$ the function $g$ 
applied in \eqref{eq:c11} to extremal $\ell^\a$-blocks is no longer bounded. If $c(1) < \infty$ we can apply Theorem~\ref{cor:consistency} with $p=1$ and $g(\bfx)=\big({\|\bfx\|_\a^\a}/{\|\bfx\|_1^\a} \big) \1(\|\bfx\|_1>1)$ 
to obtain a consistent estimator of $c(1)$. Indeed,   the right-hand side of \eqref{eq:consistent_estimator2} turns into
\beao
\int_{0}^{\infty} \E\big[g(y \bfQ^{(1)})\big] d(-y^{-\alpha})
&=&\int_{0}^{\infty} \E\Big[\dfrac{\|\bfQ^{(1)}\|_\a^\a}{\|\bfQ^{(1)}\|_1^\a} \1(\|\bfQ^{(1)}\|_1^\a>y^{-\a})\Big] d(-y^{-\alpha})\\
&=& \E[\| \bfQ^{(1)} \|^\alpha_\a]=(c(1))^{-1}\,,
\eeao
where the last identity follows from Proposition~\ref{lem:representation}. Then Theorem~\ref{cor:consistency} for $p=1$ and $g$ as above yields a consistent estimator of $c(1)$. 
Note that $c(1) \in [1,\infty)$ for $\alpha \in (1,2)$. Hence the number $k_n$ of extremal $\ell^1$-blocks for this estimator does not decrease in comparison with the iid case. This feature can also make this estimator robust for cluster inference.
\ere
\bre 
Arguing as in Cissokho and Kulik \cite{cissokho:kulik:2021},  
Kulik and Soulier \cite{kulik:soulier:2020}, and  assuming $\conditionsmallsetsp{1}$,
we can extend Theorem~\ref{cor:consistency} for $p=\infty$ to hold for $\ell^1$-functionals. 
Then we can find $k=k_n(\infty)\to\infty$ \st, with $g(\bfx) := \1(\|\bfx\|_1 > 1)$ and $p=\infty$
in \eqref{eq:consistent_estimator2},
    \beam\label{eq:c12}
    \frac{\sum_{t=1}^m  \1(\|\bfB_t\|_1 >  \|\bfB\|_{\infty,(k+1)})}{ \sum_{t=1}^n  \1(|\bfX_t| >  \|\bfB\|_{\infty,(k+1 )}) }\stp c(1) \,,\quad \nto\,.
    \eeam
   Here, following \eqref{eq:k}, we have $k_n \sim   c(\infty)\, n\P(|\bfX_0| > x_b).
$ This alternative estimator of $c(1)$ based on extremal $\ell^\infty$-blocks is consistent for $\alpha\in(0,2)$. Then, as in the extremal index example, the tuning parameter $k_n$ in \eqref{eq:c12} is linked to the constant $c(\infty) \in (0,1]$ and must be chosen carefully in agreement with the clustering effect of the series. 
    %$\sim c(\infty) (c(1))^{-1}\, m_n\P(\|\bfB\|_1 > x_b) \sim c(\infty)(c(1))^{-1}\, k_n$. In particular, for $\alpha \in (1,2)$, $c(1) \ge 1$, thus a small proportion (even smaller than $c(\infty)$) of blocks should be used to infer on $c(1)$. 
\ere
    %As in the extremal index example, %we highlight the importance of adaptive threshold selection for cluster inference, and 
   % we promote inference based on extremal $\ell^p$--blocks with $p < \infty$ for a {better} control of the bias compared to inference with extremal $\ell^\infty$--blocks.
   \par 
 {Theorem~\ref{cor:consistency} provides estimators of 
the parameters of the $\a$-stable limit $\bfxi_\alpha$ of $(\bfS_n/a_n)$. Indeed, following the theory in
 Bartkiewicz {\it et al.} \cite{bartkiewicz:jakubowski:mikosch:wintenberger:2011}, we characterize the $\a$-stable limit in terms of $\bfQ^{(1)}$; the proof is given in Section \ref{proof:stable}.}
\bpr\label{prop:stable:limit}
Consider a stationary \regvary\ \seq\ $(\bfX_t)$ with index $\alpha\in(0,1)\cup(1,2)$.
We assume  the mixing condition
\beao
\E[ \ex^{ i \bfu^\top  \bfS_{n}/a_n  } ] = (\E[ \ex^{ i \bfu^\top  \bfS_{{b_n}}/a_n  } ])^{m_n}  +o(1)\,,\qquad \nto\,,\quad \bfu\in\bbr^d\,,
\eeao and the
anti-clustering condition, 
for every $\delta > 0$,
\beam\label{eq:cond:small:sets:stable}
\lim_{l\to \infty }\limsup_{n \to \infty} n\,\mbox{$\sum_{t=l}^{b_n}$} \E[(|\bfX_t/a_n|\wedge\delta)\,(|\bfX_0/a_n| \wedge \delta)] = 0\,.
\eeam
Then $\bfS_n/a_n\std \bfxi_\a$ for an $\a$-stable random vector $\bfxi_\a$ 
with \chf\  
    $\E[\exp (i \bfu^\top \bfxi_\alpha )] 
    =  
     \exp(-c_\alpha\, \sigma_\alpha(\bfu)\,( 1- i\, \beta(\bfu)\tan( \alpha \pi/2)) )$, $\bfu\in\bbr^d$, 
where $c_\alpha := (\Gamma(2-\alpha)/|1-\alpha|)(1 \land \alpha) \cos(\alpha \pi/2)$,
and the \textit{scale} and \textit{skewness} parameters have \rep
\beao
    \sigma_\alpha(\bfu) &:=& c(1)\,\E[ |\bfu^\top \mbox{$\sum_{t \in \mathbb{Z}}$}\bfQ^{(1)}_t|^\alpha]\,, \,\\
    \beta(\bfu) &:=&  \big(\E[(\bfu^\top \mbox{$\sum_{t \in \mathbb{Z}}$}\bfQ^{(1)}_t)^\alpha_{+} - (\bfu^\top \mbox{$\sum_{t \in \mathbb{Z}}$}\bfQ^{(1)}_t)^\alpha_{-}]\big)/\E[ |\bfu^\top \mbox{$\sum_{t \in \mathbb{Z}}$}\bfQ^{(1)}_t|^\alpha]\,.
\eeao
\epr
As for $c(1)$, an application of Theorem~\ref{cor:consistency} with $p=1$ for $\alpha\in (1,2)$ and $p=\a$ for $\a\in(0,1)$ yields natural estimators of the parameters $(\sigma_\alpha(\bfu), \beta(\bfu))$ in the central limit theorem of Proposition~\ref{prop:stable:limit}.

\section{{A discussion} of the assumptions of the large deviation principle in Theorem~\ref{thm:main:theorem}}\setcounter{equation}{0}\label{sec:large:deviation}\label{subsec:assumptions}
%We derive large deviation principles in terms of $\bfQ^{(\alpha)}$ and show that, for the purposes of statistical inference, $\bfQ^{(\alpha)}$ might be a competitor to $\bfQ^{(\infty)}$. 

%\subsection{Assumptions}%\label{subsec:assumptions}
Consider a stationary \seq\ $(\bfX_t)$
satisfying $\rvalpha$ and let $(x_n)$ be a  {threshold} \seq s \st\ 
$n\,\P(|\bfX_0|>x_n)\to 0$. 
%{\green superfluous: and $x_n \to\infty$ as $\nto$.}
In the conditions $\anticlust$ and $\conditionsmallsets{p}$ below 
we refer to the same \seq\ $(x_n)$.  {In this section we will discuss the conditions of} Theorem~\ref{thm:main:theorem}.
\par 

\subsection{\bf Anti-clustering condition $\anticlust$.}  For every $\delta>0$, 
\beao 
\lim_{\kto}\limsup_{\nto}
\P\big(\|\bfX_{[k,n]}\|_\infty> \delta \,x_n \mid  \,|\bfX_0|> \delta\, x_n \big)=0.\,
\eeao
Condition $\anticlust$ ensures that a large value at present time
does not persist indefinitely in the extreme future of the \ts . 
This anti-clustering is weaker than the more common two-sided one:
\beam\label{eq:twosided} 
\lim_{k\to\infty}\limsup_{\nto}
\P\big(\max_{k\le |t|\le n}|\bfX_t|> \delta \,x_n \mid  \,|\bfX_0|> \delta\, x_n \big)=0.\,
\eeam
A simple sufficient condition, which breaks block-wise extremal dependence into pair-wise, is given by 
\beao
\lim_{\kto}\limsup_{\nto} \sum_{t=k}^{n} 
\P\big(|\bfX_{t}|> \delta \,x_n \mid  \,|\bfX_0|> \delta\, x_n \big)\,.
\eeao
For $m$-dependent $(\bfX_t)$ the latter condition turns into 
$n\,\P(|\bfX_0|>\delta\,x_n)\to 0$ which is always satisfied. 
\par
If $p\le \alpha$ an extra assumption is required for controlling the accumulation of moderate extremes within a block.
\subsection{\bf Vanishing-small-values condition $\conditionsmallsetsp{p}$.} For $p\in( 0,\a]$ %and for every {\red bounded} Lipschitz-continuous  function  $f:(\tilde{\ell}^p, \tilde{d}_p) \to \mathbb{R}$ 
we assume that for a \seq\ $(x_n)$ satisfying $n\P(|\bfX_0|>x_n)\to 0$ and
for every $\delta>0$, we have
\beam\label{eq:june2aa}
\lim_{\epsilon \to 0}\limsup_{n \to \infty}
\frac{\P\big( \big| \big\|\ov {x_n^{-1}\bfX_{[1,n]}}^\epsilon\big\|_p^p -\E\big[\big\|\ov {x_n^{-1}\bfX_{[1,n]}}^\epsilon\big\|_p^p\big]\big|>\delta\big)}
%\sum_{t=1}^n\big(|\ov{ \bfX_t}^\epsilon| -\E[\ov {|\bfX|}^\epsilon]\big)\big|> \delta \,x_n^p\big) }
{ n\,\P(|\bfX_0| >  x_n )  } = 0.
\eeam
%\[
%\lim_{\epsilon \to 0}\limsup_{n \to \infty}
%\frac{\P\big( \big| f( \ov{x_n^{-1}\bfX_{[1,n]}}^\epsilon) - \E[f(\ov {x_n^{-1}\bfX_{[1,n]}}^\epsilon)]\big|> \delta \big) }{ n\,\P(|\bfX_0| >  x_n )  } = 0\,.
%\]
\par
{
We refer to \eqref{eq:june2aa} as condition $\conditionsmallsetsp{p}$ in what follows. If $\alpha<p<\infty$ then by Karamata's theorem (see Bingham {\it et al.} \cite{bingham:goldie:teugels:1987}) and since $n\,\P(|\bfX_0|>x_n)\to 0$,
\beao
\E[\|\ov {x_n^{-1}\bfX_{[1,n]}}^\epsilon|_p^p]=n\,\E[|\ov{x_n^{-1}\bfX}^\epsilon|^p]=o(1)\,,\quad \nto\,.
\eeao
Also, if $p<\a$, then $\E[|\bfX|^p]<\infty$. If  we also have $n/x_n^p\to 0$ then
\beao 
\E[\|\ov {x_n^{-1}\bfX_{[1,n]}}^\epsilon\|_p^p]\le n\,x_n^{-p}\E[|\bfX|^p]\to 0\,,\quad\nto\,.
\eeao
If $p=\a$, $\E[|\bfX|^\a]<\infty$ and $n/x_n^\a\to 0$ then the latter
relation remains valid. If $\E[|\bfX|^\a]=\infty$ then {$\E [|\ov{x_n^{-1}\bfX}^\epsilon|^\a]=x_n^{-\a}\,\ell(x_n)$ for some \slvary\ \fct\ $\ell$ depending on $\epsilon$}, hence for every small $\kappa>0$ and large $n$, $\ell(x_n)\le x_n^{\kappa}$. Then the condition 
$n x_n^{-\a+\kappa}\to 0$ also implies that $\E[\|\ov {x_n^{-1}\bfX_{[1,n]}}^\epsilon\|_\a^\a]=o(1)$. {Thus we retrieve $\conditionsmallsets{p}$ as {used} in Theorem~\ref{thm:main:theorem}.
In sum, 
%for $p>\a$, \eqref{eq:june2aa} always holds without the centering term and, for $p\le \a$, 
under the aforementioned additional
growth conditions on $(x_n)$ centering in \eqref{eq:june2aa} can be avoided. This is similar  to condition $\conditionsmallsets{p}$ in Theorem~\ref{thm:main:theorem}.}
\par 
We mentioned that conditions of
a similar type as $\conditionsmallsetsp{p}$ are standard when dealing with sum \fct als acting on $(\bfX_t)$ 
(see for example  Davis and Hsing \cite{davis:hsing:1995}, Bartkiewicz {\it et al.} \cite{bartkiewicz:jakubowski:mikosch:wintenberger:2011},  Mikosch and Wintenberger 
\cite{mikosch:wintenberger:2013,mikosch:wintenberger:2014,mikosch:wintenberger:2016}), and are also discussed in Kulik and Soulier \cite{kulik:soulier:2020}. 
\bre\label{rem:may27}
Assume $\alpha<p<\infty$. Then applications of  Markov's inequality of order $1$ and Karamata's theorem  yield for $\delta>0$, as $\nto$,
\beao
\dfrac{\P\big(  \|\ov {x_n^{-1}\bfX_{[1,n]}}^\epsilon\|^p_p > \delta\big)}{ n\,\P(|\bfX_0| > x_n )} &=&
\dfrac{\P\big(\sum_{t=1}^n |\ov {x_n^{-1}\bfX_t}^\epsilon|^p > \delta \big)}{ n\,\P(|\bfX_0| > x_n )   }\\&\le & 
\dfrac{\E [|\ov{x_n^{-1}\bfX_0}^\epsilon|^p]}{\delta\,\P(|\bfX_0|> \epsilon \, x_n)}
\dfrac{\P(|\bfX_0|>\epsilon\,x_n)}{\P(|\bfX_0|>x_n)}\to  c\,\epsilon^{p-\a}\,.
\eeao
The \rhs\ converges to zero as $\epsilon\to 0$. Here and in what follows,
$c$ denotes any positive constant whose value is not of interest. We conclude
that \eqref{eq:june2aa} is automatic for $p>\a$.
\ere  
\bre 
Condition $\conditionsmallsetsp{p}$ is challenging to check for $p\le \a$. For $p/\a \in (1/2,1]$,
by \v Cebyshev's inequality, 
\beao
\lefteqn{\P\big( \big| \big\|\ov{x_n^{-1}\bfX_{[1,n]}}^\epsilon\big\|_p^p - \E[\big\|\ov{x_n^{-1}\bfX_{[1,n]}}^\epsilon\big\|_p^p]\big|> \delta \big)/[ n\,\P(|\bfX_0| >  x_n )]}\\
    &\le & \delta^{-2}\var \big(\big\|\ov {x_n^{-1}\bfX_{[1,n]}}^\epsilon\big\|_p^{p}\big)/[n\,\P(|\bfX_0|>x_n)]\\
%&\le& c\,\dfrac{\var\big(|\ov{x_n^{-1}\bfX}^\epsilon|^{p}\big)}{\P(|\bfX_0| > x_n)}
%        +c\,\sum_{h=1}^{n-1}|\corr\big(|\ov{x_n^{-1}\bfX_0}^\epsilon|^{p},\ov{|x_n^{-1}\bfX_h}^\epsilon|^{p}\big)|\,
%\dfrac{\E[|\ov{x_n^{-1}\bfX}^\epsilon|^{2p}]}
%{\P(|\bfX_0| > x_n)}\,. 
%\eeao
%\beao
&\le& \delta^{-2} \dfrac{\E[|\ov{x_n^{-1}\bfX_0}^\epsilon|^{2p}\big]}{\P(|\bfX_0| > x_n)}\,\Big[1 +2\sum_{h=1}^{n-1}|\corr\big(|\ov{x_n^{-1}\bfX_0}^\epsilon|^{p},\ov{|x_n^{-1}\bfX_h}^\epsilon|^{p}\big)|\Big]\,.
%&\sim& c\epsilon^{2p - \alpha}\Big[1+\sum_{h=1}^\infty \rho_h\Big]. 
\eeao
Now assume that  $(\bfX_t)$ is $\rho$--mixing with 
summable rate \fct\ $(\rho_h$); cf.
Bradley \cite{bradley:2005}. Then the \rhs\ is bounded by
\beao
\delta^{-2} \,\dfrac{\E[|\ov{x_n^{-1}\bfX}^\epsilon|^{2p}\big]}{\P(|\bfX_0| > x_n)}\,\Big[1+2\sum_{h=1}^\infty \rho_h\Big]\sim 
\delta^{-2}\epsilon^{2p - \alpha}\Big[1+2\sum_{h=1}^\infty \rho_h\Big]\,,\qquad \epsilon\to 0\,,
\eeao
where we applied Karamata's theorem in the last step, and 
$\conditionsmallsetsp{p}$ follows. For Markov chains
weaker assumptions such as the 
drift condition $({\bf DC})$
in Mikosch and Wintenberger \cite{mikosch:wintenberger:2014,mikosch:wintenberger:2016}
 can be used for checking $\conditionsmallsetsp{p}$. 
\ere
\bre\label{rem:xn}
Condition $\conditionsmallsetsp{p}$ not only restricts the serial  
dependence of the time series $(\bfX_t)$ but also the level of 
thresholds $(x_n)$. Indeed, for $p/\a<1/2$ and $(\bfX^\prime_t)$ iid, since $\big(\| n^{-1/2} \bfX^\prime_{[1,n]} \|_p^p-\E[\| n^{-1/2} \bfX^\prime_{[1,n]} \|_p^p]\big)$ 
converges in \ds\ to a Gaussian limit by virtue of the \clt , $\conditionsmallsetsp{p}$  implies 
 necessarily that $x_n/\sqrt n\to \infty$ as $n\,\P(|\bfX_0|>x_n)\to0$.
 \ere
 \subsection{Threshold condition}\label{subsec:Moderate:threshold}
{
In Theorem~\ref{thm:main:theorem} {we assume} growth conditions on $(x_n)$: $n/x_n^{p} \to 0$ if $p<\alpha$ and $n/x_n^{\alpha - \kappa}\to 0$ for some $\kappa > 0$ if $p=\a$. 
\par 
For inference purposes it is tempting to decrease the threshold level $x_n$ 
{\st\ more exceedances are included in the estimators.} 
Indeed, the assumptions on $(x_n)$ can be {relaxed}, justified by results such as Nagaev's large deviation principle in \cite{nagaev:1979}, by adding a centering term as we will show in Lemma~\ref{prop2:spt3:finite_constant}. However, in this section we aim at pointing at the difficulties that might arise while doing so in practice. }
\par 
To motivate the results of this section we start by considering an
iid \seq\ $(\bfX_t)$ satisfying   $\rvalpha$ for some $\a>0$. Then, for 
$p> \a$, \eqref{eq:ld} holds with limit $c(p)=1$ and 
$S_n^{(p)}= \sum_{t=1}^n|\bfX_i|^p$ has infinite expectation.
If $p<\a$ the process $(S_{n}^{(p)})$ 
has finite expectation
and by the law of large numbers, for $n/x_n^p\to 0$,
\beam\label{eq:moddev}
\lefteqn{\P\big(\|\bfX_{[0,n]}\|_p > x_n\,(n\,x_n^{-p}\,\E[|\bfX|^p]+1)^{1/p}\big)}\nonumber\\
&=&\P\big( S_{n}^{(p)}-\E[S_{n}^{(p)}] > x_n^p (1+o(1))\big)\to 0\,.
\eeam
Following Nagaev \cite{nagaev:1979},
a \ld\ result for the centered process holds: 
\beao
 \P\Big(S_{n}^{(p)} -\E[S_{n}^{(p)}] >x_n^p\Big)
\sim n \,\P(|\bfX_0|>x_n)\,,\quad \nto\,,
\eeao 
provided $n/x_n^{\a-\kappa}\to 0$ for $p/\a\in (1/2,1)$ and some $\kappa>0$, and $\sqrt{n\,
\log n}/x_n^p\to0$ for $p/\a<1/2$. These conditions are {satisfied} for extreme thresholds satisfying $n/x_n^p \to 0$. In this case the centering term $\E[S_{n}^{(p)}]$ in \eqref{eq:moddev} is always negligible which allows us to derive \eqref{eq:ld}. Next, we extend the previous ideas to regularly varying time series.
\par 
%For inference purposes it is tempting to decrease the threshold level $x_n$  to include more exceedances justified by results such as  Nagaev's large deviation principle \cite{nagaev:1979}. This can be achieved  by carefully  dealing with the centering $\E[S_{n}^{(p)}]$. This is the content of the next lemma.
%\par
%Let $(a_n)$ be a sequence {\blue satisfying  $n\,\P(|\bfX_0| > a_n) \to 1$ as $n \to \infty$. Assume there exists a sequence {\red $r_n=b_n?$} $r := r_n \to \infty$ satisfying  $n/r_n \to \infty$ such that the sequence $(r_n,a_n)$ satisfies the conditions 
%{\red what is $(x_n)$?} $\anticlust$ and $\conditionsmallsets{p}$. }
\begin{lemma}\label{prop2:spt3:finite_constant} Consider an $\bbr^d$-valued 
stationary process $(\bfX_t)$ satisfying the conditions 
$\rvalpha$, $\anticlust$, $\conditionsmallsets{p}$ and
$c(p)<\infty$ for some $p>0$. If $p<\a$ then 
\beam\label{eq2:cp}
 \lim_{n \to  \infty} \frac{\P\big(\|\bfX_{[0,n]}\|_p>x_n\,\big( n\,x_n^{-p}\,\E[|\bfX|^p] +1\big)^{1/p}\big)}{ n \,\P(|\bfX_0| > x_{n}) } = c(p)\,.
\eeam
If $p=\a$ then 
\beam\label{eq2:cpa}
 \lim_{n \to  \infty} \frac{\P\big(\|\bfX_{[0,n]}\|_{\a}>x_{n} \, \big(n\E[|\overline{\bfX/x_{n}}^1|^\alpha] + 1 \big)^{1/\a}\big)}{n \,\P(|\bfX_0| > x_{n}) } = c(\a) = 1\,.
\eeam
Moreover, if also $\E[|\bfX|^\alpha] < \infty$ then equation $(\ref{eq2:cp})$ holds  for $p = \alpha$.
\end{lemma}
The proof is given in Section \ref{subsec:lemmay}. Now the restrictions on the level of the thresholds $(x_n)$ are the ones implicitly implied by condition $\conditionsmallsetsp{p}$ in \eqref{eq:june2aa}; see Remark \ref{rem:xn}.
\par 
%For moderate thresholds $x_n$ 
We define an auxiliary sequence of  levels:
\beao
    z_n :=z_n(p)= \begin{cases}x_n\,
         \big(n\,x_n^{-p}\E[|\bfX|^p]+ 1\big)^{1/p}    &\text{ if } p < \alpha,\\
           x_{n} \, \big(n\E[|\overline{\bfX/x_{n}}^1|^\alpha] + 1 \big)^{1/\a} &\text{ if } p = \alpha\,,\\
      x_{n}&\text{ if } p > \alpha\,. 
    \end{cases}
\eeao
For thresholds satisfying the growth conditions $n/x_n^p \to 0$ 
we have $z_n\sim x_n$, while for moderate thresholds satisfying $\conditionsmallsetsp{p}$ and
$n/x_n^p \to \infty$ this is no longer the case. 
%\par
%We extend Proposition~\ref{Prop:large_deviations} by correcting the 
%threshold $x_n$ with a bias term as in $z_n$. 
%Then,  conditioning on moderate threshold levels, we deduce the following 
%proposition whose proof is postponed to Section~\ref{proof:lem:lip}.
%\begin{proposition}\label{lem:lip}
%Consider an $\bbr^d$-valued 
%stationary process $(\bfX_t)$ satisfying the conditions  
%$\rvalpha$, $\anticlust$, $\conditionsmallsets{p}$ and   $c(p)<\infty$ for some $p>0$. Then, for every bounded Lipschitz-continuous function $f:(\tilde{\ell^p}, \tilde{d_p}) \to \mathbb{R}$ the following relation holds 
%\beam\label{eq:Lip2}
%\lefteqn{ \E[f(\bfX_{[0, n]}/x_{n}) \, | \, \|\bfX_{[0, n]}\|_p > z_n(p) ] - \E[f(\bfX_{[0, n]}/x_{n})]} \nonumber\\
%   & \to& \E[f(Y\bfQ^{(p)})]-f({\bf0})\,, \qquad {\nto}\,.
%\eeam
%\end{proposition} 
\par 
For the purposes of inference Lemma~\ref{prop2:spt3:finite_constant} is not as satisfactory as \eqref{eq:ld} in Theorem~\ref{thm:main:theorem}. Indeed, 
the level $z_n$ in the selection of the exceedances is not 
the original threshold $x_n$.
For any moderate threshold $x_n$ with $z_n/x_n\to \infty$ 
the use of $x_n$ instead of $z_n$ might yield to a different limit. 
As a toy example, consider the problem of inferring the constant $c(q)/c(p)$ for
$p < \alpha$, $q > p$. Then an application of Lemma~\ref{prop2:spt3:finite_constant}  ensures that
\begin{align*}
    &\P( \|\bfX_{[0,n]}\|_q>z_n(q)   \,| \, \|\bfX_{[0,n]}\|_p > z_n(p))\\
        & \quad \quad \to \P(\|Y\bfQ^{(p)}\|_q >1) = \E[\|\bfQ^{(p)}\|^\alpha_q] = c(q)/c(p), \qquad \nto\,.
\end{align*}
However, choosing the same moderate threshold $z_n = z_n(q)$, we would have 
\beao
    \P( \|\bfX_{[0,n]}\|_q>z_n   \,| \, \|\bfX_{[0,n]}\|_p > z_n)
    &\sim&\dfrac{\P( \|\bfX_{[0,n]}\|_q>z_n )}{\P( n^{1/p}\E[|\bfX|^p]^{1/p} > z_n)} \\
        &\to& \begin{cases}
    1 & \text{ if } q < \alpha,\\ 
    0 & \text{ if } q > \alpha, \, \quad \nto\,.
    \end{cases}
\eeao 
{
{By this argument, the %assumption $n/x_n^{p\land(\alpha - \kappa)} \to 0$ 
growth conditions on $(x_n)$ are justified to simplify inference procedures. Otherwise, the choice of the threshold \seq\ becomes delicate.} }
%\bre\label{rem:portmanteau} {\red needed here?}
%The weak convergence of bounded signed measures on the metric space $(\tilde{\ell^{p}},\tilde{d_p})$ does not follow directly from equation (\ref{eq:Lip2}) since the portemanteau theorem does not hold any longer.   If $\conditionsmallsetsp{p}$  holds a sufficient condition
%for the weak \con\ is that, for every open Borel set $G$,
%\begin{align}\label{eq:wlimit}
% \lim_{\delta \to 0} \limsup_{n \to + \infty} {\P \big( \bfX_{[0,r_n]}/a_{n} \in G, \tilde{d}_p( \bfX_{[0,r_n]}/a_{n}, G^c) < \delta \big) } = 0
%\end{align}
%A discussion on this is given in Section \ref{sec:portsigned}. 
%\ere

\section{Inference beyond shift-invariant functionals}\label{Declustering Methods in inference procedures}\label{sec:beyond}\setcounter{equation}{0}

So far we only considered inference for shift-invariant functionals acting on $(\tilde{\ell}^p, \tilde{d_p})$ such as maxima and sums. Following the shift-projection ideas in Janssen \cite{janssen:2019}, jointly with continuous mapping arguments, we extend inference to functionals on $(\ell^p, d_p)$.
\par 
\subsection{Inference for cluster functionals in $(\ell^p,d_p)$} 
Let $g:(\ell^p,d_p) \to \mathbb{R}$ be a bounded measurable function. We define the functional $\psi_g: (\wt{\ell}^{p},\wt{d}_p) \to  \mathbb{R}$  by 
\beam\label{eq:psif}
        [\bfz] \mapsto \psi_g([ \bfz]) &:=&\sum_{j \in \mathbb{Z}} |\bfz_{-j}^*|^\alpha g \big ((B^{j}\bfz_t^*)_{t\in\bbz}\big),
\eeam
where $ \bfz_t^{*} := \bfz_{t - T^*(\bfz)},$ for $ t \in \mathbb{Z}$, such that $T^*(\bfz) := \inf\{s \in \mathbb{Z} : |\bfz_s| = \|\bfz\|_\infty\}$ and $B:\ell^p \to \ell^p$ is the backward-shift map.
%and we recall the backward-shift map definition: $B^j(x_t) := (x_{t-j})$. 
\par 
We link the distribution of the spectral cluster process $\bfQ^{(\alpha)}$ from Equation~\eqref{eq:def:cluster:process} and the distribution of the class $[\bfQ^{(\alpha)}]$ through the mappings \eqref{eq:psif} in the next proposition whose proof is given in Section~\ref{subsub:ex2}.
\begin{proposition}\label{ex2}
The following relation holds for every real-valued bounded measurable function $g$ on $\ell^\a$
\beao
   \E[g(\bfQ^{(\alpha)})] &=& \E[\psi_g ([\bfQ^{(\alpha)}])]\,,
\eeao
where $\psi_g$ is as in \eqref{eq:psif}. This relation remains valid if $\a$ is replaced by $p$, whenever the spectral cluster process in $\ell^p$ is well defined.
\end{proposition}
For $p \le \a$ the mappings in \eqref{eq:psif} are continuous functionals 
on $(\tilde{\ell}^p, \tilde{d}_p)$ and we can extend 
Theorem~\ref{cor:consistency} to continuous functionals on $(\ell^p,d_p)$ evaluated at the spectral cluster process $\bfQ^{(p)}$ taking values in $(\ell^p, d_p)$. 
\begin{theorem}\label{thm:random_sampling}
Assume the conditions of Theorem~\ref{cor:consistency} for $ p\le \a$. Then for any continuous bounded 
function $g: \ell^p \cap \{\bfx : \|\bfx\|_p = 1, |\bfx_0| > 0\} \to  \mathbb{R}$, 
%there exist sequences $k_n \to \infty$, $m_n \to \infty$ {\green this is not needed: such that $m_n/k_n \to  \infty$} \st\ 
\beam\label{eq:estimate:f}
     \widehat{g}^{(p)} &:=& 
     \frac{1}{k}  \sum_{t=1}^{m} \,  \underbrace{\sum_{j=1}^{b} W_{j,t}(p) \, g\Big( \frac{ B^{j-1}\mathbf{B}_{t}}{\|\mathbf{B}_t\|_p} \Big)}_{=:\psi_g(\mathbf{B}_t/\|{\bf B}_t\|_p)}  \, \1(\| \mathbf{B}_t\|_p > \|\mathbf{B}\|_{p,( k  +1)} )\\& \stp &\E[g(\bfQ^{(p)})]\,,\qquad \nto\,,\nonumber
\eeam
where $W_{j,t}(p)=|\bfX_{(t-1)b+j}|^\alpha/\|\mathbf{B}_t\|_p^\alpha$ for all
$j=1,\dots, b$. 
\end{theorem}
The proof is given in Section~\ref{subsub:thm:rs}.
%\subsection{Applications}
\subsection{Applications}
Examples of {non-}shift-invariant functionals on 
$(\ell^{p}, d_p)$ are measures of serial dependence, probabilities of 
large deviations such as the supremum of a random walk and ruin probabilities, and functionals of the spectral tail process $\bfTh$.
%\eqd \bfQ^{(\a)}/|\bfQ^{(\a)}_0|$
 We study these examples in the remainder of this section.
\begin{comment}
\par We conclude the section with a discussion of $p-$cluster based methods for improving inference on $(\bfTh_t)$ and $\bfTh_0$.
\end{comment}
\subsubsection{Measures of serial dependence} Define $g_h(\bfx_t) = |\bfx_h|^{\a} \frac{\bfx_0^\top}{|\bfx_0|}\frac{\bfx_h}{|\bfx_h|}$. Then the following result 
is straightforward from Theorem~\ref{thm:random_sampling}.
\begin{corollary}\label{cor:measres:serial:dependence}
Assume the conditions of Theorem~\ref{thm:random_sampling} for $p=\a$. 
Then 
%there exists sequences $k_n \to \infty$, $m_n \to \infty$ {\green with $m_n/k_n \to  \infty$} such that
\begin{align*}
    \widehat g_{h}^{(\alpha)}& :=  \, 
        \frac{1}{k} {  \sum_{t=1}^{m} \underbrace{\sum_{j=1}^{b-h} W_{j,t} W_{j+h,t} \,  \frac{\bfX_{j,t}^{\top}}{|\bfX_{j,t}|}\,\frac{\bfX_{j+h,t}}{|\bfX_{j+h,t}|}}_{=: \psi_{g_h}(\mathbf{B}_t/\|\bfB_t\|_p)}\, \1(\| \mathbf{B}_t\|_\alpha > \|\mathbf{B}\|_{\alpha,( k +1 )})}. \\
        & \xrightarrow[]{\P} \,  \E[g_h(\bfQ^{(a)})], \quad \quad n \to + \infty,
\end{align*}
where the weights  $W_{j,t}=W_{j,t}(\a)$ are defined in Theorem~\ref{thm:random_sampling}, satisfying $\sum_{j=1}^{b}W_{j,t}= 1$, and $\bfX_{j,t} := \bfX_{(t-1)b + j}$ for $j=1,\dots,b$.
\end{corollary}
\begin{comment}
This is a measure of serial dependence in the spirit of 
the extremogram introduced by Davis and Mikosch \cite{davis:mikosch:2009}. In the univariate case, the extremogram is defined for sets $A=(a,\infty), B=(b,\infty)$, $a,b>0$, as {\green this looks complicated, but is just 
$(a/b)^\a \E[\Theta_h^\a\wedge (a/b)^\a]$ for positive $(X_t)$, it would be even 
simpler for $a=b=1$. It is difficult to see the relation with the next
display. Recommend to cancel. The extremogram is an autocorrelation \fct , this
function is not in general, but maybe it is? It is better to compare with an ACF as done below.} 
\begin{align*}
\rho_{AB}(h) &:= \lim_{h \to \infty}\P(X_h > x\,b \,| \, X_0> x\,a) \\
    &= (a/b)^{\alpha}\E[ |\Theta_h|^\alpha\,  (\Theta_h/|\Theta_h|)_+ (\Theta_0)_{+} \, \big( 1 \land (b/a)^\alpha 1/|\Theta_h|^\alpha \big) ] < 1\,,\qquad h\in\bbr\,.
\end{align*}
Recalling that $\bfQ^{(\alpha)} \eqd \bfTh/\|\bfTh\|_\alpha $, 
\begin{align*}
\E[g_h(\bfQ^{(\alpha)})] := \E\Big[ \frac{|\bfTh_h|^\alpha}{\|\bfTh\|^{ \a}_\alpha}\, \big({\bfTh_0}/{|\bfTh_0|}\big)^{\top} \big({\bfTh_h}/{|\bfTh_h|}\big)   \Big] < 1,
\end{align*}
We interpret this constant as a summary of the magnitude and direction of the time series $h$ lags after recording a high-level exceedance of the norm. 
It is also informative about the probability of recording future extreme events after lag $h$ since the spectral cluster process is re-normalized to have $\alpha$-modulus one.
\end{comment}
The function $g_h$ gives a summary of the magnitude and direction of the time series $h$ lags after recording a high-level exceedance of the norm, and satisfies the relation $\sum_{h \in \mathbb{Z}}\E[g_h(\bfQ^{(\alpha)})]=1$. 
%In particular, for the linear model we obtain a link with the autocovariance functions as we show in the example below.
\begin{example}\rm 
Let $(X_t)$ be a linear process satisfying the assumptions in Example~\ref{ex:AR}, then
\beao%\label{eq:exampleACF}
    \E[g_h(Q^{(\alpha)})] = \frac{\sum_{t \in \mathbb{Z}}  |\varphi_t|^{\a}  |\varphi_{t+h}|^{\a}  \sign(\varphi_{t})\sign(\varphi_{t+h}) }{ \big(\|\varphi\|_\alpha^{\alpha}\big)^2}, \qquad h\in\bbz\,.
\eeao
This \fct\ is proportional to the autocovariance \fct\ of a finite variance
 linear process
with coefficients $(|\varphi_t|^\a\,\sign(\varphi_t))$. In particular, for $\a=1$ it is proportional to the autocovariance \fct\ of a finite variance linear process 
with coefficients $(\varphi_t)$.
\end{example}
\subsubsection{Large deviations for the supremum of a random walk} We start by reviewing Theorem 4.5 in Mikosch and Wintenberger \cite{mikosch:wintenberger:2016}; the proof is given in Section~\ref{subsub:LGsup}.
\begin{proposition}\label{prop:LGsup}
Consider a univariate stationary \seq\ $(X_t)$ satisfying  $\rvalpha$ 
for some $\alpha \ge 1$, $\anticlust$, $\conditionsmallsets{1}$, and $c(1) <  \infty$. Then for all $p \ge 1$,
\begin{align}\label{eq:Dec2}
    &\Big| \, \frac{\P(\sup_{1 \le t \le n}S_t > x_n)}{n\P(|X_1| > x_n)}   \nonumber \\
        & \qquad \qquad - c(p)\, \E\Big[ \lim_{s \to  \infty} \Big(\mbox{$\sup_{t \ge -s} \sum_{i=-s}^t Q^{(p)}_i$} \Big)^{\alpha}_{+} \Big] \, \Big| \to 0, \quad n \to  \infty.
\end{align}
\end{proposition}
\begin{comment}
It is not straightforward to estimate the constant in \eqref{eq:Dec2} using Theorem~\ref{thm:random_sampling} with $p = \infty$ since the function $g(\bfx) = \lim_{s \to + \infty } \big( \sup_{t\ge -s} \sum_{i=-s}^t x_i \big)_+^\a$ is not bounded over the $\ell^\infty$-sphere. 
\end{comment}
\par 
If $\a \ge 1$, then $\|Q^{(1)}\|_\alpha^\alpha \le \|Q^{(1)}\|^\alpha_1 = 1$ and a consistent estimator of $c(1) = 1/\E[ \|Q^{(1)}\|^\alpha_\alpha]$ was suggested in Section~\ref{sec:sums}. A consistent estimator of the term in \eqref{eq:Dec2} is given next.
\begin{corollary}\label{cor:ruin}
Assume the conditions of  Theorem~\ref{thm:random_sampling} for $p = 1$. Then
%, there exist sequences $k_n \to + \infty$, $m_n \to + \infty$ with $m_n/k_n \to + \infty$ such that
\begin{align}\label{eq:largedevestimator}
    &\Big| \frac{ \sum_{t=1}^{m} \big(\sup_{1\le j \le b} \frac{X_{t,j} }{\|\mathbf{B}_t\|_1} \big)_{+}^{\alpha} \1(\| \mathbf{B}_t\|_1 > \|\mathbf{B}\|_{1,(k+1)}) }{  \sum_{t=1}^{m}\frac{\|\mathbf{B}_t\|^\alpha_\alpha }{\|\mathbf{B}_t\|^\alpha_1} \1(\| \mathbf{B}_t\|_1 > \|\mathbf{B}\|_{1,(k+1)})} \nonumber \\
        &\qquad \qquad \qquad  -c(1)\, \E\Big[ \lim_{s \to  \infty} \Big(\mbox{$\sup_{t \ge -s} \sum_{i=-s}^t Q^{(1)}_i$} \Big)^{\alpha}_{+} \Big] \Big| \xrightarrow[]{\P} 0\,,\qquad n \to  \infty\, , \nonumber
\end{align}
where $X_{t,j}:=X_{(t-1)b+j}$, for $1\le j\le b$, $1 \le t  \le m$.
\end{corollary}
Following the same ideas and using Theorem 4.9 in \cite{mikosch:wintenberger:2016}, one can also derive a consistent estimator for the constant 
in the related ruin problem.
\subsubsection{Application: a cluster-based method for inference on $(\bfTh_t)$}\label{sec:theta}\label{sec:forward:tail}\label{sec:inftheta:ge0}
Exploiting the relation  $(\bfQ^{(\alpha)}_t)/|\bfQ^{(\alpha)}_0| \eqd (\bfTh_t)$ discussed in Section~\ref{sec:spectral:a}, we propose cluster-based estimation methods for the spectral tail process.
\par 
Cluster-based approaches with the goal to improve inference 
on $\bfTh_1$ for Markov chains were considered in 
Drees {\it et al.} \cite{drees:segers:warchol:2015}; see also 
Davis {\it et al.} \cite{davis:drees:segers:warchol:2018} and Drees {\it et al.} \cite{drees:janssen:neblung:2021} for related cluster-based procedures on $(\bfTh_{t})_{|t| \le h}$ for fixed $h \ge 0$. Our approach can be seen as an 
extension for inference on the $\ell^\a$-valued \seq\ $(\bfTh_t)$.
\par 
Consider the continuous re-normalization function
 $\zeta(\bfx)=\bfx/|\bfx_0|$ on
$\{\bfx \in \ell^\alpha: |\bfx_0| > 0\}$.
We derive the following result from Theorem~\ref{thm:random_sampling}; the proof is given in Section~\ref{subsub:cc}.
\begin{proposition}\label{cor:infthe}
Assume the conditions of 
Theorem \ref{thm:random_sampling} for $p=\a$. 
Let $\rho: (\ell^\alpha,d_\a) \to  \mathbb{R}$ be a homogeneous continuous function and $\rho_\zeta(\bfx) := (\rho^\alpha \land 1) \circ \mathbf{\zeta}(\bfx) $.  
Then for $k=k_n\to\infty$,
 \beao
 \widehat{\rho_\zeta}^{(\a)}:=\frac{1}{k}\sum_{t=1}^m \psi_{\rho_\zeta}(\bfB_t) \, \1(\| \mathbf{B}_t\|_p > \|\mathbf{B}\|_{p,(k+1)} )
\stp \P(\rho(Y\,\bfTh) > 1) \,, \quad n \to  \infty\,,
 \eeao
 where $\psi_{\rho_\zeta}(\bfB_t)$ is defined in \eqref{eq:estimate:f} and
the Pareto$(\a)$ random variable $Y$ is independent of $\bfTh$.
\end{proposition}
Classical examples of such functionals are
$\rho(\bfx ) = \max_{i \ge 0,j \ge i} (x_i - x_{j})_{+}$, functionals related to 
large deviations such as  $\rho(\bfx ) = \sup_{t \ge 0} (\sum_{i=0}^{t}x_i )_{+}$, or measures of serial dependence such as $\rho(\bfx) = |\bfx_h|$.

\section{Cluster inference implementation for regularly varying linear process}\label{ex:AR} 
%\subsubsection{An example: a \regvary\ linear process
{In this section we illustrate the index estimators of Corollaries~\ref{cor:extremal_index} and \ref{cor:stable:limit} for  a \regvary\ linear process 
\beao 
X_t &:=& \sum_{j \in \mathbb{Z} } \varphi_j Z_{t-j}, \quad t\in\bbz,
\eeao 
where $(Z_t)$ 
is an iid real-valued \regvary\   \seq\ with (tail)-index $\a>0$, 
and $(\varphi_j)$ are real coefficients \st\ $\sum_{j \in \mathbb{Z}}|\varphi_j|^{1\wedge (\a-\vep)}<\infty$ for some $\vep>0$. 
\par 
In this setting, $(X_t)$ is \regvary\
with the same (tail)-index $\a>0$, and the \ds s of $Z_t$
and $X_t$ are tail-equivalent; see Davis and Resnick \cite{davis:resnick:1986}.
The spectral cluster process of $(X_t)$  is given by  
$Q^{(\alpha)}_t = (\varphi_{t+J}/\|(\varphi_t)\|_\alpha)\, \Theta_0^{Z}$, $t\in\bbz$,
where
$\lim_{\xto}\P(\pm Z_0 > x)/\P(|Z_0| > x) = \P(\Theta_0^Z = \pm 1)$, $\Theta_0^Z$ is independent of a random shift $J$ with \ds\
$\P( J = j) = |\varphi_j|^\alpha/\|(\varphi_t)\|_\alpha^\alpha$; see Kulik and Soulier \cite{kulik:soulier:2020}, (15.3.9). Then 
\begin{align*}
   c(\infty)\; = \; \mbox{$ \max_{t\in\bbz}$}|\varphi_t|^\a/\|\varphi\|_\a^\a
%\sum_{t\in\bbz}|\varphi_t|^\a
\,,\qquad  c(1) \; = \;   \big(\mbox{$\sum_{t\in\bbz}$}|\varphi_t|\big)^\a/\|\varphi\|_\a^\a\,.
%\sum_{t\in\bbz}|\varphi_t|^\a\,.
\end{align*}
For the causal AR(1) model given by
$X_t = \varphi\, X_{t-1} + Z_t$, $t \in \mathbb{Z}$, $|\varphi| < 1$,  one retrieves $\theta_{|X|} = c(\infty) = 1-|\varphi|^\alpha $ and $c(1) =  (1- |\varphi|^\alpha)/(1- |\varphi|)^{\alpha}$.
\par 
{
We aim to illustrate the estimators of $\theta_{|X|}$ and $c(1)$ built on extremal $\ell^\alpha$--blocks for the causal AR$(1)$ model with student$(\a)$ noise. 
Guided by \eqref{eq:k},  we take $k= k_n  = \lfloor n/b_n^{2} \rfloor$ as
\[
    {k_n =[ m_n \P(\|\bfB\|_\a > x_{b_n})]} \;\sim\;  n\, \P(|\bfX_0| > x_{b_n}) \;=\; o(n/b_n^{1+\kappa} ),
\]
for $\kappa>0$ sufficiently small using the Potter bound.
For estimation of $\alpha$, we follow
the bias-correction procedure in de Haan {\it et al.}~\cite{dehaan:mercadier:zhou:2016}. 
This estimator is plugged into \eqref{eq:estimate:thet}, \eqref{eq:c11}, resulting in 
the estimators $\widehat{\theta}_{|\bfX|}$, $\widehat{c}(1)$, as a function of block lengths. %, based on $\ell^\a$-%and $\ell^1$-extremal blocks. %$, respectively. 
Figures~\ref{fig:2} and \ref{fig:3} present boxplots (in blue) of these estimators 
as a function of $b_n$ and for different sample sizes $n$.  
For comparison, we also show boxplots (in white) of the estimators 
in \eqref{eq:extremal:index:2} and \eqref{eq:c12} based on extremal $\ell^\infty$--blocks. 
Inference based on $\ell^\alpha$--block, %for $p < \infty$ 
coupled with a Hill-type estimate of $\alpha$,
seems to be robust compared to the $\ell^\infty$--blocks approach. In all examples the block length $b=32$ gives nice results for the $\ell^\a$--approach in terms of bias and dispersion. Instead, the $\ell^\infty$--estimator appears to be highly sensitive to the block length choice. 
%As expected from \eqref{eq:k}, we should use a smaller $k$ for $\ell^\infty$--block based inference to reduce the bias. 
%In our simulation, $b_n = 32$ for $n=8\,000$ are reasonable choices for all models considered, regardless of the quantity to be estimated.
Also, notice that the bias for large block lengths decreases as $n$ increases.
%since if we fix $n$ 
Indeed, if we fix $n$, the relation $\lfloor n/b^{2}\rfloor \to 0$ as $ b \to \infty$ restricts the block length for small sample sizes.
We also refer to Buritic\'a {\it et al.} \cite{buritica:mikosch:meyer:wintenberger} for further simulation 
experiences showing that the estimator of the extremal index in 
\eqref{eq:estimate:thet} compares favorably with various classical
estimators as regards bias.} }
 %with the blocks estimator in Hsing \cite{hsing:1993} support this idea, and show $\widehat{\theta}_{|\bfX|}$ is also competitive compared to further estimators of the extremal index regarding bias.
\begin{figure}[htbp]
  \centering
   \includegraphics[width=0.9\textwidth, height=0.265\textwidth]{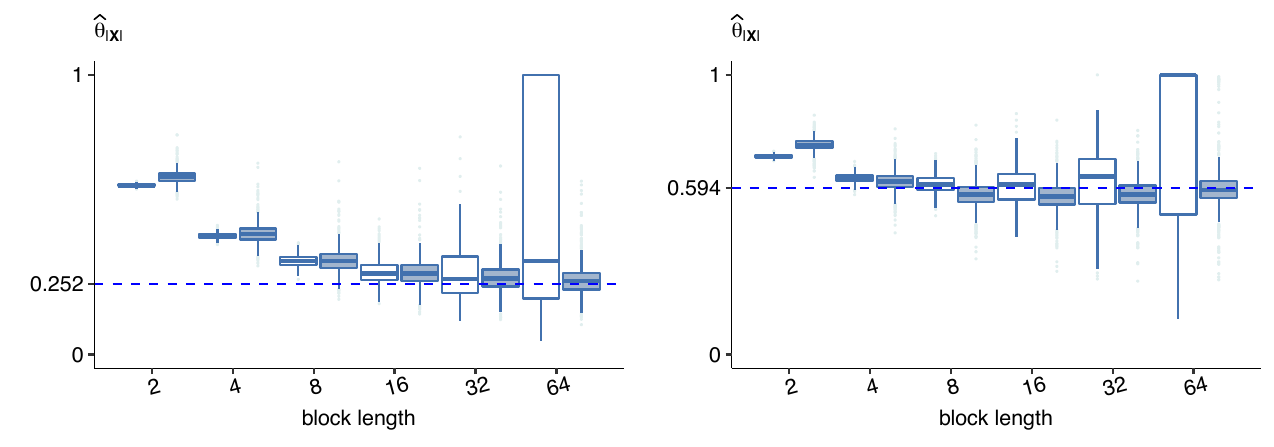}
    \includegraphics[width=0.9\textwidth, height=0.265\textwidth]{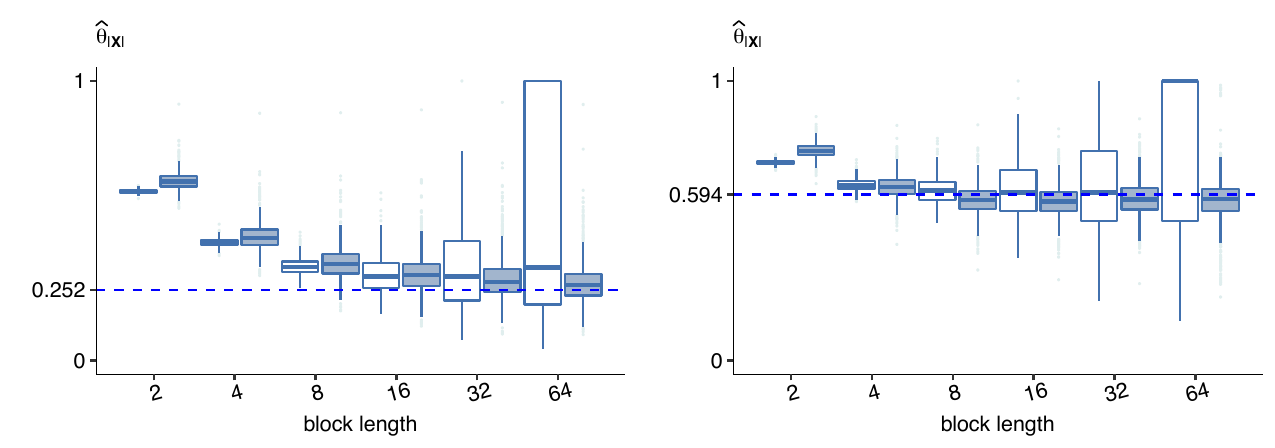}
    \includegraphics[width=0.9\textwidth, height=0.265\textwidth]{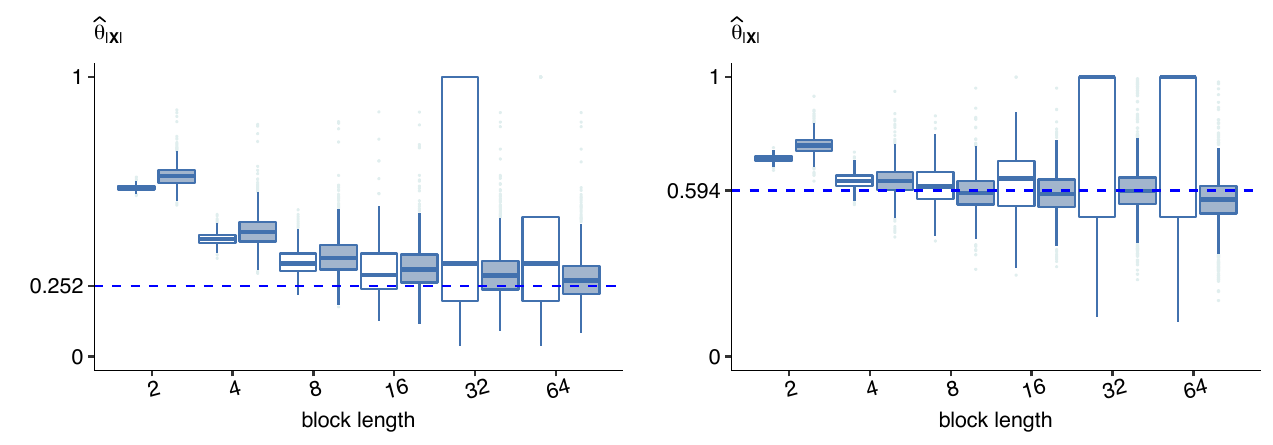}
    \bfi{\rm Boxplot of estimates $\widehat{\theta}_{|X|}$ as a function of $b_n$ from \eqref{eq:estimate:thet} for inference through $\bfQ^{(\alpha)}$ (in blue) and from \eqref{eq:extremal:index:2} through $\bfQ^{(\infty)}$  (in white). $1\,000$ simulated samples $(X_t)_{t=1,\dots,n}$ from a causal AR(1) model with student$(\a)$ noise with $\a=1.3$ and $\varphi = 0.8$ (left column), $\varphi = 0.5$ (right column) were considered. Rows correspond to results for $n = 8\, 000$, $4\, 000$, $2\, 000$ from top to bottom.} 
    \label{fig:2}
\efi
\end{figure}

\begin{figure}[htbp]
  \centering
    \includegraphics[width=0.9\textwidth, height=0.265\textwidth]{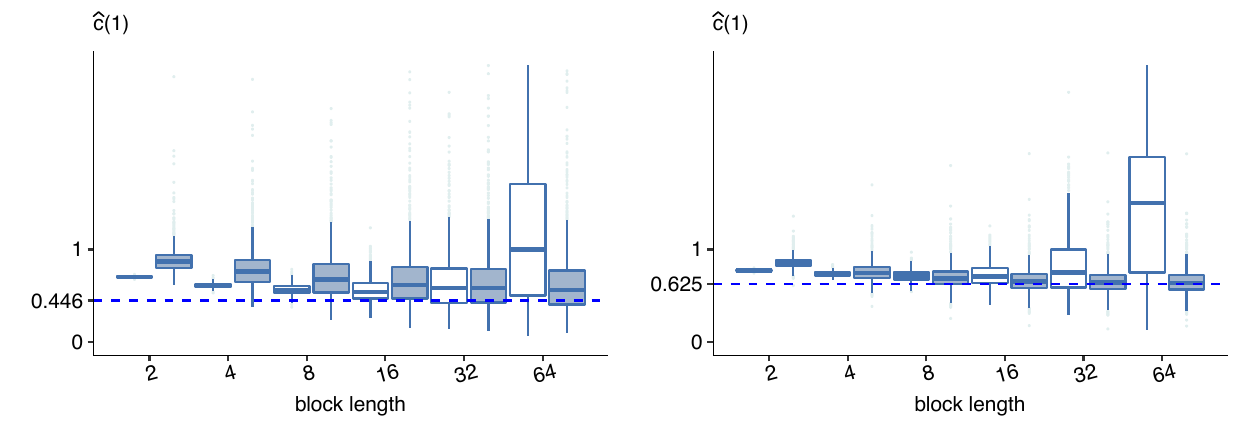}
    \includegraphics[width=0.9\textwidth, height=0.265\textwidth]{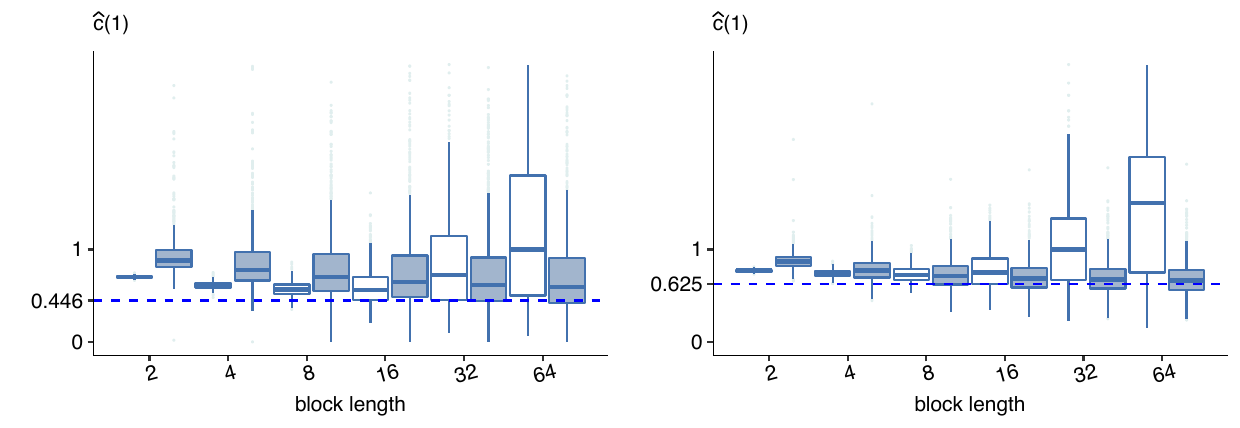}
    \includegraphics[width=0.9\textwidth, height=0.265\textwidth]{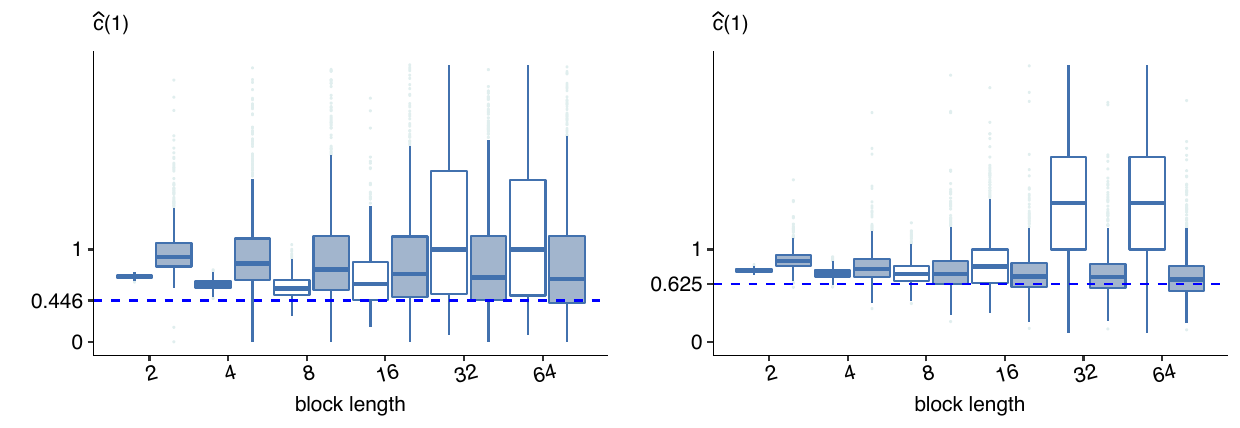}
    \bfi{\rm
    Boxplot of estimates
$\widehat{c}(1)$ as a function of $b_n$ from \eqref{eq:c11} for inference through $\bfQ^{(\a)}$ (in blue) and from \eqref{eq:c12} for inference through $\bfQ^{(\infty)}$  (in white). We simulate $1\,000$ samples $(X_t)_{t=1,\dots,n}$ from an AR(1) model with student($\a$) for $\a=0.7$ and $\varphi = 0.8$ (left column), $\varphi = 0.5$ (right column) were considered. Rows correspond to $n= 8\,000, 4\,000, 2\,000$ from top to bottom. }
    \label{fig:3}
\efi
\end{figure}

\section{Proofs}\setcounter{equation}{0}\label{sec:proof}

\subsection{Proof of Theorem~\ref{thm:main:theorem}}\label{sec:LD}
%We start with a con\seq\ of 
%Theorem 3.1 in Mikosch and Wintenberger~\cite{mikosch:wintenberger:2013}.
Recall the properties of the
\seq\ $(x_n)$ from Section~\ref{subsec:assumptions}, in particular   $n\,\P(|\bfX_0|>x_n)\to 0$. The main result in Theorem~\ref{thm:main:theorem} follows 
{by applications of} Lemma~\ref{prop:spt3:finite_constant} and Proposition~\ref{Prop:large_deviations} below; {their proofs are given at the end of this section.}
\begin{comment}
and the fact that $c(p)<\infty$ for $p\ge \a$ while one needs additional
conditions for this to hold in the case $p\in (0,\a)$; see Section~\ref{subsec:ellp}.
\end{comment}
%{\red The lemma remains valid for any \seq s $(n)$, $(a_n)$ satisfying
%$\anticlust$ and $\conditionsmallsets{p}$, $n\,\P(|\bfX_0|>a_n)\to 1$ is not neede%d.}
\begin{lemma}\label{prop:spt3:finite_constant} Consider an $\bbr^d$-valued 
stationary time series $(\bfX_t)$ satisfying the conditions 
$\rvalpha$, $\anticlust$, $\conditionsmallsets{p}$.  %Assume $c(p)<\infty$ for some $p>0$ and, in addition,
If $p < \alpha$, assume also $n/x_n^p \to 0$ and, if $p=\alpha$, $n/x_n^{\alpha - \kappa} \to 0$ for some $\kappa > 0$. Then the following relation holds
\beam\label{eq:cp}
 \lim_{n \to  \infty} \frac{\P(\|\bfX_{[0,n]}\|_p > x_n)}{n \,\P(|\bfX_0| > x_n) } = c(p) \,,
\eeam
where $c(p)$ is given in \eqref{eq:dec10a}.
\end{lemma}
%The proof is given in Section~\ref{subsec:lemmaya}.  
We recall from Remark~\ref{rem:may27} that \eqref{eq:june2aa} in $\conditionsmallsets{p}$ 
is always satisfied for $p>\a$. Moreover, for $p\le \a$, under the growth conditions on $(x_n)$ in Theorem~\ref{thm:main:theorem}, centering with the expectation in   \eqref{eq:june2aa} is not necessary.
\par
{
%We refer to a relation of the type \eqref{eq:cp} as {\em \ld\  \pro ies} motivated by the following observation. Write $S_{k}^{(p)}=\sum_{t=1}^{k} |\bfX_t|^p$ for $k\ge 1$. Then  $|\bfX|^p$ is \regvary\ with index $\alpha/p$. Relation \eqref{eq:cp} implies that
%\beao
%\P(\|\bfX_{[0,n]}\|_p > x_n)&=& \P\big(S_{n}^{(p)} >x_n^p\big)\\
%&\sim& c(p)\,n\,\P(|\bfX_0|>x_n)\to 0\,,\quad\nto\,.
%\eeao
%Thus the left-hand \pro y describes the rare event that the sum process $S_{n}^{(p)}$ exceeds the  extreme threshold $x_n^p$.
\par
%Proposition~\ref{Prop:large_deviations} below extends the \ld\ result for $\|\bfX_{[0,n]}\|_p$ in \eqref{eq:cp} to a large deviation result for the process $\bfX_{[0,n]}$ in the \seq\ space $\ell^p$. The proof is given in Section~\ref{subsec:prold}.}
\par
\begin{proposition}\label{Prop:large_deviations}
Assume the conditions of Lemma~\ref{prop:spt3:finite_constant}.
 Then,
\beam \label{eq:prop_LD2}
& &\P(x_n^{-1}\bfX_{[0,n]} \in \cdot \, | \, \|\bfX_{[0,n]}\|_p > x_n) \stw
\P(Y \bfQ^{(p)} \in \cdot ), \quad \nto\,,  \nonumber \\
& &
\eeam
in the space $(\tilde \ell^p,\tilde d_p)$ where the Pareto$(\a)$ \rv\
$Y$ and $\bfQ^{(p)}$ are independent.
\end{proposition}

\subsection*{Proof of Lemma~\ref{prop:spt3:finite_constant}}\label{subsec:lemmaya}
Choose some $\epsilon > 0$, $\delta\in (0,1)$. Since $\|a_n^{-1}\bfX_{[0,n]}\|_p^p$
is a sum of non-negative \rv s we have the following bounds via truncation
\beam\label{eq:dec11a}\lefteqn{
    \P \big( \|\underline{x_n^{-1}\bfX_{[0,n]}}_\epsilon\|_p^p > 1 )  \;\leq 
    \;\P(\|x_n^{-1}\bfX_{[0,n]}\|_p^p >1 \big)} \nonumber\\
    &\le& \P\big(\|\underline{x_n^{-1}\bfX_{[0,n]}}_\epsilon\|_p^p > (1-\delta^p) \big) + \P\big( \|\overline{x_n^{-1}\bfX_{[0,n]}}^\epsilon\|_p^p >  \delta^p \big)\,.
\eeam
By $\conditionsmallsets{p}$ and in view of Remark~\ref{rem:may27} 
we have
\beam\label{eq:june2bb}
\lim_{\epsilon\downarrow 0}\limsup_{n\to \infty} \P\big( \|\overline{x_n^{-1}\bfX_{[0,n]}}^\epsilon\|^p_p >  \delta^p 
\big)/(n \P(|\bfX_0| > x_n))=0\,.
\eeam 
%In particular, for $p \le \alpha$ we used that $\E[ \|\overline{x_n^{-1}\bfX_{[0,n]}}^\epsilon\|^p_p] = o(1)$ as $n \to \infty$. This holds since $n/x_n^{p} \to 0$ for $p < \alpha$ and $n/x_n^{\alpha - \delta} \to 0$ for any small $\delta > 0$ if $p = \alpha$. 
Now, for any choice of $u>0$, it remains to determine the limits  
of the terms $\P \big( \|\underline{x_n^{-1}\bfX_{[0,n]}}_\epsilon\|_p^p > u )/(n \P(|\bfX_0| > x_n))$. We start with a telescoping sum \rep 
\beao
\lefteqn{\P(\|\underline{x_n^{-1}\bfX_{[0,n]}}_\epsilon\|_p^p > u)-\P(| \underline{x_n^{-1}\bfX_{0}}_\epsilon|_p^p > u )}\\
        &=& \mbox{$\sum_{i = 1}^{n}$}      \big(\P(  \|\underline{x_n^{-1}\bfX_{[0,i]}}_\epsilon\|_p^p > u) - \P(  \|\underline{x_n^{-1}\bfX_{[0,i-1]}}_\epsilon\|_p^p > u)\big)\\
&=&
\mbox{$\sum_{i = 1}^{n}$}  \E \big[ \big(\1( \| \underline{x_n^{-1}\bfX_{[0,i]}}_\epsilon\|_p^p > u )  
      -\1(\|\underline{x_n^{-1}\bfX_{[1,i]}}_\epsilon\|_p^p > u )\big)\, \1( |\bfX_0| > \epsilon x_n)\big]\,,
\eeao
where we used stationarity in the last step
and the fact that the difference of the 
indicator \fct s vanishes on $\{|\bfX_0| \le  \epsilon x_n\}$. We also observe that the second term on the \lhs\ is of the order $o(n\,\P(|\bfX_0|>x_n))$.
For any fixed $k$ write $A_k=\{\max_{k \le t \le n}|\bfX_t| > \epsilon x_n\}$.
Regular variation of $(\bfX_t)$ ensures that, as $\nto$,
\beao
\lefteqn{\dfrac{\P(\|\underline{x_n^{-1}\bfX_{[0,n]}}_\epsilon\|_p^p > u)}{n\,\P(|\bfX_0|>x_n)}}\\
        &\sim&  \epsilon^{-\alpha} \frac{1}{n} 
\sum_{i = 1}^{n} \E \big[ \1( \|\underline{x_n^{-1}\bfX_{[0,i]}}_\epsilon\|_p^p > u) -\1( \|\underline{x_n^{-1}\bfX_{[1,i]}}_\epsilon\|_p^p > u) \; \big| \; |\bfX_0| > \epsilon x_n \big] \\
        &\sim&\epsilon^{-\alpha} 
\E \big[ \big(\1(\|\underline{x_n^{-1}\bfX_{[0,k-1]}}_\epsilon\|_p^p > u) -\1(\|\underline{x_n^{-1}\bfX_{[1,k-1]}}_\epsilon\|_p^p > u)\big)\\
    &&   \times \1(A_k^c)\, \big|\,|\bfX_0| > \epsilon x_n \big]              + \epsilon^{-\alpha} O\big(\P(A_k  \,\mid \,   |\bfX_0| > \epsilon x_n )\big)\,,
\eeao
where the second term vanishes, first letting $\nto$ and then $k\to\infty$,
by virtue of $\anticlust$. Now the \regvar\ property of $(\bfX_t)$ implies that
\beao
\lefteqn{\lim_{n \to  \infty} \frac{\P(\|\underline{x_n^{-1}\bfX_{[0,n]}}_\epsilon\|_p^p > u)}{n\P(|\bfX_0| >x_n)}}\\
  &=&                      
\lim_{k \to  \infty} \epsilon^{-\alpha} \big( \P\big( \mbox{$\sum_{t = 0}^{k-1}$}         |\epsilon\,Y \bfTh_t|^p \1(|Y\bfTh_t| > 1) > u\big)\\&&\hspace{1.5cm} - \P\big( \mbox{$\sum_{t = 1}^{k-1}$}  |\epsilon Y\bfTh_t|^p \1(|Y\bfTh_t| > 1) > u ) \big) \big),
\eeao 
by a change of variable this term equals
\beao 
%such that
%\begin{align*}
%    &\P( \sum_{t = 0}^{k-1}         |\epsilon \bfY_t|^p \1_{|\bfY_t| > 1} > u) - \P( \sum_{t = 1}^{k-1} |\epsilon\bfY_t|^p \1_{|\bfY_t| > 1} > u) \\
%            &\;\;= \E\Big[ \int_{1}^\infty \1( \sum_{t = 0}^{k-1}         |\epsilon y \bfTh_t|^p \1_{|y \bfTh_t| > 1} > u) \\
%            &\;\;\;\;\;\;\;\;\;\;\;\;\;\;\;\;\;\;\;\;\;\;\;\;\;\;\;\;\;\;\;\; - \1( \sum_{t = 1}^{k-1} |\epsilon y \bfTh_t|^p \1_{|y\bfTh_t| > 1} > u)  \; d(-y^{-\alpha})\Big]  
%\end{align*}
%\end{comment}
&=&\lim_{k\to \infty}\E\big[ \mbox{$\int_{\epsilon}^\infty$} \big( \1\Big(\mbox{$\sum_{t = 0}^{k-1}$}      |y \bfTh_t|^p \1(|y \bfTh_t| > \epsilon) > u\big)\\&&\hspace{2cm}- \1\big( \mbox{$\sum_{t = 1}^{k-1}$}       |y \bfTh_t|^p \1(|y\bfTh_t| > \epsilon) > u\big)\big)  \; d(-y^{-\alpha})\big] \\
            &=& \lim_{k \to  \infty}  \E\big[ \mbox{$\int_{0}^\infty$}  \big(\1\big(\|\underline{y\bfTh_{[0,k-1]}}_\epsilon\|_p^p > u\big) - \1\big(\|\underline{y\bfTh_{[1,k-1]}}_\epsilon\|_p^p > u\big)\big)\,d(-y^{-\alpha})\big]\,.
\eeao
In the last step we used the fact that the integrand vanishes for $0\le y\le \epsilon$. The integrand is non-negative and bounded by 
$\1(y > \epsilon)$ which is integrable. Thus we may take 
the limit as $\kto$ inside the integral to derive the quantity
\begin{align*}
            & \E\big[ \mbox{$\int_{0}^\infty$}  \Big(\1\big(\|\underline{y\bfTh_{[0,\infty]}}_\epsilon\|_p^p > u\big) - \1\big(\|\underline{y\bfTh_{[1,\infty]}}_\epsilon\|_p^p > u\big)\Big)\,d(-y^{-\alpha})\big] \,.
\end{align*}
By monotone convergence as $\epsilon\downarrow 0$ we get the limit
\beam\label{eq:june3aa}
u^{-\alpha/p}\E \left[  \|\bfTh_{[0,\infty]}\|_p^\alpha  - \|\bfTh_{[1,\infty]}\|_p^\alpha   \right] =u^{-\alpha/p}\,c(p)\,. 
%            & =& u^{-\alpha/p}  \E \big[  \big(1 + \|\bfTh_{[1,\infty]}\|_p^p \big)^{\alpha/p}   -\|\bfTh_{[1,\infty]}\|_p^\alpha    \big]\,. 
\eeam
An application of this
formula  and a telescoping sum argument yield
\beam\label{eq:dec11b}
\lefteqn{\E\Big[\|(\bfTh_t)_{t\ge 0}\|^\alpha_p - \|(\bfTh_t)_{t\ge 1}\|^\alpha_p \Big] \,.}\\   
    &=&\nonumber\E\Big[\|\bfTh\|_\alpha^\alpha\big(\|(\bfTh_t)_{t\ge 0}/\|\bfTh\|_\alpha\|^\alpha_p - \|(\bfTh_t)_{t\ge 1}/\|\bfTh\|_\alpha\|^\alpha_p\big) \Big]\\
    & =& \nonumber\sum_{s\in\Z}\E\Big[|\bfTh_s|^\alpha\big(\|(\bfTh_t)_{t\ge 0}/\|\bfTh\|_\alpha\|^\alpha_p - \|(\bfTh_t)_{t\ge 1}/\|\bfTh\|_\alpha\|^\alpha_p\big) \Big] \nonumber\\
    & =& \nonumber\sum_{s\in\Z}\E\Big[ \big(\|(\bfTh_t)_{t\ge -s}/\|\bfTh\|_\alpha\|^\alpha_p - \|(\bfTh_t)_{t\ge -s+1}/\|\bfTh\|_\alpha\|^\alpha_p\big) \Big]\\
    &=& \E[\|\bfTh\|_p^\alpha/\|\bfTh\|_\a^\a]\nonumber\; = \; c(p).
\eeam
Now an appeal to \eqref{eq:dec11a} with $u=1$ and $u=1-\delta^p$
yields
\beao
c(p)&\le&
\liminf_{\nto} \dfrac{\P(\|x_n^{-1}\bfX_{[0,n]}\|_p^p >1 \big)}{n\,
\P(|\bfX_0|>x_n)}\\
&\le &\limsup_{\nto}\dfrac{\P(\|x_n^{-1}\bfX_{[0,n]}\|_p^p >1 \big)}{n\,\P(|\bfX_0|>x_n)}\le (1-\delta^p)^{-\a/p}\,c(p)\,.
\eeao
The limit relation \eqref{eq:cp} follows as $\delta\downarrow 0$.

\subsection*{Proof of Proposition~\ref{Prop:large_deviations}}\label{subsec:prold}
Consider any bounded Lipschitz-continuous \fct\ 
 $f: (\tilde{\ell^p}, {\wt d_p})  \to \mathbb{R}$.
% From the previous lemma we have that the constant $\lim_{h \to + \infty}c(p,h) := c(p)$ is finite. 
The statement is proved if we can show that 
 \beao\lefteqn{
 \lim_{n \to \infty} \E[ f(x_n^{-1} \bfX_{[0,n]} )\, | \, \|\bfX_{[0,n]}\|_p > x_n]}\\  & =&   c(p)^{-1} \mathbb{E}\big[ \|{\bfTh}/{\|\bfTh\|_\alpha}\|^\a_p f  \big(Y \,{\bfTh}/{\|\bfTh\|_p}\big)  \big]\,.
\eeao
In view of Lemma~\ref{prop:spt3:finite_constant} it suffices to show
\beao
\lefteqn{\lim_{n \to \infty} \dfrac{\E[ f(x_n^{-1} \bfX_{[0,n]} )\,\1( \|\bfX_{[0,n]}\|_p > x_n)]}{n\,\P(|\bfX_0|>x_n)}}\\
& =&  \, \mathbb{E}\big[ \|{\bfTh}/{\|\bfTh\|_\alpha}\|^\a_p f  \big(Y \,{\bfTh}/{\|\bfTh\|_p}\big)  \big]\,.
\eeao
In these limit relations we may replace $f(x_n^{-1} \bfX_{[0,n]} )$
by $f(\underline{x_n^{-1} \bfX_{[0,n]}}_{\epsilon} )$ since by \eqref{eq:june2bb}
for  any $\delta>0$, some $K_f>0$,
\beao\lefteqn{\lim_{\vep_\downarrow 0}\limsup_{\nto}
\dfrac{\P\big(|f(x_n^{-1}\bfX_{[0,n]})-  f(\underline{x_n^{-1} 
\bfX_{[0,n]}}_{\epsilon})|>\delta\,,\|\bfX_{[0,n]}\|_p>x_n\big)}{n\,\P(|\bfX_0|>x_n)}}
\\
&\le &\lim_{\vep\downarrow 0}\limsup_{\nto}\dfrac{\P\big(K_f \,d_p(\ov{x_n^{-1} 
\bfX_{[0,n]}}^{\epsilon},{\bf0})>\delta\big)}
{n\,\P(|\bfX_0|>x_n)}\\
&\le &\lim_{\vep\downarrow 0}\limsup_{\nto}\dfrac{\P\big(K_f\,\big\|\ov{x_n^{-1} 
\bfX_{[0,n]}}^{\epsilon}\big\|_p^p>\delta(p)\big)}{n\,\P(|\bfX_0|>x_n)}=0\,,
\eeao
where $\delta(p)=\delta^p$ for $p\ge 1$ and $=\delta$ for $p\in(0,1)$.
We also have  for $\delta\in (0,1)$,
\beao G_{n,\epsilon}&=&
\dfrac{\E\big[ f\big(\underline{x_n^{-1} \bfX_{[0,n]}}_\epsilon\big )\,\big|\1\big( \|x_n^{-1}\bfX_{[0,n]}\|_p > 1\big)-\1\big( \|\underline{x_n^{-1}\bfX_{[0,n]}}_\epsilon\|_p > 1\big)\big|\big]
}{n\,\P(|\bfX_0|>x_n)}\\
&\le &c\,\dfrac{\P(\big \|x_n^{-1}\bfX_{[0,n]}\|_p > 1\ge\|\underline{x_n^{-1}\bfX_{[0,n]}}_\epsilon\|_p\big)}{n\,\P(|\bfX_0|>x_n)}\\
&\le &c\,\dfrac{\P\big( \|\ov{x_n^{-1}\bfX_{[0,n]}}^\epsilon\|_p > \delta
\big)}{n\,\P(|\bfX_0|>x_n)}
+c\,\dfrac{\P\big( 1\ge \|\underline{x_n^{-1}\bfX_{[0,n]}}_\epsilon\|_p>1-\delta\big)}{n\,\P(|\bfX_0|>x_n)}\\
&=&G_{n,\epsilon,\delta}^{(1)}+G_{n,\epsilon,\delta}^{(2)}\,.
\eeao
Applying \eqref{eq:june2bb} to $G_{n,\epsilon,\delta}^{(1)}$ and using the 
calculations in the proof of Lemma~\ref{prop:spt3:finite_constant} 
leading to \eqref{eq:june3aa} for  $G_{n,\epsilon,\delta}^{(2)}$, we conclude that
\beao
\lim_{\epsilon\downarrow 0}\limsup_{\nto} G_{n,\epsilon}
&\le& \lim_{\epsilon\downarrow 0}\limsup_{\nto} G_{n,\epsilon,\delta}^{(1)}
+ \lim_{\epsilon\downarrow 0}\limsup_{\nto} G_{n,\epsilon,\delta}^{(2)}\\
&=& 0+ c\, \big((1-\delta)^{ -\a/p}-1\big)\downarrow 0\,,\qquad \delta\downarrow 0\,.
\eeao
Thus it suffices to show
\beam\lefteqn{
\lim_{\epsilon\downarrow 0} \lim_{\nto}\dfrac{\E\big[f\big(\underline{x_n^{-1} \bfX_{[0,n]}}_\epsilon \big)\,\1\big( \|\underline{x_n^{-1} \bfX_{[0,n]}}_\epsilon\|_p > 1\big)\big]}{n\,\P(|\bfX_0|>x_n)}}\qquad\qquad\qquad\qquad\qquad\qquad\nonumber\\
&=& \mathbb{E}\big[ \|\bfTh/\|\bfTh\|_\a\|^\a_p f  \big(Y \,\bfTh/\|\bfTh\|_p\big)  \big]\,.\label{eq:dec14a}
\eeam
This is the goal of the remaining proof.
\par
Choose any $\epsilon > 0$. Noticing that $\underline{x_n^{-1} \bfX_{[0,n]}}_\epsilon = \underline{x_n^{-1} \bfX_{[1,n]}}_\epsilon$ on 
$\{|x_n^{-1}\bfX_0| \le  \epsilon\}$,
 we have 
\beao
I&:=&\E\big[  f\big(\underline{x_n^{-1} \bfX_{[0,n]}}_\epsilon \big)\,\1\big( \|\underline{x_n^{-1} \bfX_{[0,n]}}_\epsilon\|_p > 1\big) \big] \\
        & =& \E\Big[\Big( f\big(\underline{x_n^{-1} \bfX_{[0,n]}}_\epsilon \big)\,\1\big( \|\underline{x_n^{-1} \bfX_{[0,n]}}_\epsilon\|_p > 1\big)     \\
         &&\hspace{3mm}- f\big((0,\underline{x_n^{-1} \bfX_{[1,n]}}_\epsilon )\big)\,\1\big( \|\underline{x_n^{-1} \bfX_{[1,n]}}_\epsilon\|_p > 1\big)\Big)\,
\1(|\bfX_0| > \epsilon\,x_n\big) \Big]  \\
        % && + \E\big[\Big(  f\big(\underline{a_n^{-1} \bfX_{[0,n]}}_\epsilon \big)\,\1\big( \|\underline{a_n^{-1} \bfX_{[0,n]}}_\epsilon\|_p > 1\big)   \\
         %& &\hspace{6mm}-  f\big((0,\underline{a_n^{-1} \bfX_{[1,n]}}_\epsilon )\big)\,\1\big( \|\underline{a_n^{-1} \bfX_{[1,n]}}_\epsilon\|_p > 1\big)\,\Big)\,
%\1( |\bfX_0| \leq  \epsilon\,a_n) \big] \Big) \\
         && + \E\big[  f((0,\underline{x_n^{-1} \bfX_{[1,n]}}_\epsilon ))\,\1\big( \|\underline{x_n^{-1} \bfX_{[1,n]}}_\epsilon\|_p > 1\big) \big] \\
%&=& \Big( \E\big[  f\big(\underline{a_n^{-1} \bfX_{[0,n]}}_\epsilon \big)\,\1\big( \|\underline{a_n^{-1} \bfX_{[0,n]}}_\epsilon\|_p > 1\big)     \\
%&&\hspace{3mm}- f\big((0,\underline{a_n^{-1} \bfX_{[1,n]}}_\epsilon )\big)\,\1\big( \|\underline{a_n^{-1} \bfX_{[1,n]}}_\epsilon\|_p > 1\,,|a_n^{-1} \bfX_0| > \epsilon\big) \big] \Big) \\
%&=&\E\big[  f(\big(0,\underline{a_n^{-1} \bfX_{[1,n]}}_\epsilon )\big)\,\1\big( \|\underline{a_n^{-1} \bfX_{[1,n]}}_\epsilon\|_p > 1\big) \big] \\
&=& \E\Big[\Big( f\big(\underline{x_n^{-1} \bfX_{[0,n]}}_\epsilon \big)\,\1\big( \|\underline{x_n^{-1} \bfX_{[0,n]}}_\epsilon\|_p > 1\big)     \\
         &&\hspace{3mm}- f\big((0,\underline{x_n^{-1} \bfX_{[1,n]}}_\epsilon )\big)\,\1\big( \|\underline{x_n^{-1} \bfX_{[1,n]}}_\epsilon\|_p > 1\big)\Big)\,
\1(|\bfX_0| > \epsilon\,x_n\big) \Big]  \\&&+
\E\big[  f\big((0,\underline{x_n^{-1} \bfX_{[0,n-1]}}_\epsilon )\big)\,
\1\big( \|\underline{x_n^{-1} \bfX_{[0,n-1]}}_\epsilon\|_p > 1\big) \big]\,.
\eeao
where we used the stationarity in the last step. Using the same idea 
recursively, we obtain
\beao
I&=&%    &\E[  f(\underline{a_n^{-1} \bfX_{[0,n]}}_\epsilon )\1( \|\underline{a_n^{-1} \bfX_{[0,n]}}_\epsilon\|_p > 1) ] \\ 
 %%%%%%% %%%%%%% %%%%%%% %%%%%%% %%%%%%% %%%%%%% %%%%%%% %%%%%%% %%%%%%% %%%%%%% %%%%%%% %%%%%%% %%%%%%% %%%%%%% First summand    
       \mbox{$\sum_{j = 1}^{n}$} \E\big[\big(  f\big(({\bf0}^{n-j},\underline{x_n^{-1} \bfX_{[0,j]}}_\epsilon )\big)\,\1\big( \|\underline{x_n^{-1} \bfX_{[0,j]}}_\epsilon\|_p > 1\big)     \\
        & & \hspace{10mm}- f\big(({\bf0}^{n-j +1},\underline{x_n^{-1} \bfX_{[1,j]}}_\epsilon )\big)\,\1\big( \|\underline{x_n^{-1} \bfX_{[1,j]}}_\epsilon\|_p > 1\big)\Big)\, \1(|\bfX_0| > \epsilon\,x_n) \big]  \\[2mm]
         && + \E\big[  f\big(({\bf0}^{n} ,\underline{x_n^{-1} \bfX_{0}}_\epsilon )\big)\,\1\big(|\bfX_{0}|>\epsilon\,x_n\big)\big]\,,
\eeao
where $\bfzero^k := \{0\}^k$ for $k \ge 1$.
%the last term simplifies so we obtain
%\begin{align*}         
 %%%%%%% %%%%%%% %%%%%%% %%%%%%% %%%%%%% %%%%%%% %%%%%%% %%%%%%% %%%%%%% %%%%%%% %%%%%%% %%%%%%% %%%%%%% %%%%%%% 2 summand
%         &= \sum_{j = 1}^{n} \Big( \E[  f(0^{n-j},\underline{a_n^{-1} \bfX_{[0,j]}}_\epsilon )\1( \|\underline{a_n^{-1} \bfX_{[0,j]}}_\epsilon\|_p > 1)    % \\
%         & \;\;\;\;\;\;\;\;\;\;\;\;\;\;\;\;\;\;\; - f((0^{n-j +1},\underline{a_n^{-1} \bfX_{[1,j]}}_\epsilon )\1( \|\underline{a_n^{-1} \bfX_{[1,j]}}_\epsilon\|_p > 1) \; ; \; |a_n^{-1} \bfX_0| > \epsilon ] \Big) \\
%         & \;\;\;\;\;\;\;\;\;\;\;\;\;\;\;\;\;\;\;\;\;\;\;\;\;\;\;\;\;\;\;\;\;\;\;\;\;\;\;\;\;\;\;\;\;\;\;+ \E[  f(0^{n} ,\underline{a_n^{-1} \bfX_{0}}_\epsilon ) \; ; \; |a_n^{-1} \bfX_{0}| > \epsilon  ] \\ 
%\end{align*}
By regular variation of $\bfX_0$ the last right-hand term is
$o(n\,\P(|\bfX_0|>x_n))$.
%\beao
%\E\big[  f\big(({\bf0}^{n} ,\underline{x_n^{-1} \bfX_{0}}_\epsilon )\big)\,\1\big(|\bfX_{0}|>\epsilon\,x_n\big)\big]=o\big(n\,\P(|\bfX_0|>x_n)\big)\,,\qquad \nto\,.
%\eeao
Therefore by \regvar\ of $(\bfX_t)$ we obtain as $\nto$,
\beao
\lefteqn{I/\big(
%\E[  f(\underline{a_n^{-1} \bfX_{[0,n]}}_\epsilon )\1( \|\underline{a_n^{-1} \bfX_{[0,n]}}_\epsilon) > 1\| ]}
n \P( |\bfX_0| > x_n)\big)} \\
&\sim&  \frac{ \epsilon^{-\alpha}}{n} 
    \mbox{$\sum_{j = 1}^{n}$} \E\big[  f\big(({\bf0}^{n-j},\underline{x_n^{-1} \bfX_{[0,j]}}_\epsilon )\big)\,\1\big( \|\underline{x_n^{-1} \bfX_{[0,j]}}_\epsilon\|_p > 1\big)     \\
      &   & \hspace{1.4cm}- f\big(({\bf0}^{n-j +1},\underline{x_n^{-1} \bfX_{[1,j]}}_\epsilon )\big)\,\1\big( \|\underline{x_n^{-1} \bfX_{[1,j]}}_\epsilon\|_p > 1\big) \, \big| \, |\bfX_0| > \epsilon\,x_n \big]\\
&=:&II \,.
\eeao
Write  $A_k=\{\|\bfX_{[k,n]}\|_\infty> \epsilon\,x_n\}$ for fixed $k\ge 1$.
 By $\anticlust$, $\P(A_k\mid |\bfX_0|> \epsilon\, x_n)$ vanishes by first letting $\nto$ then $\kto$.
Since each of the summands in II is uniformly bounded in absolute value
we may restrict the summation to $j\in\{k-1,\ldots,n\}$ for any fixed
$k\ge 1$.
Therefore we have as $\nto$,  
\beao
\lefteqn{II- O\big( 
\P(A_k \mid |\bfX_0| > \epsilon\,x_n )\big)
 }\\&\sim&
\frac{ \epsilon^{-\alpha}}{n} 
\sum_{j = k-1}^{n} \E\big[ \big( f\big(({\bf0}^{n-j},\underline{x_n^{-1} \bfX_{[0,j]}}_\epsilon )\big)\,\1\big(\{ \|\underline{x_n^{-1} \bfX_{[0,j]}}_\epsilon\|_p > 1\}\big)     \\
      &   & - f\big(({\bf0}^{n-j +1},\underline{x_n^{-1} \bfX_{[1,j]}}_\epsilon )\big)\,\1\big( \{\|\underline{x_n^{-1} \bfX_{[1,j]}}_\epsilon\|_p > 1\}\big)\Big)\,\1(A_k^c)\big| \, |\bfX_0| > \epsilon\,x_n \big]\\
&=&
\frac{ \epsilon^{-\alpha}}{n} \sum_{j = k-1}^{n}  \E\big[  f\big(({\bf0}^{n-j},\underline{x_n^{-1} \bfX_{[0,k-1]}}_\epsilon,{\bf0}^{j-k} )\big)\,
\1\big( \|\underline{x_n^{-1} \bfX_{[0,k-1]}}_\epsilon\|_p > 1\big)     \\
         & & \hspace{3mm}- f\big(({\bf0}^{n-j +1},\underline{x_n^{-1} \bfX_{[1,k-1]}}_\epsilon, {\bf0}^{j-k} )\big)\,\1\big( \|\underline{x_n^{-1} \bfX_{[1,k-1]}}_\epsilon\|_p > 1\big)\, \big|\, |\bfX_0| > \epsilon\,x_n \big]\,.
\eeao
%By virtue of $\anticlust$ the last term vanishes, first letting $\nto$ and 
%then $k\to\infty$.   
\begin{comment}
\begin{align*}
    & \frac{\E[  f(\underline{x_n^{-1} \bfX_{[0,n]}}_\epsilon )\1( \|\underline{x_n^{-1} \bfX_{[0,n]}}_\epsilon\|_p > 1) ]}{ n \P( |\bfX_0| > x_n)}  \\
    %%%%%%%%%% %%%%%%%%%% %%%%%%%%%% %%%%%%%%%% %%%%%%%%%% %%%%%%%%%% %%%%%%%%%% %%%%%%%%%% %%%%%%%%%% %%%%%%%%%% %%%%%%%%%% %%%%%%%%%% %%%%%%%%%% %%%%%%%%%% 1
        & \; \; \sim_{n \to + \infty}  \frac{ \epsilon^{-\alpha}}{n} \sum_{j = 0}^{n-k+1} \big( \E[  f(0^{j},\underline{x_n^{-1} \bfX_{[0,k-1]}}_\epsilon )\1( \|\underline{x_n^{-1} \bfX_{[0,k-1]}}_\epsilon\|_p > 1)     \\
         & \;\;\;\;\;\;\;\;\;\;\;\;\;\;\;\;\;\;\; - f((0^{j+1},\underline{x_n^{-1} \bfX_{[1,k-1]}}_\epsilon )\1( \|\underline{x_n^{-1} \bfX_{[1,k-1]}}_\epsilon\|_p > 1) \; | \; |x_n^{-1} \bfX_0| > \epsilon ] \big) \\
         & \;\;\;\;\;\;\;\;\;\;\;\;\;\;\;\;\;\;\;\;\;\;\;\;\;\;\;\;\;\;\;\;\;\;\;\;\;\;\;\;\;\;\;\;\;\;\; +  \frac{\epsilon^{-\alpha}}{n} \E[  f(0^{n} ,\underline{x_n^{-1} \bfX_{0}}_\epsilon ) \; | \; |x_n^{-1} \bfX_{0}| > \epsilon  ]  + o(k)
\end{align*}
\end{comment}
Next we apply shift-invariance and \regvar\ in $\tilde \ell^p$:
\beao
%\lefteqn{II-O\big( 
%\P(A_k \mid |\bfX_0| > \epsilon\,x_n )\big)}\\
& \sim         & \epsilon^{-\alpha} \E\big[  
f\big(\underline{x_n^{-1} \bfX_{[0,k-1]}}_\epsilon \big)\,\1\big( \|\underline{x_n^{-1} \bfX_{[0,k-1]}}_\epsilon\|_p > 1\big) 
       \\  &&\hspace{6mm}- f\big(\underline{x_n^{-1} \bfX_{[1,k-1]}}_\epsilon \big)\,
\1\big( \|\underline{x_n^{-1} \bfX_{[1,k-1]}}_\epsilon\|_p > 1\big) \, \big| \, 
|\bfX_0| > \epsilon\,x_n \big] \\
         &\to & \epsilon^{-\alpha} 
\E\big[  f\big(\underline{ \epsilon \,Y\,\bfTh_{[0,k-1]}}_\epsilon \big)\,
\1\big( \|\underline{\epsilon Y\,\bfTh_{[0,k-1]}}_\epsilon\|_p > 1\big)\\     
&&\hspace{7mm}- f\big(\underline{\epsilon\,Y \bfTh_{[1,k-1]}}_\epsilon \big)\,\1\big( \|\underline{ \epsilon \bfY_{[1,k-1]}}_\epsilon\|_p > 1\big)\big]\\
   &= &   \E\big[  f\big(\epsilon 
\, \underline{  Y\,\bfTh_{[0,k-1]}}_1 \big)\,\1\big( \|\epsilon  \,\underline{ \bfY_{[0,k-1]}}_1\|_p > 1\big)     
    \\&&\hspace{7mm}     - f\big(\epsilon \,  \underline{ \bfY_{[1,k-1]}}_1 \big)\,\1\big( 
\epsilon \,\|\underline{ Y \bfTh_{[1,k-1]}}_1\|_p > 1 \big)\big]=:J_{k,\epsilon}  \,.
\eeao
By Proposition~\ref{prop:Janssenequivalences} we have 
$\|\bfTh\|_\a<\infty$ a.s., $|\bfTh_t|\stas 0$ as $|t|\to\infty$, hence 
$T := \inf_{t \ge 0} \{t : Y\,|\bfTh_t|< 1 \}<\infty$ a.s. Then by monotone 
\con\ as $k\to\infty$,
\beao
         J_{k,\epsilon} & =& \epsilon^{-\alpha} \E\big[\big(  f\big(\epsilon \,\underline{ Y \bfTh_{[0,\infty]}}_1 \big)\,\1\big( \|\epsilon  \,\underline{ Y\,\bfTh_{[0,\infty]}}_1\|_p > 1\big)     \\
       &  & \hspace{9mm} - f\big(\epsilon \,\underline{ Y\,\bfTh_{[1,\infty]}}_1 \big)\,\1\big( \|\epsilon  \,\underline{  Y\,\bfTh_{[1,\infty]}}_1\|_p > 1\big)\big)  
\1(T < k)\big]\\&& +O(\P(T\ge k))\\
&\to& \epsilon^{-\alpha} \E\big[ f\big(\epsilon \,\underline{ Y \bfTh_{[0,\infty]}}_1 \big)\,\1\big( \|\epsilon  \,\underline{ Y\,\bfTh_{[0,\infty]}}_1\|_p > 1\big)     \\
       &  & \hspace{7mm} - f\big(\epsilon \,\underline{ Y\,\bfTh_{[1,\infty]}}_1 \big)\,\1\big( \|\epsilon  \,\underline{  Y\,\bfTh_{[1,\infty]}}_1\|_p > 1\big)\big]\\
&=&\int_0^{\infty}  \E\big[  f\big(\underline{  y\,\bfTh_{[0,\infty]}}_\epsilon \big)\,\1\big( \|\underline{y\, \bfTh_{[0,\infty]}}_\epsilon\|_p > 1\big)\\&&  \hspace{8mm}    - 
f\big(\underline{ y\,\bfTh_{[1,\infty]}}_\epsilon \big)\,\1\big( \|\underline{  y\,\bfTh_{[1,\infty]}}_\epsilon\|_p > 1\big)\big]\,d(-y^{-\alpha})
=:J_\epsilon \,.
\eeao
In the last step we changed variables, $u = \epsilon y$, and observed that 
the integrand vanishes for $y < \epsilon$.
\par
Finally, we want to let $\epsilon\downarrow 0$.
We start by inter-changing expectation and integral in $J_\epsilon$,
and change variables, $u = y \|\bfTh_{[0,\infty]}\|_\alpha$, in the first term
of the integrand and then proceed similarly for the second term
with the convention that it is zero on $\{\|\bfTh_{[1,\infty]}\|_\a=0\}$:
\beao
   J_\epsilon          & =& \E\Big[  \int_0^{\infty}\Big(\|\bfTh_{[0,\infty]}\|_\alpha^\alpha 
         \,  f\Big(  \underline{   \tfrac{y\,\bfTh_{[0,\infty]}}{ \|\bfTh_{[0,\infty]}\|_\alpha} }_{\epsilon}  \Big)\,
        \1\Big( \big \| \underline{   \tfrac{y\,\bfTh_{[0,\infty]}}{ \|\bfTh_{[0,\infty]}\|_\alpha} }_{\epsilon} \big\|_p > 1\Big)      \\
         & &-  \|\bfTh_{[1,\infty]}\|_\alpha^\alpha  \,f\Big( \underline{   \tfrac{y\,\bfTh_{[1,\infty]}}{ \|\bfTh_{[1,\infty]}\|_\alpha} }_{\epsilon} \Big)\,
         \1\Big( \big \| \underline{   \tfrac{y\,\bfTh_{[1,\infty]}}{ \|\bfTh_{[1,\infty]}\|_\alpha} }_{\epsilon} \big\|_p > 1\Big)\Big) \, d(-y^{-\alpha}) \Big]\\
          &  = & \sum_{t = 1}^{\infty}   \int_0^{ \infty} 
 \E\Big[ |\bfTh_t|^\alpha \Big(
           f\Big(  \underline{   \tfrac{y\,\bfTh_{[0,\infty]}}{ \|\bfTh_{[0,\infty]}\|_\alpha} }_{\epsilon}  \Big)\,
        \1\Big( \big \| \underline{   \tfrac{y\,\bfTh_{[0,\infty]}}{ \|\bfTh_{[0,\infty]}\|_\alpha} }_{\epsilon} \big\|_p > 1\Big)   \\
         & &-   f\Big(  \underline{   \tfrac{y\,\bfTh_{[1,\infty]}}{ \|\bfTh_{[1,\infty]}\|_\alpha} }_{\epsilon}  \Big)\,
        \1\Big( \big \| \underline{   \tfrac{y\,\bfTh_{[1,\infty]}}{ \|\bfTh_{[1,\infty]}\|_\alpha} }_{\epsilon} \big\|_p > 1\Big) \Big)    \Big] d(-y^{-\alpha}) \\
         &&  +
          \E\Big[
        \int_0^{\infty}     f\left(  \underline{   \tfrac{y\,\bfTh_{[0,\infty]}}{ \|\bfTh_{[0,\infty]}\|_\alpha} }_{\epsilon}  \right)
        \1\left( \big \| \underline{   \tfrac{y\,\bfTh_{[0,\infty]}}{ \|\bfTh_{[0,\infty]}\|_\alpha} }_{\epsilon} \big\|_p > 1\right) d(-y^{-\alpha})  \Big]\,.
\eeao
Next we apply the time-change formula \eqref{eq:june5a}
to each 
summand. 
\beao
     J_\epsilon     &=& \sum_{t = 1}^{\infty} \int_0^{ \infty} \E\Big[  
            f\Big(  \underline{   \tfrac{y\bfTh_{[-t,\infty]}}{ \|\bfTh_{[-t,\infty]}\|_\alpha} }_{\epsilon}  \Big)\,
        \1\Big( \big \| \underline{   \tfrac{y\bfTh_{[-t,\infty]}}{ \|\bfTh_{[-t,\infty]}\|_\alpha} }_{\epsilon} \big\|_p > 1\Big)     \\
         && -   f\Big(  \underline{   \tfrac{y\bfTh_{[1-t,\infty]}}{ \|\bfTh_{[1-t,\infty]}\|_\alpha} }_{\epsilon}  \Big)\,
        \1\Big( \big \| \underline{   \tfrac{y\bfTh_{[1-t,\infty]}}{ \|\bfTh_{[1-t,\infty]}\|_\alpha} }_{\epsilon} \big\|_p > 1\Big)   \Big]  d(-y^{-\alpha}) \\
         && +
          \E\Big[  
        \int_0^{ \infty}     f\Big(  \underline{   \tfrac{y\bfTh_{[0,\infty]}}{ \|\bfTh_{[0,\infty]}\|_\alpha} }_{\epsilon}  \Big)\,
        \1\Big( \big \| \underline{   \tfrac{y\bfTh_{[0,\infty]}}{ \|\bfTh_{[0,\infty]}\|_\alpha} }_{\epsilon} \big\|_p > 1\Big) d(-y^{-\alpha})  \Big]\,.
\eeao
This is a telescoping sum in $t$ with value
\beao
         J_\epsilon & =&
\E\Big[  
        \int_0^{ \infty}     f\Big(  \underline{   \tfrac{y\,\bfTh}{ \|\bfTh\|_\alpha} }_{\epsilon}  \Big)\,
        \1\Big( \big \| \underline{   \tfrac{y\,\bfTh}{ \|\bfTh\|_\alpha} }_{\epsilon} \big\|_p > 1\Big) d(-y^{-\alpha})  \Big]\,.
\eeao
By monotone \con\ we have
\beam\label{eq:change:of:norm}
\lim_{\epsilon\downarrow 0} J_\epsilon= \E\Big[  
        \int_0^{ \infty}     
f\big( y \tfrac{\bfTh}{\|\bfTh\|_\alpha}   \big)\,
        \1\Big( \big \| y \tfrac{\bfTh}{\|\bfTh\|_\alpha}  
 \big\|_p > 1\Big) d(-y^{-\alpha})  \Big]\,.
\eeam
Combining the arguments above,
we proved \eqref{eq:dec14a} as desired. \qed

{
\subsection{Proof of Lemma \ref{prop2:spt3:finite_constant}}\label{subsec:lemmay}
{\bf The case $p < \a$.}
Choose some $\epsilon > 0$, $\delta\in (0,1)$. We have the following bounds via truncation
\beao
I_1-I_2&:=&
    \P \big( \|\underline{x_n^{-1}\bfX_{[0,n]}}_\epsilon\|_p^p-\E\big[\|\underline{x_n^{-1}\bfX_{[0,n]}}_\epsilon\|_p^p\big] > 1 + \delta^p\big)\\&&-
\P\big(\|\ov{x_n^{-1}\bfX_{[0,n]}}^\epsilon\|_p^p-\E\big[\|\ov {x_n^{-1}\bfX_{[0,n]}}^\epsilon\|_p^p\big]\le -\delta^p
\big)\\&\leq& 
    \;\P(\|x_n^{-1}\bfX_{[0,n]}\|_p^p - \E\big[\|x_n^{-1}\bfX_{[0, n]}\|_p^p\big]>1 \big) \nonumber\\
    &\le& \P\big(\|\underline{x_n^{-1}\bfX_{[0,n]}}_\epsilon\|_p^p-\E\big[\|\underline{x_n^{-1}\bfX_{[0,n]}}_\epsilon\|_p^p
\big]
 > 1-\delta^p \big)\\&&
 + \P\big( 
\|\overline{x_n^{-1}\bfX_{[0,n]}}^\epsilon\|_p^p-\E\big[\|\overline{x_n^{-1}\bfX_{[0,n]}}^\epsilon\|_p^p\big] >  \delta^p \big)
=:I_3+I_4 \, .\eeao
Taking into account $\conditionsmallsets{p}$ for $p < \a$, we have
\[\lim_{\epsilon\downarrow 0}\limsup_{\nto}(I_2+I_4)/(n\,\P(|\bfX_0|>x_n))=0.\]
Moreover, 
we observe that by Karamata's theorem for $p<\a$
\beao
\E\big[\|\underline{x_n^{-1}\bfX_{[0,n]}}_\epsilon\|_p^p\big]&=&
 n \E\big[|\bfX_0/x_n|^p\,\1(|\bfX_0|>\epsilon x_n)\big]\\&=&
O\big( n\,\P(|\bfX_0|>\epsilon x_n)\big)=o(1)\,.
\eeao
Thus centering in $I_1$ and $I_3$ is not needed, and one can follow the
lines of the proof of Lemma \ref{prop:spt3:finite_constant} to conclude. \\[1mm]
{\bf The case $p = \a$.} It requires only slight changes; we omit details.
\qed

\subsection{Proofs of the results of Section~\ref{sec:spectral:cluster:process}}\label{sec:proof:sec:3}

\subsubsection{Proof of Proposition~\ref{lem:representation}}\label{subsec:lemma:representation:proof}
The representation {\eqref{eq:def:cluster:process} follows by identifying $\lim_{\epsilon \downarrow 0} J_\epsilon$ } as on the right-hand side of \eqref{eq:change:of:norm}. In particular, taking {$f$ as the} constant map $(\bfx_t) \mapsto 1$ in \eqref{eq:def:cluster:process} we obtain the representation of the constant $c(p)$ in \eqref{eq:dec10a}.

\subsubsection{Proof of Proposition~\ref{prop:limitSI}}\label{subsec:cor32}
Our goal is first to relate
the \seq\ of spectral components $(\bfQ^{(p)}(h))_{h\ge 0}$ to $(\bfTh_t)$.
%, and second to deduce a relation of $(\bfQ^{(p)}(h))_{h\ge 0}$ with $\bfQ^{(\a)}$ under the assumption $|\bfTh_t|\to 0$ a.s. as $t\to\infty$, which ensures that $\bfQ^{(\a)}$ is well defined.  
{We start with} two auxiliary results whose proofs are given at the end
of this section.

\ble\label{prop:constant_term_rho}
Let $(\bfX_t)$ be a stationary time series satisfying $\rvalpha$.
Then for $h \ge 0$, 
\beam\label{eq:limitlaw}
\P( \bfQ^{(p)}(h) \in \cdot) =
   \dfrac1 { c(p,h)}\,  \sum_{k = 0}^{h} 
             \E\Big[       \frac{\|\bfTh_{-k+[0,h]}\|_p^\alpha}{\|\bfTh_{-k+[0,h]}\|^\alpha_\alpha}   \1 \Big({  \tfrac{\bfTh_{-k+[0,h]}}{\|\bfTh_{-k+[0,h]}\|_p}}   \in \cdot  \Big)\Big]\,, \nonumber\\
\eeam
where $
    c(p,h) := \sum_{k=0}^{h}\E[\|\bfTh_{-k+[0,h]}\|_p^\alpha/\|\bfTh_{-k+[0,h]}\|_\alpha^\alpha]$.
In particular, $c(\alpha,h) = h+1$ and 
\begin{equation}\label{eq:unif:shift}
\P( \bfQ^{(\alpha)}(h) \in \cdot) = \P\big( \bfTh_{-U^{(h)} + [0,h]}/ \|\bfTh_{-U^{(h)} + [0,h]}\|_\alpha
 \in \cdot \big)\,,
\end{equation}
where $U^{(h)}$ is uniformly distributed on  $\{0,\dots,h\}$ and independent of $\bfTh$.
\ele
\ble\label{lem:cc} Assume $|\bfTh_t|\to 0$ as $t\to\infty$ and 
let $f:\tilde{\ell}^{\a} \cap\{\bfx : \|\bfx\|_p = 1\} \to (0,\infty)$ 
be any bounded Lipschitz-continuous function in 
$(\tilde{\ell}^{\a}, \tilde d_{\a})$. Then, for every $p \ge \a$,
\beam\label{eq:cc}
    &\frac{c(p,h)}{h+1}\,\E[f(\bfQ^{(p)}(h))] \to  \,\E[\|\bfTh/\|\bfTh\|_\a\|_p^{\a} f(\bfTh/\|\bfTh\|_p )]\,,
\eeam
as $h \to + \infty$
\ele
%{\red The following text is outcommented}
\begin{comment}{\green
This result can be generalized  by replacing $\|\cdot\|_p$
by some suitable modulus \fct\ as  advocated in  Segers {\it et al.} 
\cite{segers:zhao:meinguet:2017}.
For all $h \ge 0$ the spectral component $ \bfQ^{(p)}(h)\in \mathbb{R}^{d(h+1)}$ 
can be embedded in the \seq\ space $ \tilde \ell^p$. Moreover, if $|\bfTh_t|\to 0$ as $|t| \to \infty$ we may use the relation $\bfQ^{(\a)} \eqd \bfTh/\|\bfTh\|_\a$, then using \eqref{eq:unif:shift} the shift-invariant class 
$[\bfQ^{(\alpha)}(h)]$ satisfies: 
\[
[\bfQ^{(\alpha)}(h)] \eqd
[\bfQ^{(\alpha)}_{-U^{(h)} +[0,h]}/ \|\bfQ^{(\alpha)}_{-U^{(h)} + [0,h]}\|_\alpha].
\]
%With some abuse of notation, one can rewrite this relation as 
%$\bfQ^{(\alpha)}(h)\eqd \bfQ^{(\alpha)}_{U^{(h)} +[0,h]}/ \|\bfQ^{(\alpha)}_{U^{(h)} + [0,h]}\|_\alpha$ in $\wt\ell^\alpha$. 
Observe that $-U^{(h)}\stp -\infty$ as $h\to \infty$ and therefore we 
might expect
that $\P(\bfQ^{(\alpha)}(h) \in A)\to \P(\bfQ^{(\alpha)} \in A)$ 
for a class of shift-invariant Borel sets in $A\subset (\mathbb{R}^d)^{\mathbb{Z}}$ 
where $\bfQ^{(\a)}(h)$ are considered to be embedded in 
$(\mathbb{R}^d)^{\mathbb{Z}}$.} 
\end{comment}
 
We conclude from \eqref{eq:cc} for $f(\bfx)\equiv 1$ that 
%for any positive 
%bounded Lipschitz-continuous function $f$ on
%$\tilde{\ell}^p \cap \{\bfx : \|\bfx\|_p = 1\}$,
%\beao
%\lim_{h \to \infty}\dfrac{c(p,h)}{h+1}\E[f(\bfQ^{(p)}(h))]=
%c(p)\,\E[f(\bfQ^{(p)})]\,.
%\eeao 
$\lim_{h \to \infty}c(p,h)/(h+1) = c(p)$. If
$0< c(p)< \infty $,
\beam\label{eq:lim}
\lim_{h\to\infty}\E[f(\bfQ^{(p)}(h))] &=& c(p)^{-1} \,\E[\|\bfTh/\|\bfTh\|_\a\|^\a_p f(\bfTh/\|\bfTh\|_p )].
\eeam 
Finally, the portmanteau theorem yields $\bfQ^{(p)}(h)\std \bfQ^{(p)}(\infty)$
in $(\tilde{\ell}^p \cap \{\bfx : \|\bfx\|_p = 1\}, \tilde{d}_p)$ where $\bfQ^{(p)}(\infty)$ is well defined {in view of} the right-hand side of \eqref{eq:lim}.
This finishes the proof of the proposition. \qed  
\begin{proof}[Proof of Lemma~\ref{prop:constant_term_rho}] 
If  $\|\bfX_{[0,h]}/x\|_p>1$ then  for sufficiently small $\epsilon> 0$,
$\|\bfX_{[0,h]}/x\|_\infty>\epsilon$.  Therefore, on $\{\|\bfX_{[0,h]}/x\|_p>1\}$,
\beao
\sum_{i = 0}^{h} |\bfX_{i}/x|^\a\,\1\big(
|\bfX_i/x|>\epsilon\big) >0 \,.
\eeao
Using stationarity, we obtain
\beao
\lefteqn{\P(\|\bfX_{[0,h]}/x\|_p>1)} \\&=&
   \sum_{i = 0}^{h}
            \E\Big[ \dfrac{|\bfX_i/x|^\alpha\1\big(|\bfX_i/x| > \epsilon \big) }{\sum_{t = 0}^{h} |\bfX_{t}/x|^\a\,\1\big(
|\bfX_t/x|> \epsilon \big)}\1(\|\bfX_{[0,h]}/x\|_p>1)
\Big]\\
&=&   \sum_{i = 0}^{h} 
            \E\Big[ \frac{|\bfX_0/x|^\alpha\1\big(|\bfX_0/x| >  \epsilon \big) }{\sum_{t = -i}^{h-i} |\bfX_{t}/x|^\a\,\1\big(
|\bfX_t/x|> \epsilon \big)}  
            \1( \|\bfX_{[-i,h-i]}/x\|_p> 1 )\Big]\\ 
        & = &\P(|\bfX_0| > x \epsilon  )
            \sum_{i = 0}^{h}\E \Big[ \frac{|\bfX_0/x|^\alpha \1( \|\bfX_{[-i,h-i]}/x\|_p> 1 )}{ \sum_{t = -i}^{h-i} |\bfX_{t}/x|^\a\,\1\big(
|\bfX_t/x|> \epsilon \big)}  \; 
               \Big|\; |\bfX_0| > x\epsilon  \Big] \,.
\eeao
Applying the definition \eqref{eq:may27a} of \regvar\ and dominated \con ,
we obtain as $\xto$,
\beao
\dfrac{\P(\|\bfX_{[0,h]}/x\|_p>1)}{\P(|\bfX_0|>x)}
&\to& \epsilon^{-\a}
\sum_{i = 0}^{h}\E \Big[ \frac{|\epsilon Y\,\bfTh_0|^\alpha \1( \|\epsilon Y\,\bfTh_{[-i,h-i]}\|_p>  1)}{ \sum_{t = -i}^{h-i} |\epsilon Y\,\bfTh_{t}|^\a\,\1\big(
|Y\,\bfTh_t|>1\big)
} \Big]\\
&=& 
\sum_{i = 0}^{h}\int_\epsilon^\infty \E \Big[  \frac{\1( y\,\|\bfTh_{[-i,h-i]}\|_p> 1)  }{ \sum_{t = -i}^{h-i} |\bfTh_{t}|^\a\,\1\big(
y\,|\bfTh_t|>\epsilon  \big)
}  \Big] \, d(-y^{-\a})\,.
\eeao
The \lhs\ does not depend on $\epsilon$. Therefore,
letting $\epsilon\downarrow 0$, we arrive at
\beam\label{eq:may27aa}
\lim_{\xto} \dfrac{\P(\|\bfX_{[0,h]}/x\|_p>1)}{\P(|\bfX_0|>x)}\nonumber&=&
\sum_{i = 0}^{h}\int_0^\infty \E \Big[\dfrac{ \1( y\,\|\bfTh_{[-i,h-i]}\|_p> 1)}{ \|\bfTh_{[-i,h-i]}\|_\a^\a
}   
              \Big] \, d(-y^{-\a})\nonumber\\
&=&  \sum_{i = 0}^{h}\E \Big[\dfrac{\|\bfTh_{-i+[0,h]}\|_p^\a}{ \|\bfTh_{-i+[0,h]}\|_\a^\a\,
}\Big]
%=\sum_{i = -h}^{0}\E \Big[\dfrac{\|\bfTh_{i+[0,h]}/\|\bfTh\|_\a\|_p^\a}{\|\bfTh_{i+[0,h]}/\|\bfTh\|_\a\|_{\a}^{\a}} \Big]\\
%=
%{\blue \E\Big[\dfrac{\|\bfQ_{i+[0,h]}^{(\a)}\|_p^\a}{\|\bfQ^{(\a)}_{i+[0,h]}\|_\a^\a}\Big]}
=c(p,h)\,.
\eeam
This constant is finite since $\|\bfTh_{i+[0,h]}\|_p\le (h+1)\|\bfTh_{i+[0,h]}\|_\infty$.  
\par
Next we prove \eqref{eq:limitlaw}. For this reason,
let $A$ be a continuity set \wrt\ the limit law in \eqref{eq:limitlaw}.
An appeal to \eqref{eq:may27aa} yields
\beao
\lefteqn{   c(p,h)\, \mathbb{P}( x^{-1} \bfX_{[0,h]} \in A \;|\; 
\|\bfX_{[0,h]}\|_p > x)}\\
   &\sim& \frac{\mathbb{P}( x^{-1} \bfX_{[0,h]} \in A\,,  
\|\bfX_{[0,h]}\|_p > x)}{\P(|\bfX_0| > x)}=: I(x)\,. 
\eeao
Proceeding as for the derivation of \eqref{eq:may27aa}, we obtain
\beao
   I(x)&\sim & 
        \int_0^\infty  \sum_{i = 0}^{h}  
\E\Big[ \frac{\1(y\,\|\bfTh_{[-i,h-i]}\|_p > 1 ) }{\|\bfTh_{[-i,h-i]}\|_\alpha^\alpha} \,\1(y\,\bfTh_{[-i,h-i]} \in A )\Big]\,d(-y^{-\alpha})\\&
        =&\int_{1}^{\infty} 
\sum_{i = 0}^{h} \E\Big[\dfrac{\|\bfTh_{[-i,h-i]}\|_p^\a}
{\|\bfTh_{[-i,h-i]}\|_\alpha^\a}\, \1\Big(y\dfrac{\bfTh_{[-i,h-i]}}{\|\bfTh_{[-i,h-i]}\|_p} \in A \Big)\Big]d(-y^{-\alpha})\,.
\eeao
In the last step we changed the variable, $u = y\,\|\bfTh_{[-i,h-i]}\|_p>0$ a.s., observing that
$\|\bfTh_{[-i,h-i]}\|_p \ge |\bfTh_0|=1$. This proves~\eqref{eq:limitlaw} and the lemma. 
\end{proof}

\begin{comment}
\begin{remark}
The proposition above holds if we take as radius function any continuous function $\rho_h: \mathbb{R}^{d(h+1)} \to \mathbb{R}_{\ge 0}$ satisfying the homogeneity property $\rho_h(\lambda \bfx) = \lambda \rho_h(\bfx)$. For the lemma (\ref{lem:rvequi}) to hold too, the radius function $\rho_h$ must also satisfy, for all $\vep> 0$, $\inf\{ \rho_h(\bfx) : |\bfx| >\epsilon \} > 0$ for it to be a modulus function, see e.g. Segers {\it et al.} \cite{segers:zhao:meinguet:2017}. We provide the proof of this result for a general modulus function $\rho_h(\cdot)$ with re-normalizing constant $c(\rho_h)$.
\end{remark}
\end{comment}

\begin{proof}[Proof of Lemma~\ref{lem:cc}]
Assume $f:\tilde{\ell}^{\a} \cap\{\bfx : \|\bfx\|_p = 1\} \to (0,\infty)$ is any bounded Lipschitz-continuous function in 
$(\tilde{\ell}^{\a}, \tilde d_{\a})$. By Lemma~\ref{prop:constant_term_rho} we have for all $p \ge \a$,
\begin{align*}
    &\frac{c(p,h)}{h+1}\E[f(\bfQ^{(p)}(h))] - c(p) \E[f(\bfQ^{(p)})]=\\
    &= \frac{1}{h+1} \sum_{k=0}^{h} \E \Big[\Big( \frac{\|\bfTh_{-k+[0,h]}\|^\alpha_p}{\|\bfTh_{-k+[0,h]}\|^\alpha_\alpha} - \frac{\|\bfTh\|^\alpha_p}{\|\bfTh\|^\alpha_\alpha}\Big) f(\bfTh_{-k+[0,h]}/\|\bfTh_{-k+[0,h]}\|_p)\Big] +\\
        &\quad \quad \quad   +\E \Big[ \frac{\|\bfTh\|^\alpha_p}{\|\bfTh\|^\alpha_\alpha} \Big(f(\bfTh_{-k+[0,h]}/\|\bfTh_{[-k+[0,h]}\|_p) - f(\bfTh/\|\bfTh\|_p) \Big)\Big]  \\
        &=: I + II\,.    
\end{align*}
We will prove that $I$ and $II$ vanish as $h\to\infty$. 
Since $p \ge \alpha$ subadditivity yields for $k\in [0,h]$,
\begin{align*}
 &\Big| \frac{\|\bfTh_{-k+[0,h]}\|^\alpha_p}{\|\bfTh_{-k+[0,h]}\|^\alpha_\alpha} - \frac{\|\bfTh\|^\alpha_p}{\|\bfTh\|^\alpha_\alpha}\Big|   =  
\frac{\big|\big\|\|\bfTh\|_\alpha\bfTh_{-k+[0,h]}\big\|^\alpha_p-  \big\|\|\bfTh_{-k+[0,h]}\|_\alpha\bfTh\big\|^\alpha_p\big|}{\|\bfTh_{-k+[0,h]}\|^\alpha_\alpha\|\bfTh\|^\alpha_\alpha}\\
    &\quad \quad \le  \frac{\big|\big\|\|\bfTh\|_\alpha\bfTh_{-k+[0,h]}\big\|^p_p-
 \big\|\|\bfTh_{
-k+[0,h]}\|_\alpha\bfTh\big\|^p_p\big|^{\alpha/p}}{\|\bfTh_{-k+[0,h]}\|^\alpha_\alpha\|\bfTh\|^\alpha_\alpha}\,.
\end{align*}
Moreover,
\beao\lefteqn{
  \big|\big\|\|\bfTh\|_\alpha\bfTh_{-k+[0,h]}\big\|^p_p- \big\|\|\bfTh_{-k+[0,h]}\|_\alpha\bfTh\big\|^p_p\big|}  \\
    &\le &  \|\bfTh_{-k+[0,h]}\|_\alpha^p\Big(\sum_{t=-\infty}^{-k-1}|\bfTh_t|^{p} +\sum_{t=-k+h+1}^{+\infty}|\bfTh|^p \Big) \\
        &&+ \big|\|\bfTh\|_{\alpha}-\|\bfTh_{-k+[0,h]}\|_\alpha\big|^p\,\sum_{t=-k}^{-k+h}|\bfTh_{t}|^p \,.
\eeao
Thus, $|I|$ is bounded from above  by
\beao
    \lefteqn{\frac{1}{h+1} \|f\|_\infty \sum_{k=0}^{h}\Big(\E\Big[ \frac{\|\bfTh_{[-\infty, -k-1]}\|_p^\alpha}{\|\bfTh\|_\alpha^\alpha}\Big] + \E\Big[\frac{\|\bfTh_{[-k+h+1,\infty]}\|_p^\alpha}{\|\bfTh\|_\alpha^\alpha} \Big]}\\&&  
         +\E\Big[ \frac{\|\bfTh_{[-\infty,-k-1]}\|_\alpha^\alpha + \|\bfTh_{[-k+h+1,+\infty]}\|_\alpha^\alpha }{\|\bfTh\|_\alpha^\alpha} \frac{\|\bfTh_{-k+[0,h]}\|^\alpha_p}{\|\bfTh_{-k+[0,h]}\|_\alpha^\alpha}\Big] \Big)\\
            &\le& \frac{1}{h+1} \|f\|_\infty\sum_{k=0}^{h}\Big(
            \E\Big[ \frac{\|\bfTh_{[-\infty, -(k+1)]}\|_p^\alpha +\|\bfTh_{[-\infty, -(k+1)]}\|_\alpha^\alpha }{\|\bfTh\|_\alpha^\alpha}\Big] \\
                &&+ \E\Big[\frac{\|\bfTh_{[k+1,+\infty]}\|_p^\alpha +\|\bfTh_{[k+1,+\infty]}\|_\alpha^\alpha }{\|\bfTh\|_\alpha^\alpha} \Big]\Big)\,.
\eeao
\begin{comment}
Then, separating the sum above this term can be rewritten as 
\begin{align*}
    &\frac{1}{h+1} \|f\|_\infty\sum_{k=0}^{h}
            \E\Big[ \frac{\|\bfTh_{[-\infty, -(k+1)]}\|_p^\alpha +\|\bfTh_{[-\infty, -(k+1)]}\|_\alpha^\alpha }{\|\bfTh\|_\alpha^\alpha}\Big] \\
                &\quad \quad +\frac{1}{h+1} \|f\|_\infty\sum_{k=0}^{h} \E\Big[\frac{\|\bfTh_{[k+1,+\infty]}\|_p^\alpha +\|\bfTh_{[k+1,+\infty]}\|_\alpha^\alpha }{\|\bfTh\|_\alpha^\alpha} \Big]
\end{align*}
\end{comment}
Taking the limit as $h\to\infty$, the C\`esaro limit on the \rhs\ converges to zero. 
\par
We use the Lipschitz-continuity of $f$ to obtain an upper bound of $|II|$:
\begin{align*}
    | II| &\le \frac{1}{h+1}  c\,\sum_{k=0}^{h}\E\Big[ \frac{\|\bfTh_{-k+[0,h]}\|^\alpha_p}{\|\bfTh_{-k+[0,h]}\|^\alpha_\alpha} \tilde{d}_\alpha\Big( \frac{\bfTh_{-k+[0,h]}}{\|\bfTh_{-k+[0,h]}\|_p} , \frac{\bfTh}{\|\bfTh\|_p}\Big) \Big]\,. 
\end{align*}
Similar arguments as for $ |I|\to 0$ show that $|II|\to 0$.
\end{proof}
\subsection{Proofs of the results of Section~\ref{sec:Applications}}\label{sec:proof:4}
\subsubsection{ Proof of Theorem~\ref{cor:consistency} }\label{prooof:propcons}
We start with a version of Theorem~\ref{cor:consistency} for deterministic thresholds $(x_b)$.
\begin{lemma}\label{prop:consistency}
Assume the conditions of Theorem~\ref{cor:consistency}.
Then for every $g\in \mathcal{G}_{+}(\tilde{\ell^p})$,
      \begin{equation}\label{eq:consistent_estimator}
         \frac{1}{k} \sum_{t=1}^{ m}  g(x_{b}^{-1} \bfB_t) \xrightarrow[]{\P}  \int_{0}^{\infty} \E\big[g(y \,\bfQ^{(p)})\big] \,d(-y^{-\alpha})\,,\qquad       n \to  \infty\,,
      \end{equation}
      holds for sequences $k_n\to \infty$ and
$m_n := [n/b_n]  \to  \infty$ as in $\conditionMX{p}$.
\end{lemma}
\begin{proof}
%It suffices to prove that under $\conditionMX{p}$ for Lipschitz-continuous $f \in \mathcal{G}_{+}(\tilde{\ell}^p)$ as $\nto$,
If $\conditionMX{p}$ holds for Lipschitz-continuous $f \in \mathcal{G}_{+}(\tilde{\ell}^p)$, then it holds for functions $g \in \mathcal{G}_{+}(\tilde{\ell}^p)$ of the form $g(\bfx_t) = \1( \bfx_t \in A)$ where $A$ is a continuity-set of $\tilde\ell^p$ and $\bfzero \not \in \overline A$. It suffices to prove that
\beam\label{eq:mixproof}
\big(\E\big[ \ex^{ -  \frac{1}{k} \sum_{t=1}^{\lfloor m/k \rfloor} g(x_{b}^{-1}\bfB_t)}\big]\big)^{ k} &\to &
\ex^{ - \E\big[ \int_{0}^{\infty} g(y\,\bfQ^{(p)}) d(-y^{-\alpha})\big] }\,.
\eeam
By stationarity,
\beam\label{eq:june7b}
\E\big[ 1- \ex^{ - \frac{1}{k} \sum_{t=1}^{\lfloor m/k \rfloor} g(x_{b}^{-1}\bfB_t)}\big] = 
O\big(k^{-2}\,m\,\E[g(x_{b}^{-1}\bfB_1)]\big)\,.
\eeam
Since $g$ vanishes in some neighborhood of the origin there exists $c_g > 0$ such that $g(\bfx) = g(\bfx)\,\1(\|\bfx\|_p > c_g)$.
Therefore and by virtue of Proposition~\ref{Prop:large_deviations}
the \rhs\ of  \eqref{eq:june7b} vanishes as $\nto$. 
Now a Taylor expansion argument shows that the left-hand side of \eqref{eq:mixproof} is of the \asy\ order
%\beao
%    \E\big[ \exp\big\{ - \frac{1}{k} \sum_{t=1}^{\lfloor m/k \rfloor} f(x_{b}^{-1}\mathcal B_t)\big\}\Big]^k\\
$\sim\exp\{- (m/k) \E[g(x_{b}^{-1}\bfB_t)]\}$, and another application of 
 Proposition~\ref{Prop:large_deviations} yields \eqref{eq:mixproof}. We conclude by the 
portmanteau theorem for $M_0(\tilde{\ell}^p)$--convergence in Hult and Lindskog \cite{hult:lindskog:2006}, Theorem 2.4. that \eqref{eq:consistent_estimator} holds. 
\end{proof}

We continue with the proof of Theorem~\ref{cor:consistency}.
Lemma~\ref{prop:consistency} implies convergence of the 
empirical measures in   $M_0(\tilde{\ell}^p)$:
\beao
{P}_n(\cdot) := \frac{1}{k}\sum_{t=1}^m \1(x_{b}^{-1}\bfB_t \in \, \cdot \, )  \stp P(\cdot):= \int_{0}^\infty \P(y\bfQ^{(p)} \in \cdot) d(-y^{-\alpha})\,.
\eeao
Using the argument in Resnick \cite{resnick:2007}, p.~81, 
we may conclude $\|\bfB\|_{p,( k +1 )}/x_{b} \xrightarrow[]{\P}  1$, and thus the joint convergence in 
$(P_n,\,  \|\bfB\|_{p,(k+1  )}/x_{b}) \stp (P,1)$  in 
$M_0(\tilde{\ell^p}) \times \mathbb{R}_+$
follows. Now
\eqref{eq:consistent_estimator2} follows by an application of the 
continuous mapping theorem to the scaling function $ s(P(\cdot),t)=P(t\, \cdot)$. To prove continuity of $s$ we use again the 
portmanteau theorem for $M_0(\tilde{\ell}^p)$--convergence in Hult and Lindskog \cite{hult:lindskog:2006}, Theorem 2.4. Thus it suffices to check whether the 
limit $P_n\,f(\cdot/t)  \stp P\,f$ holds as $(n,t) \to (\infty,1)$ 
for Lipschitz-continuous $f\in \mathcal{G}_+(\tilde{\ell^p})$. But we have
with Lemma~\ref{prop:consistency}
\beao
 |P_n\,f(\cdot/t)  - P\,f| &\le& |P_n\,f(\cdot/t)  - P_n\,f| + |P_n\,f - P\,f| \\
 & =& |P_n\,f(\cdot/t)  - P_n\,f| + o_\P(1), \quad n \to  \infty.
 \eeao
Then, for all $ 0 < t_0 \le t < 2 $, for $t_0 \le 1$, { setting} $g(\bfx)=(\| \bfx \|_p \land {\|f\|_\infty}) \,  \1(\{\bfx : \| \bfx \|_p > c_{f}/t_0 \})$, we have
 \beao
|P_n\,f(\cdot/t)  - P\,f|   
    & \le& \big|t^{-1} - 1 \big|  \, P_n\, g + o_\P(1) \\
    & \le& \big|t^{-1} - 1 \big|  \, \big( c + o_\P(1) \big) + o_\P(1)\, ,
\eeao
for some $c > 0$, $c_f > 0$ as above. Letting $t \to 1$, continuity of $s$ follows.
\qed
\subsubsection{Proof of Proposition \ref{prop:stable:limit}}\label{proof:stable}
The result follows by a direct application of Theorem 3.1 in 
Bartkiewicz {\it et al.} \cite{bartkiewicz:jakubowski:mikosch:wintenberger:2011} on 
$\bfu^\top\bfS_n$ for every $\bfu\in\R^d$ such that $|\bfu|=1$ by checking their conditions {\bf (AC)}, {\bf (TB)}.  
Condition \eqref{eq:cond:small:sets:stable} implies that for all $\delta>0$, 
\beao
\lim_{l\to\infty}\limsup_{\nto} n\,\sum_{t=l}^{b_n} \P(|\bfX_t|>\delta\,a_n\,,|\bfX_0|>\delta\,a_n)\,,
\eeao
from which {\bf (AC)} is immediate. This condition also implies 
{\bf (TB)}. We show this in two steps. First, we identify 
the coefficients $b(v)$ in {\bf (TB)} in terms of the 
spectral tail process. Mikosch and Wintenberger \cite{mikosch:wintenberger:2013} showed that
\beao
b_\pm(v) - b_{\pm}(v-1) =\E\Big[\Big(\sum_{j=0}^v \bfu^\top\bfTh_j\Big)^\alpha_\pm\Big]- 
\E\Big[\Big(\sum_{j=1}^v \bfu^\top\bfTh_j\Big)^\alpha_\pm\Big]\,,
\eeao
where we suppress in the notation the dependence of the \lhs\ on $\bfu$ in what follows. 
{\bf (TB)} amounts to verifying that 
$b_\pm(v)-b_\pm(v-1)$ converges as $v\to\infty$. 
For $\a \in (0,1)$ this follows by concavity 
since $\|\bfTh\|_\a<\infty$ a.s. For $1<\alpha<2$ this will 
follow by a convexity argument if
$\E[(\sum_{j\ge 0}|\bfTh_j|)^{\a-1}]<\infty$. By subadditivity and Jensen's inequality, it is enough to check 
\beam\label{eq:bound}
\sum_{j = 0}^\infty \big(\E[|\bfTh_j|^{\alpha-1}\1(|\bfTh_j| > 1)] +  \E[|\bfTh_j|\land 1]\big) < + \infty\,.
\eeam 
We start by showing 
\beam\label{eq:prsumstep1}
\sum_{j=0}^\infty\E\big[|\bfTh_{j}| \wedge  1 \big]<\infty\,.
\eeam
Condition \eqref{eq:cond:small:sets:stable} implies
\beao
\lim_{l\to\infty} \limsup_{\nto}n\,\sum_{j=l}^{b_n} \E\big[(|a_n^{-1}\bfX_{j} |\wedge  1)\,\1(|  \bfX_{0}|>
a_n)\big]=0\,,
\eeao
which yields the following 
Cauchy criterion: for every $\vep>0$ there exists $K$ sufficiently large \st\
for $l\ge K, h\ge 0$,
\beao
\limsup_{\nto}n\,\sum_{j=l}^{l+h}  \E\big[(|a_n^{-1}\bfX_{j} |\wedge  1)\,\1(|  \bfX_{0}|>
a_n)\big] =\sum_{j=l}^{l+h}
 \E\big[|Y\,\bfTh_{j}|\wedge  1 \big]
\le \vep\, ,
\eeao
where we used \regvar\ of $(\bfX_t)$ in the last step. Then, we conclude \eqref{eq:prsumstep1} holds. By stationarity we can show similarly
\beam\label{eq:sumtruncnegative}
 \sum_{j=0}^{\infty}
    \E[|\bfTh_{-j}|\land 1] < + \infty. 
\eeam 
Then, by the time-change formula in \eqref{eq:june5a} we deduce 
\beao
    \infty &>&\sum_{j=0}^{\infty} \E[|\bfTh_{-j}|\land 1] 
    = \sum_{j=0}^{\infty} \E[|\bfTh_j|^\alpha\,( |\bfTh_{j}|^{-1}\land 1)] \\
    &=& \sum_{j=0}^{\infty} \E[ |\bfTh_{j}|^{\alpha-1}\land |\bfTh_j|^\alpha]>
    \sum_{j=0}^{\infty} \E[ |\bfTh_{j}|^{\alpha-1}\1(|\bfTh_j| > 1)],
\eeao 
and \eqref{eq:bound} holds.
This finishes the proof of the fact that 
$\E[(\sum_{j \ge 0}|\bfTh_t|)^{\a -1}] < + \infty$, in particular $c(1)<\infty$. Applying the mean value theorem and dominated convergence we arrive at the relation
\beao
b_\pm(v)-b_\pm(v-1)& \to& \E\Big[\Big(\sum_{j=0}^\infty \bfu^\top\bfTh_j\Big)^\alpha_\pm- \Big(\sum_{j=1}^\infty \bfu^\top\bfTh_j\Big)^\alpha_\pm\Big]\,,
\qquad v\to \infty\,.
\eeao
Reasoning for the limit as for \eqref{eq:dec11b} and recalling that $c(1)<\infty$, we identify
\beao\lefteqn{
\E\Big[\Big(\sum_{j=0}^\infty \bfu^\top\bfTh_j\Big)^\alpha_\pm- \Big(\sum_{j=1}^\infty \bfu^\top\bfTh_j\Big)^\alpha_\pm\Big]}\\&=&\E\Big[\Big(\sum_{j=-\infty}^\infty \bfu^\top\bfTh_j\Big)^\alpha_\pm\big/\|\bfTh\|_\a^\alpha\Big]
=\E\Big[\Big(\sum_{j=-\infty}^\infty \bfu^\top\bfQ_j^{(\alpha)}\Big)^\alpha_\pm\Big]\\
&=&c(1)\E\Big[\Big(\sum_{j=-\infty}^\infty \bfu^\top\bfQ_j^{(1)}\Big)^\alpha_\pm\Big]\,.
\eeao

\subsection{Proofs of the results of Section \ref{sec:beyond}}\label{proof:rest}\label{sec:proof:sec7}
\subsubsection{Proof of Proposition~\ref{ex2}}\label{subsub:ex2}

Notice that $\psi_g$ is bounded and measurable. 
For $p = \alpha$ we have $\bfQ^{(\alpha)} \eqd \bfTh/\|\bfTh\|_\alpha$. Then the result follows from Proposition 3.6 in Janssen \cite{janssen:2019}. For $p>0$, assuming the spectral cluster process $\bfQ^{(p)}$
%{\green ?  $p$-cluster} process 
is well defined we have $\|\bfTh\|_p < \infty$ a.s. and $c(p) < \infty$. Then, we introduce the Radon-Nikodym derivative of $\mathcal{L}( \bfQ^{(p)})$ with respect to $\mathcal{L}(\bfTh/\|\bfTh\|_p)$ which by \eqref{eq:def:cluster:process1} is the function $h: \ell^p \cap \{\bfx : \|\bfx\|_p = 1\} \to \mathbb{R}_{\ge 0}$ defined by $h(\bfy/\|\bfy\|_p) := \|\bfy\|_\alpha/\|\bfy\|_p$. Finally, the result follows by another application of Proposition 3.6 in Janssen  \cite{janssen:2019}. \qed

\subsubsection{Proof of Theorem \ref{thm:random_sampling}}\label{subsub:thm:rs}
The proof is given for $p = \alpha$ only; the case $p\le \a$
extends in a natural way. 
Let $g:\ell^\alpha \to \mathbb{R}$ be a continuous bounded function. We start by proving that $\psi_g$ defined in \eqref{eq:psif} 
is a continuous bounded function on $\tilde{\ell^\alpha}$. Fix $\epsilon > 0$ and $[\bfz] \in \tilde{\ell}^\alpha \cap \{  [\bfy] : \|\bfy\|_\alpha = 1\}$. Then for all $[\bfx] \in \tilde{\ell}^\alpha\cap \{[\bfy] : \|\bfy\|_\alpha = 1\}$, $k \in \mathbb{Z}$ and $N \in \mathbb{N}$, we have
    \beao
 &&   |\psi_g(\bfx) - \psi_g(\bfz)| =
        \Big| \sum_{j \in \mathbb{Z}} |\bfx_{j}^*|^\alpha { g ((\bfx_{j+t}^*)_t)}  - \sum_{j \in \mathbb{Z}} |\bfz_{j}^*|^\alpha g ((\bfz_{j+t}^*)_{t}) \Big| \\
          &=&\Big| \sum_{j \in \mathbb{Z}} |\bfx_{j}^* -\bfz_{j-k}^*|^\alpha  g ((\bfx_{j+t}^*)_t)  
             - \sum_{j \in \mathbb{Z}} |\bfz_{j}^*|^\alpha \big( g ( (
           \bfz_{j+t}^*)_t) - g  ((\bfx_{j+t+k}^*)_t) \big) \Big|\\ 
         & \le&  {   \|g\|_\infty} d_\alpha^{ \a}(B^{-k}\bfz^* , \bfx^*) 
            + { 2\|g\|_\infty}\, d_\alpha^{ \a}( \bfz^*, \bfz_{[-N,N]}^*)\,    \\
            & &+ \sum_{|j| < N}{ |\bfz^*_j|^\alpha}{ \big|g((\bfz_{j+t}^*)_t) - g((\bfx_{j+t+k}^*)_t)\big|}\,.
         \eeao
         If $[\bfx]$ satisfies $\tilde{d}_\alpha^{ \a}(\bfz, \bfx) < \epsilon(3\|g\|_\infty)^{-1}$ then there exists $k_0 \in \mathbb{Z}$ such that 
     \[\tilde{d}_\alpha^{ \a}(\bfz, \bfx) < d_\alpha^{ \a}( B^{-k_0}\bfz^*,\bfx^*) < \epsilon (3 \|g\|_\infty)^{-1}.\] 
     Furthermore, choose $N_0 \ge 0$ such that  
$d_\alpha^{ \a}(\bfz^*,\bfz_{[-N_0,N_0]}^*) < \epsilon (2\times 3\|g\|_\infty)^{-1}$ 
and consider the finite set $C_{[\bfz]} \subset \ell^\alpha\cap\{ \bfy : \|\bfy\|_\alpha=1\} $, defined by $C_{[\bfz]}:= \{(\bfz^*_{j+t})_t \in  \ell^\alpha: |j| < N_0, |\bfz^*_j| > 0\}$. Notice that for every $\tilde{\bfz} \in C_{[\bfz]}$ 
there exists $\delta(\tilde{\bfz})$ such that if $d_\alpha^{\a}( \tilde{\bfz}, \bfx) < \delta(\tilde{\bfz})$ implies $|g(\tilde{\bfz}) - g(\bfx)| < \epsilon/3 $. Finally, define $\eta(\bfz) := \min\{ \delta( \tilde{\bfz}) : \tilde{\bfz} \in C\} \land \epsilon(3\|g\|_\infty)^{-1}$. {Then, noticing that $\sum_{|j| < N_0}|\bfz^*_j|^\alpha \le \|\bfz\|_\alpha^\alpha = 1$, we also obtain a bound for the last term}. Hence, for every $[\bfx] \in \tilde{\ell^\alpha}$ satisfying $d_\alpha^{ \a}(\bfz, \bfx) <  \eta(\bfz)$ we have~$|\psi_g(\bfx) - \psi_g(\bfz)| < \epsilon$. 
         \par
This finishes the proof of the continuity of the 
function $\psi_g$ on  $\tilde{\ell^\alpha} \cap \{ \bfy : \|\bfy\|_\alpha = 1\}$. We conclude with applications of Lemma~\ref{prop:consistency} and Proposition~\ref{ex2}. \qed
{
\subsubsection{Proof of Proposition \ref{prop:LGsup}}\label{subsub:LGsup}
Theorem 4.5 in Mikosch and Wintenberger \cite{mikosch:wintenberger:2016}
yields immediately
\beam\label{eq:june15ccc}
    \Big| \frac{\P(\sup_{1 \le t \le n}S_t > x_n)}{n\P(|X_1| > x_n)} - \E\Big[ \big(\sup_{t \ge 0} \sum_{i=0}^t \Theta_i\big)^{\alpha}_{+} - \big(\sup_{t \ge 1} \sum_{i=1}^t \Theta_i\big)^{\alpha}_{+}\Big] \Big| \to 0 ,\quad n \to  \infty\,,\nonumber\\
\eeam
and $n\,\P(|X_1| > x_n) \to 0$.  We multiply the function inside the limiting 
expected value by the constant 
$1=\|\Theta\|^\alpha_\alpha/\|\Theta\|^\alpha_\alpha$. Moreover, since
$c(1) < \infty$, then
$\E [ (\sum_{t=1}^\infty |\Theta_t|)^{\alpha-1} ] < \infty$; see Lemma 3.11 in Planini\'c and Soulier \cite{planinic:soulier:2018}. Then, by Fubini's theorem,
\beao 
\lefteqn{\E\Big[ \big(\sup_{t \ge 0} \sum_{i=0}^t \Theta_i\big)^{\alpha}_{+} - \big(\sup_{t \ge 1} \sum_{i=1}^t \Theta_i\big)^{\alpha}_{+}\Big]}\\
&=& \sum_{j \in \mathbb{Z}}\E\Big[ |\Theta_j|^\alpha \, \Big( \big(\sup_{t \ge 0} \sum_{i=0}^t \frac{\Theta_i}{\|\Theta\|_\alpha}\big)^{\alpha}_{+} - \big(\sup_{t \ge 1} \sum_{i=1}^t \frac{\Theta_i}{\|\Theta\|_\alpha}\big)^{\alpha}_{+} \Big)\Big]\,.
\eeao
At this point we apply the time-change formula for positive measurable functions of $\bfTh$ at every term of the sum in $j \in \mathbb{Z}$; see Corollary 2.8. 
 in Dombry {\it et al.} \cite{dombry:hashorva:soulier:2018}.  
By the same argument as in the proof of Proposition \ref{Prop:large_deviations} we obtain  the representation of the expectation in \eqref{eq:june15ccc} in terms of the univariate spectral cluster process $Q^{(p)} $.
\par 
Now we apply Theorem~\ref{thm:random_sampling} to 
$f(\bfx):= \lim_{k \to \infty}( \sup_{t \ge- k} \sum_{i=-k}^{t} x_i)^\alpha_{+}$ on $\ell^1$. It is uniformly continuous and bounded by one on the sphere of $\ell^p$, hence 
\eqref{eq:estimate:f} holds for $f$. Similarly, the constant $c(1)^{-1}$ can be estimated by employing the \fct\ $g(\bfx):= \|\bfx\|_\alpha$ on $\ell^1$ 
which is bounded by one on the unitary $\ell^1$-sphere for $\alpha \ge 1$.
\qed}
\subsubsection{Proof of Proposition~\ref{cor:infthe} }\label{subsub:cc}
The re-normalization function $\zeta$ is continuous on the unit 
sphere of $(\ell^\alpha,d_\a)$, except for sequences with $\bfx_0=0$. Then
 \beao
     \P( \rho(Y \,\bfTh) > 1)  
        & =& \E\big[ \rho(\bfTh_t)^\alpha \land 1 \big] 
             = \E\Big[ \rho ( \bfQ^{(\alpha)}_t/|\bfQ^{(\alpha)}_0|)^\alpha \land 1 \Big] \\
            &  =& \E\Big[ (\rho^\a \land 1) \circ \zeta ( \bfQ^{(\a)}) \, \Big]. 
 \eeao
The proof is finished by an application of  Theorem \ref{thm:random_sampling}.
\qed


\begin{thebibliography}{}

\end{thebibliography}


\begin{thebibliography}{99}
%BBBBBBBBBBBBBBBBBBBBBBBBBBBBBBBBBBBBBBBBBBBBBBBBBBBBBBBB
\bibitem{bartkiewicz:jakubowski:mikosch:wintenberger:2011}
{\sc Bartkiewicz, K., Jakubowski, A., Mikosch, T. and Wintenberger, O.}\ (2011)
Stable limits for sums of dependent infinite variance random variables.
{\em Probab. Th. Relat. Fields} {\bf 150}, 337--372.
\bibitem{basrak:krizmanic:segers:2012}
{\sc Basrak, B., Krizmani\'c, D. and Segers, J.}\ (2012)
A functional limit theorem for dependent sequences with infinite variance stable limits. 
{\em Ann. Probab. }{\bf 40}, 2008--2033.
\bibitem{basrak:planinic:soulier:2018}
{\sc Basrak, B., Planini{\'c}, H. and Soulier, P.}\ (2018)
An invariance principle for sums and record times of regularly varying stationary sequences. {\em Probab. Th. Rel. Fields} {\bf 172}, 869--914.
\bibitem{basrak:segers:2009}
{\sc Basrak, B. and Segers, J.}\ (2009)
Regularly varying multivariate time series.
{\em Stoch. Proc. Appl.}
{\bf 119}, 1055--1080.
\bibitem{bingham:goldie:teugels:1987}
{\sc Bingham, N., Goldie, C.M. and Teugels, J.}\ (1987)
{\em Regular Variation.} Cambridge University Press, Cambridge (UK).
\bibitem{bradley:2005}
{\sc Bradley, R.C.}\ (2005)
Basic properties of strong mixing conditions. A survey and some open questions.
{\em Probability Surveys}
{\bf 2}, 107-144.

\bibitem{buritica:mikosch:meyer:wintenberger}
{\sc Buritic\'a, G. Mikosch, T. Meyer, N. and Wintenberger, O.}\ (2021)
Some variations on the extremal index. 
{\em Zap. Nauchn. Semin. POMI. Volume 501, Probability and Statistics.} {\bf 30}, 5277. To be translated in J.Math.Sci. (Springer). 
%CCCCCCCCCCCCCCCCCCCCCCCCCCCCCCCCCCCCCCCCCCCCCCCCCCCCCCCCCCCC
\bibitem{cissokho:kulik:2021}
{\sc Cissokho, Y. and Kulik, R.}\ (2021) Estimation of cluster functionals for regularly varying time series: sliding blocks estimators. 
{\em Electronic Journal of Statistics} {\bf 15}, 2777--2831.

\bibitem{cline:hsing:1998}
{\sc Cline, D.B.H. and Hsing, T.}\ (1998)
{\em Large deviation probabilities for sums of random variables with heavy or subexponential tails},
Technical Report, Texas A\& M University.
%DDDDDDDDDDDDDDDDDDDDDDDDDDDDDDDDDDDDDDDDDDDDDDDDDDDDDDDD
\bibitem{davis:drees:segers:warchol:2018}
{\sc Davis, R.A., Drees, H., Segers, J. and Warchol, M.}\ (2018)
Inference on the tail process with applications to financial time series modeling.
{\em J. Econometrics} {\bf 205}, 508--525.
\bibitem{davis:hsing:1995}
{\sc Davis, R.A. and Hsing, T.}\  (1995)
Point process and partial sum convergence for weakly dependent random variables with infinite variance. {\em Ann. Probab.}
{\bf 23}, 879--917.
%\bibitem{davis:mikosch:2009}
%{\sc Davis, R.A., and Mikosch, T.}\ (2009) The extremogram: A correlogram for extreme events. {\em Bernoulli} {\bf 15}, 977--1009.
\bibitem{davis:resnick:1986}
{\sc Davis, R.A., and Resnick, S.}\ (1985) Limit theory for moving averages of random variables with regularly varying tail probabilities.
{\em Ann. Probab. } {\bf 3}, 179--195.
\bibitem{dombry:hashorva:soulier:2018}
{\sc Dombry, C., Hashorva, E. and Soulier, P.}\ (2018) Tail measure and spectral tail process of regularly varying time series. {\em Ann. Appl. Probab.} {\bf 28}, 3884--3921. 
\bibitem{drees:janssen:neblung:2021}
{\sc Drees, H., Janssen, A. and Neblung, S.}\ (2021) Cluster based inference for extremes of time series.
{\em arXiv preprint  arXiv:2103.08512.}
\bibitem{drees:neblung:2021}
{\sc Drees, H. and Neblung, S.}\ (2021) Asymptotics for sliding blocks estimators of rare events. {\em Bernoulli} {\bf 27}, 1239--1269.
\bibitem{drees:rootzen:2010}
{\sc Drees, H. and Rootz\'en, H.}\ (2010)
Limit theorems for empirical processes of cluster functionals. {\em Ann. Stat.} {\bf 38}, 2145--2186.
\bibitem{drees:segers:warchol:2015}
{\sc Drees, H., Segers, J.,and Warchol, M.}\ (2015) 
Statistics for tail processes of Markov chains. {\em Extremes} {\bf 18}, 369--402.
%EEEEEEEEEEEEEEEEEEEEEEEEEEEEEEEEEEEEEEEEEEEEEEEEEEEEEEEEE
\bibitem{embrechts:kluppelberg:mikosch:2013}
{\sc Embrechts, P., KlÃŒppelberg, C., and Mikosch, T.}\ (2013).
{\sc Modelling extremal events: for insurance and finance.} {\bf 33}, Springer Science \& Business Media.
%HHHHHHHHHHHHHHHHHHHHHHHHHHHHHHHHHHHHHHHHHHHHHHHHHHHHHHHHH
\bibitem{dehaan:mercadier:zhou:2016}
{\sc Haan, L. de, Mercadier, C. and Zhou, C.}\ (2016)
Adapting extreme value statistics to financial time series: dealing with bias and serial dependence. 
{\em Finance and Stochastics} {\bf 20}, 321--354.
\bibitem{hsing:1993}
{\sc Hsing, T.}\ (1993)
Extremal index estimation for a weakly dependent stationary sequence.
{\em Ann. Stat.} {\bf 21}, 2043--2071. 
\bibitem{hult:lindskog:2006}
{\sc Hult, H. and Lindskog, F.}\ (2006)
Regular variation for measures on metric spaces.
{\em Publ. de l'Institut Math.} (Belgrade)
{\bf 80} (94), 121--140.
%JJJJJJJJJJJJJJJJJJJJJJJJJJJJJJJJJJJJJJJJJJJJJJJJJJJJJJJJJ
\bibitem{janssen:2019}
{\sc Janssen, A.}\ (2019)
Spectral tail processes and max-stable approximations of 
multivariate regularly varying time series.
{\em Stoch.  Proc. Appl.} {\bf 129}, 1993--2009.
%KKKKKKKKKKKKKKKKKKKKKKKKKKKKKKKKKKKKKKKKKKKKKKKKKKKKKKKKK
\bibitem{kulik:soulier:2020}
{\sc Kulik, R. and Soulier, P.}\ (2020)
{\em Heavy-Tailed Time Series.} Springer, New York.
%LLLLLLLLLLLLLLLLLLLLLLLLLLLLLLLLLLLLLLLLLLLLLLLLLLLLLLLLL
\bibitem{leadbetter:1983}
{\sc Leadbetter, M.R.}\ (1983)
Extremes and local dependence in stationary sequences.
{\em Probab. Th. Relat. Fields} {\bf 65}, 291--306.
\bibitem{leadbetter:lindgren:rootzen:1983}
{\sc Leadbetter, M.R., Lindgren, G., and Rootz\'en, H.}\ (1983)
{\em Extremes and related properties of random sequences and processes.} Springer, Berlin. 
\bibitem{lindskog:resnick:roy:2014}
{\sc Lindskog, F., Resnick, S.I., and Roy, J.}\ (2014)
Regularly varying measures on metric spaces: hidden regular variation and hidden jumps. {\em Probab. Surveys} {\bf 11}, 270--314.
%MMMMMMMMMMMMMMMMMMMMMMMMMMMMMMMMMMMMMMMMMMMMMMMMMMMMMMMMM
\bibitem{mikosch:wintenberger:2013}
{\sc Mikosch, T. and Wintenberger, O.}\ (2013)
Precise large deviations for dependent regularly varying sequences.
{\em Probab. Th. Rel. Fields} {\bf 156},
851--887.
\bibitem{mikosch:wintenberger:2014}
{\sc Mikosch, T. and Wintenberger, O.}\ (2014)
The cluster index of regularly varying sequences with
applications to limit theory for functions of multivariate Markov chains.
{\em Probab. Th. Rel. Fields} {\bf 159}, 157--196.
\bibitem{mikosch:wintenberger:2016}
{\sc Mikosch, T. and Wintenberger, O.}\ (2016)
A large deviations approach to limit theory for heavy-tailed
time series. {\em Probab. Th. Rel. Fields} {\bf 166}, 233--269.
\bibitem{mikosch:rodionov:2021}
{\sc Mikosch, T. and Rodionov, I.}\ (2021)
Precise large deviations for dependent subexponential variables.
{\em Bernoulli.} {\bf 27}, 1319--1347.

%NNNNNNNNNNNNNNNNNNNNNNNNNNNNNNNNNNNNNNNNNNNNNNNNNNNNNNNNN
\bibitem{nagaev:1979}
{\sc Nagaev, S.V.}\ (1979) Large deviations of sums of independent random variables. {\em Ann.  Probab.} {\bf 7}, 745--789.
%PPPPPPPPPPPPPPPPPPPPPPPPPPPPPPPPPPPPPPPPPPPPPPPPPPPPPPPPP
\bibitem{planinic:soulier:2018}
{\sc Planini{\'c}, H. and Soulier, P.}\ (2018)
  The tail process revisited.
  {\em Extremes}
  {\bf 21}, 551--579.
%RRRRRRRRRRRRRRRRRRRRRRRRRRRRRRRRRRRRRRRRRRRRRRRRRRRRRRRRR
\bibitem{resnick:2007}
{\sc Resnick, S.I.}\ (2007) 
{\em Heavy-Tail Phenomena: Probabilistic and Statistical Modeling.}
Springer, New York.
%SSSSSSSSSSSSSSSSSSSSSSSSSSSSSSSSSSSSSSSSSSSSSSSSSSSSSSSSS
\bibitem{segers:zhao:meinguet:2017}
{\sc Segers, J., Zhao, Y. and Meinguet, T.}\ (2017)
Polar decomposition of regularly varying time series in star-shaped metric spaces. {\em Extremes}
{\bf 20}, 539--566.
 
\end{thebibliography}
\end{document}